\documentclass[11pt,twoside,reqno]{amsart}
\usepackage[utf8]{inputenc}
\usepackage[T1]{fontenc}
\usepackage{import}
\usepackage{newclude}
\usepackage{amsmath,amssymb,amsthm,mathtools}
\usepackage{anysize}
\usepackage{amscd}
\usepackage{stmaryrd}
\usepackage{wasysym}
\usepackage{mathrsfs}
\usepackage[scr=boondox]{mathalfa}
\usepackage{enumitem}
\usepackage{accents}
\usepackage{dsfont}
\usepackage{soul}
\usepackage{color}
\usepackage[all]{xy}
\usepackage{float}
\usepackage{bm}
\usepackage{makecell}
\usepackage{thmtools}      
\usepackage{setspace}
\usepackage{comment}
\usepackage{tikz}
\usepackage{tikz-cd}
\usepackage{mathdots}
\usepackage{graphicx}
\usepackage[myheadings]{fullpage}
\usepackage{url}
\usepackage[hypertexnames=false,colorlinks,linkcolor=black,citecolor=black]{hyperref}
\usepackage[capitalize,nameinlink]{cleveref}

\numberwithin{equation}{section}
\newcommand\mtop{.95in}
\newcommand\mbottom{.95in}
\newcommand\mleft{1in}
\newcommand\mright{1in}
\usepackage[top = \mtop, bottom = \mbottom, left = \mleft, right=\mright]{geometry}

\newtheorem{thm}{Theorem}[section]
\newtheorem{example}[thm]{Example}
\newtheorem{prop}[thm]{Proposition}

\newtheorem{lemma}[thm]{Lemma}

\theoremstyle{definition}
\newtheorem{defi}[thm]{Definition}

\newtheorem{rmk}[thm]{Remark}

\newcommand\reallywidehat[1]{%
\savestack{\tmpbox}{\stretchto{%
  \scaleto{%
    \scalerel*[\widthof{\ensuremath{#1}}]{\kern-.6pt\bigwedge\kern-.6pt}%
    {\rule[-\textheight/2]{1ex}{\textheight}}
  }{\textheight}%
}{0.5ex}}%
\stackon[1pt]{#1}{\tmpbox}%
}
\DeclareSymbolFont{bbold}{U}{bbold}{m}{n}
\DeclareSymbolFontAlphabet{\mathbbold}{bbold}

\makeatletter
\def\@tocline#1#2#3#4#5#6#7{\relax
  \ifnum #1>\c@tocdepth 
  \else
    \par \addpenalty\@secpenalty\addvspace{#2}%
    \begingroup \hyphenpenalty\@M
    \@ifempty{#4}{%
      \@tempdima\csname r@tocindent\number#1\endcsname\relax
    }{%
      \@tempdima#4\relax
    }%
    \parindent\z@ \leftskip#3\relax \advance\leftskip\@tempdima\relax
    \rightskip\@pnumwidth plus4em \parfillskip-\@pnumwidth
    #5\leavevmode\hskip-\@tempdima
      \ifcase #1
       \or\or \hskip 1em \or \hskip 2em \else \hskip 3em \fi%
      #6\nobreak\relax
    \hfill\hbox to\@pnumwidth{\@tocpagenum{#7}}\par
    \nobreak
    \endgroup
  \fi}
\makeatother

\newcommand{\R}{\mathbb{R}}
\newcommand{\Z}{\mathbb{Z}}
\newcommand{\Q}{\mathbb{Q}}
\newcommand{\N}{\mathbb{N}}

\newcommand{\C}{\mathbb{C}}
\newcommand{\F}{\mathbb{F}}
\newcommand{\D}{\mathbb{D}}

\renewcommand{\L}{\mathcal{L}}
\newcommand{\dist}{\text{dist}}

\newcommand{\T}{\mathcal{T}}

\renewcommand{\P}{\mathcal{P}}
\renewcommand{\D}{\mathcal{D}}

\DeclareMathOperator{\loc}{loc}

\DeclareMathOperator{\ABC}{ABC}
\DeclareMathOperator{\hor}{hor}
\DeclareMathOperator{\ver}{ver}
\DeclareMathOperator{\diam}{diam}

\DeclareMathOperator{\spt}{spt}
\DeclareMathOperator{\Dir}{dir}

\DeclareMathOperator{\Proj}{Projection}

\title{$p$-adic Sum-Product, Projections, and Furstenberg Sets}
\author{Jiahe Shen}

\date{\today}
\begin{document}

\thanks{The author thanks Ruixiang Zhang for his encouragement to complete this problem.
He is especially grateful to Pablo Shmerkin for reading an earlier draft and for providing valuable comments and suggestions which helped clarify several aspects of the project. He also thanks Tuomas Orponen and Yuhan Chu for helpful discussions. The author acknowledges support from Ivan Corwin's NSF grant DMS-2246576 and Simons Investigator grant 929852.}

\maketitle

\begin{abstract}
Let $p$ be a prime number. We prove the sharp Furstenberg set bound in the $p$-adic plane $\Q_p^2$: every $(s,t)$-Furstenberg set $E\subset\Q_p^2$ satisfies
$$
\dim_H E\ge \min\left\{s+t,\frac{3s+t}{2},s+1\right\}.
$$
This matches the sharp lower bound in the Euclidean plane. We also derive two related consequences: a $p$-adic projection theorem for the maps $\pi_\theta(x,y)=x+\theta y$, together with the corresponding exceptional set estimate giving a $p$-adic analogue of Oberlin's projection question; and a discretized fractal sum-product estimate over $\Q_p$, showing that sufficiently non-concentrated subsets of $\Z_p^\times$ cannot have both small sum set and small product set.

The proof follows the projection-theoretic and multiscale machinery developed in the Euclidean works of Orponen--Shmerkin \cite{orponen2023projections} and Ren--Wang \cite{ren2023furstenberg}. The main task is to rebuild this machinery in the non-archimedean setting, and along the way we develop several new $p$-adic inputs needed to overcome the ultrametric features of the problem.
\end{abstract}



\tableofcontents

\section{Introduction}\label{sec: Introduction}

\subsection{Main results}

This paper resolves the sharp Furstenberg set problem over the $p$-adic plane $\Q_p^2$. We show that the optimal dimension lower bound for $(s,t)$-Furstenberg sets is the same as in the Euclidean plane, completing the $p$-adic Furstenberg program initiated in the previous work of Ren-Shen \cite{ren2025high}.

We begin by recalling the Euclidean formulation. Fix $s\in(0,1]$ and $t\in(0,2]$. A $(s,t)$-Furstenberg set in the real plane is a set $E\subset \R^2$ such that there exists a family $\L$ of lines with $\dim_H\L\ge t$ and
$$
\dim_H(E\cap l)\ge s,\qquad l\in\L.
$$
The Furstenberg set conjecture asserts that every such set satisfies
\begin{equation}\label{eq: sharp bound in real plane}
\dim_HE\ge\left\{s+t,\frac{3s+t}{2},s+1\right\}.
\end{equation}
There has been a substantial body of work on Furstenberg sets in the real plane over the past decades; see, for instance, Wolff \cite{wolff1999recent}, Héra-Shmerkin-Yavicoli \cite{hera2021improved}, Fu-Ren \cite{fu2024incidence}, Orponen-Shmerkin \cite{orponen2023hausdorff,orponen2023projections}, and the references therein. A decisive step was the work of Orponen-Shmerkin, which developed a projection-theoretic and additive-combinatorial approach to the regular case. The full conjecture was then resolved by Ren-Wang \cite{ren2023furstenberg}, who combined the regular theory with the semi-well-spaced analysis through a multiscale decomposition. The lower bound in \eqref{eq: sharp bound in real plane} is sharp.

Motivated by the Euclidean case, it is natural to ask whether the same sharp lower bound persists over the $p$-adic plane. The semi-well-spaced case was previously established by Ren-Shen \cite{ren2025high}. In this paper we complete the remaining regular and multiscale parts of the argument, and obtain the full $p$-adic analogue of the sharp Euclidean Furstenberg theorem.

\begin{thm}\label{thm: sharp bound of Furstenberg}
Let $s\in(0,1],t\in(0,2]$. Let $E\subset\Q_p^2$ be a $(s,t)$-Furstenberg set. Then the following sharp bound holds:
$$\dim_HE\ge\min\left\{s+t,\frac{3s+t}{2},s+1\right\}.$$
\end{thm}

The sharpness of the bound is discussed in \Cref{app: sharpness examples}, where
we give direct $p$-adic analogues of the three standard examples corresponding
to the terms $s+t$, $(3s+t)/2$, and $s+1$. Therefore, in this paper we focus on proving that the inequality in
\Cref{thm: sharp bound of Furstenberg} indeed holds.

Based on \Cref{thm: sharp bound of Furstenberg}, we wish to deduce a range of corollaries in related problems at the interface of geometric measure theory and additive combinatorics. One especially relevant theme, both in the real and $p$-adic settings, is projection theory: Furstenberg-type estimates are closely tied to understanding how a set behaves under projection, and to quantifying the (typically small) family of directions in which the projection can be atypically small. Usually, one seeks exceptional set estimates, showing that the set of directions for which a projection becomes anomalously small is itself small (for instance, has small Hausdorff dimension). In the Euclidean setting, projection theory is typically formulated in terms of orthogonal projections onto lines, which are naturally parametrized by directions in the unit circle $S^1$. In the $p$-adic setting, however, there is
no distinguished notion of angle or orthogonality compatible with the ultrametric geometry. For this reason, it is natural to define projections by affine linear functionals, as given by the following.

\begin{defi}\label{defi: projection map}
Let $\theta\in\Q_p$. Then, the projection map $\pi_\theta$ is given by
\begin{align}
\begin{split}
\pi_\theta:\Q_p^2&\rightarrow\Q_p\\
(x,y)&\mapsto x+\theta y.
\end{split}
\end{align}
\end{defi}

The above definition gives a reparametrisation of the family of affine line projections, and it matches the convention used by Orponen-Shmerkin, see in particular \cite[(4.64)]{orponen2023projections}. We now state our projection theorem in a form analogous to a standard consequence of the Euclidean projection theory developed in Orponen-Shmerkin and Ren-Wang.

\begin{thm}\label{thm: projection theorem}
Let $E\subset\Q_p^2$ and $\Lambda\subset\Q_p$ be nonempty Borel sets. Then there exists $\theta\in\Lambda$ such that
$$
\dim_H\pi_\theta(E)\ge\min\left\{
\frac{\dim_HE+\dim_H\Lambda}{2},\dim_HE,1\right\}.
$$
\end{thm}

Furthermore, one may formulate the same projection phenomenon as an exceptional set estimate. This is the form closest to Oberlin's projection question \cite[(1.8)]{oberlin2012restricted}; thus the following theorem gives its natural $p$-adic analogue.

\begin{thm}\label{thm: exceptional set estimate}
Let $E\subset\Q_p^2$ be a Borel set. Then, for all $0\le u\le \min\{\dim_H(E),1\}$, we have
$$
\dim_H\{\theta\in\Q_p:\dim_H\pi_\theta(E)<u\}
\le\max\{2u-\dim_HE,0\}.
$$
\end{thm}

Indeed, \Cref{thm: projection theorem} implies \Cref{thm: exceptional set estimate}. Let
$$
\Lambda_u:=\{\theta\in\Q_p:\dim_H\pi_\theta(E)<u\}.
$$
Fix an arbitrary Borel set $\Lambda\subset\Lambda_u$. Since
$$
\sup_{\theta\in\Lambda}\dim_H\pi_\theta(E)\le u,
$$
\Cref{thm: projection theorem} gives
$$
u\ge\min\left\{
\frac{\dim_HE+\dim_H\Lambda}{2},
\dim_HE,1\right\}.
$$
Since $u\le\min\{\dim_HE,1\}$, it follows that
$$
\dim_H\Lambda\le 2u-\dim_HE
$$
whenever $2u-\dim_HE\ge 0$. If $2u-\dim_HE<0$, the same inequality forces every Borel subset of $\Lambda_u$ to be empty. Therefore, in all cases,
$$
\dim_H\Lambda\le \max\{2u-\dim_HE,0\}.
$$
Taking the supremum over all Borel sets $\Lambda\subset\Lambda_u$, we obtain
$$
\dim_H\Lambda_u\le \max\{2u-\dim_HE,0\},
$$
which proves \Cref{thm: exceptional set estimate}.

Although \Cref{thm: projection theorem} and \Cref{thm: exceptional set estimate} are stated for directions $\theta\in\Q_p$, the proof below will be carried out after localizing to $\theta\in\Z_p$ and $E\subset\Z_p^2$. This entails no loss of generality. By countable stability of Hausdorff dimension, we may decompose both $E$ and the parameter space $\Q_p$ into countably many $p$-adic balls. After translating and dilating the ambient variables, each bounded piece of $E$ may be placed inside $\Z_p^2$. Likewise, since
$$
\Q_p={0}\bigcup\left(\bigcup_{m\in\Z}p^m\Z_p^\times\right),
$$
each nonzero direction range can be reduced to a bounded direction chart by a linear change of variables of the form $(x,y)\mapsto (x,p^m y)$, up to harmless bi-Lipschitz transformations. Such transformations preserve Hausdorff dimension and affect covering numbers only by constants depending on the chart. Therefore it suffices to prove the localized form of the projection theorem, with $E\subset\Z_p^2$ and $\theta\in\Z_p$, and then sum over the countably many charts.

Another related area of research is the discrete sum-product estimate. Let $\F$ be a field, and $A,B\subset \F$ be two subsets. Then, the sum set $A+B$ is defined as 
$$A+B:=\{a+b:a\in A,b\in B\},$$
and the product set is defined as
$$A\cdot B:=\{a\cdot b:a\in A,b\in B\}.$$
For a finite set $A$, Erd\H{o}s and Szemerédi concerned the size of $A+A$ or $A\cdot A$ and asked if they can be substantially larger than the cardinality of $A$. See Solymosi \cite{solymosi2009bounding}, Mohammadi-Stevens \cite{mohammadi2023attaining}, Rudnev-Stevens \cite{rudnev2022update}, and the author \cite{shen2025szemer} for recent progress on the Erd\H{o}s-Szemerédi sum-product theorem. We also emphasize the recent breakthrough of Bloom-Sawin-Schildkraut-Zhelezov \cite{bloom2026sum}, which disproves the Erd\H{o}s-Szemerédi sum-product conjecture over the real numbers by constructing arbitrarily large finite sets $A\subset \mathbb R$ with
$$
\max\{|A+A|,|A\cdot A|\}\le |A|^{2-c}
$$
for some absolute constant $c>0$. Their construction also yields analogous counterexamples in the $p$-adic setting, showing that the expected nearly quadratic lower bound fails far beyond the real case.

For the infinite set case, a natural generalization is to adopt a fractal analogue of the sum-product phenomenon. For a bounded set $A\subset \Q_p^d$ and $\delta\in p^{-\N}$, we will (somewhat informally) write $|A|_\delta$ for the $p$-adic $\delta$-covering number of $A$ (i.e.\ the minimal number of $p$-adic balls of radius $\delta$ needed to cover $A$).

\begin{defi}\label{defi: dsC set}
($(\delta,s,C)$-set) For $\delta=p^{-n}\in p^{-\N},s\in[0,d]$ and $C>0$, a nonempty bounded set $A\subset\Q_p^d$ is called a \emph{$(\delta,s,C)$-set} if for all $r=p^{-m}\in[\delta,1]$ and $x\in\Q_p^d$,
$$|A\cap B(x,r)|_\delta\le Cr^s|A|_\delta.$$
\end{defi}

We will also use the same terminology for subsets of the finite ring $(\Z/p^n\Z)^d$, assuming $\delta=p^{-n}$. In that setting, we identify a subset $A\subset(\Z/p^n\Z)^d$ with its canonical lift to $\Z_p^d$, namely the union of $\delta$-cosets
$$
\widetilde A:=\bigcup_{x\in  A}\bigl(x+p^n\Z_p^d\bigr)\ \subset\ \Z_p^d,
$$
and we say that $A$ is a $(\delta,s,C)$-set in $(\Z/p^n\Z)^d$ if the lifted set $\widetilde A$ is a $(\delta,s,C)$-set in the sense of \Cref{defi: dsC set}. Equivalently, all statements about $|\cdot|_\delta$ may be interpreted combinatorially in $(\Z/p^n\Z)^d$ by counting $\delta$-cosets. 

We will use two different regularity notions in the paper. The first one is the familiar
Ahlfors--David regularity. Recall that a $(\delta,t,C)$-set
$\P\subset\P_\delta$ is called $(\delta,t,C)$-AD-regular if
\begin{equation}
\label{eq: AD regular definition section2}
C^{-1}r^t|\P|
\le
|\P\cap B(x,r)|_\delta
\le
Cr^t|\P|,
\qquad x\in\P,\quad \delta\le r\le 1.
\end{equation}
This condition contains both upper and lower mass bounds. In contrast, the
recursive argument in \Cref{sec: Furstenberg regular set} only uses
an upper regularity condition, which is the $p$-adic analogue of the
$(\delta,t,C)$-regularity defined in Orponen--Shmerkin
\cite[Definition 2.23]{orponen2023projections}. It should not be confused with
Ahlfors--David regularity.

\begin{defi}[Upper regular set]
\label{def: upper regular set}
Let $\delta\in p^{-\mathbb N}$, $t>0$, and $C\ge 1$. A non-empty set
$\P\subset\P_\delta$ is called $(\delta,t,C)$-regular, or upper regular, if
\begin{enumerate}
    \item $\P$ is a $(\delta,t,C)$-set;
    \item for all $p$-adic scales $\delta\le r\le R\le 1$ and every
    $R$-cube $\mathbf p\in\P_R$,
    $$
    |\P\cap \mathbf p|_r\le C\left(\frac  Rr\right)^t.
    $$
\end{enumerate}
Here $|\P\cap \mathbf p|_r:=|\D_r(\P\cap \mathbf p)|$.
\end{defi}

Now, as a corollary of \Cref{thm: sharp bound of Furstenberg}, we obtain an estimate for the discretized sum-product theorem as the following.

\begin{thm}\label{thm: discretized sum-product}
Let $0<s<\frac 23$ and $\eta>0$. Then, there exists $\epsilon=\epsilon(s,\eta)>0$ such that the following holds for small enough $\delta>0$. Let $A\subset\Z_p^\times$ be a $(\delta,s,\delta^{-\epsilon})$-set. Then
$$\max\{|A+A|_\delta,|A\cdot A|_\delta\}\ge\delta^{-(5/4)s+\eta}.$$
\end{thm}

We have already explained above how \Cref{thm: exceptional set estimate}
follows from \Cref{thm: projection theorem}. The above discretized sum-product
consequence is obtained by the standard Furstenberg set construction
associated to additive and multiplicative translates of $A$: if both
$A+A$ and $A\cdot A$ were too small, the resulting configuration would
contradict \Cref{thm: sharp bound of Furstenberg}. This is the same
deduction as in Ren-Wang \cite[Theorems 1.2 and 1.3]{ren2023furstenberg}, with affine
$p$-adic projections in place of the Euclidean projection parametrisation. Moreover, beyond these particular corollaries, Ren-Wang also discussed a number of further developments in the Euclidean setting and its broad interactions with several other areas; we refer the reader to their introduction for an overview. We look forward to seeing further
applications of the $p$-adic analogues derived in this paper.

\subsection{New $p$-adic inputs and strategy}

The Furstenberg set problem is a representative example among Kakeya-type questions. In recent years, several problems of this kind in the Euclidean setting have been solved by a common multiscale strategy. One first treats configurations whose structure is essentially the same at all relevant scales, using projection theory developed by Bourgain together with additive-combinatorial input. One then extends the result to arbitrary configurations by an induction-on-scales or multiscale decomposition. In the proof of the Kakeya conjecture by Wang-Zahl \cite{wang2026sticky,wang2025assouad,wang2025volume}, the uniform-structure case is called the \emph{sticky} case; in the Furstenberg set problem, the corresponding role is played by the Ahlfors-David regular case.

Since $\R$ and $\Q_p$ are both local fields, it is natural to expect that some of the Euclidean ideas should have meaningful $p$-adic counterparts. Our proof is guided by this expectation: at the level of overall strategy, it follows the multiscale architecture of Orponen-Shmerkin \cite{orponen2023projections} and Ren-Wang \cite{ren2023furstenberg}. However, turning this architecture into a proof over $\Q_p$ is not a formal change of notation. The main new points are the following.

\paragraph{\textbf{Projection input and finite-local-ring additive combinatorics.}}
The additive-combinatorial input of this paper is an ABC sum-product theorem
over the finite local rings $\Z/p^n\Z$. Its final form is stated in
\Cref{thm: ABC sum-product} and proved in \Cref{sec: sum-product problem}. As
in the Euclidean work of Orponen--Shmerkin, the proof is not purely additive
combinatorial: it begins with a projection-theoretic input. In the present
setting, this is a $p$-adic Bourgain projection theorem obtained from
O'Regan's recent streamlined version of Bourgain's argument. We use this
input in \Cref{sec: sum-product problem} to establish the thin-tube estimate
which replaces the radial-projection step in the Euclidean proof. The
remaining part of the argument consists of additive-energy estimates,
Balog--Szemerédi--Gowers, Plünnecke--Ruzsa, and the restricted-graph
reduction.

The coefficient set in \Cref{thm: ABC sum-product} is assumed to lie in
$(\Z/p^n\Z)^\times$. This is the natural normalized form of the theorem. In
the regular projection argument of \Cref{sec: projections}, one first fixes a
reference direction $\theta_0$ and then pigeonholes the directions according
to the $p$-adic size of $\theta-\theta_0$. If the selected differences have
common size $\rho\in p^{-\Z}$, then they can be written as
$$
\theta-\theta_0=\rho^{-1} c_\theta,
\qquad c_\theta\in\Z_p^\times.
$$
The factor $\rho$ is absorbed into the spatial rescaling, and the corresponding
terminal resolution changes from $\delta$ to $\delta/\rho$. Thus the
normalized coefficients entering the ABC theorem are units, even though the
original differences of directions may occur arbitrarily deep in the
$p$-adic direction tree.

This is the direct non-archimedean analogue of descending to the effective
separation scale in the real argument. In the Euclidean setting, one
normalizes a direction difference by its actual size so that the resulting
coefficient lies in a fixed interval away from zero, such as $[1,2]$. In the
$p$-adic setting, the direction space is tree-like rather than ordered, and
the effective separation scale is detected by the valuation of the direction
difference. After dividing by the corresponding $p$-adic scale, the
coefficient lies in $\Z_p^\times$.

\begin{thm}\label{thm: ABC sum-product}
Let $0<\beta\le\alpha< 1$, and $\gamma\in(\alpha-\beta,1]$. Then, there exists $$\chi=\chi(\alpha,\beta,\gamma),\delta_0=\delta_0(\alpha,\beta,\gamma)>0$$
such that the following holds. Let $\delta=p^{-n}$ with $\delta\in(0,\delta_0]$, and let $A,B\subset \Z/p^n\Z$ be subsets satisfying the following properties:
\begin{enumerate}
\item $|A|\le\delta^{-\alpha}$.
\item $B$ is a nonempty $(\delta,\beta,\delta^{-\chi})$-set.
\end{enumerate}
Furthermore, let $C\subset(\Z/p^n\Z)^\times$ be a nonempty $(\delta,\gamma,\delta^{-\chi})$-set. Then, there exists $c\in C$ such that
\begin{equation}
|{a+cb:(a,b)\in G}|_\delta\ge\delta^{-\chi}|A|,\quad\forall G\subset A\times B,\ |G|\ge\delta^{\chi}|A||B|.
\end{equation}
\end{thm}

After the ABC theorem is established using the O'Regan projection
input, it will in turn be used in \Cref{sec: projections} to prove the regular-set
projection theorem, the $p$-adic analogue of
\cite[Theorem 1.6]{orponen2023projections}.

\paragraph{\textbf{A ratio substitute for radial projections.}}
In the Euclidean argument of Orponen-Shmerkin \cite{orponen2023projections}, the proof of the ABC theorem uses radial projections as a fundamental input. More precisely, their \cite[Definition 3.4]{orponen2023projections} defines, for a centre $x\in\R^d$, the radial projection
$$
\pi_x(y):=\frac{y-x}{|y-x|}\in S^{d-1}.
$$
This construction relies on the Euclidean norm, the unit sphere, and the angular parametrisation of directions. It has no direct counterpart over $\Q_p^2$. Although one can divide a vector by one of its coordinates in an affine chart, there is no canonical $p$-adic unit circle or angular space obtained by normalising a vector by its length. In particular, the expression $(y-x)/||y-x||_p$ does not produce a canonical point in a distinguished direction space analogous to $S^1$; the result depends on choices of coordinates and on which coordinate realizes the maximum norm.

We therefore replace the Euclidean radial-projection step by a two-chart ratio construction. This mechanism is made precise later in \Cref{sec: discretised expansion estimate}, where we decompose the non-diagonal pair space into the two regions $\Omega_x$ and $\Omega_y$ and introduce the corresponding ratio maps; it is then used as the projection-theoretic input for the regular projection argument in \Cref{sec: projections}. Given two points $z_1=(x_1,y_1)$ and $z_2=(x_2,y_2)$ in $\Q_p^2$, if
$$
||x_1-x_2||_p\ge ||y_1-y_2||_p,
$$
then the direction from $z_2$ to $z_1$ is recorded in the $x$-chart by the ratio
$$
\rho(z_1,z_2):=\frac{y_1-y_2}{x_1-x_2}\in\Z_p.
$$
On the complementary region, one interchanges the two coordinates and uses the reciprocal $y$-chart
$$
\rho'(z_1,z_2):=\frac{x_1-x_2}{y_1-y_2}\in\Z_p.
$$
Equivalently, in the notation used later, these are the two regions $\Omega_x$ and $\Omega_y$. The first consists of pairs for which the $x$-coordinate difference controls the distance, and the second consists of pairs for which the $y$-coordinate difference controls the distance. These two affine charts cover all non-diagonal pairs. Thus, instead of projecting to a global unit circle, we work chart by chart with $p$-adic slope ratios.

This replacement is not merely a change of notation. The fibres of the ratio maps are exactly the objects that play the role of thin tubes in the $p$-adic argument. For instance, fixing $z_2=(x_2,y_2)$ and requiring
$$
\frac{y_1-y_2}{x_1-x_2}\in J
$$
for a subset $J\subset\Z_p$ is equivalent to requiring $z_1$ to lie in a corresponding family of $p$-adic tubes through $z_2$, with slopes in $J$. Thus the thin-tube estimates obtained from radial projections in the Euclidean proof are replaced here by estimates for inverse images of balls under the two ratio maps on $\Omega_x$ and $\Omega_y$.

There is also a small but important bookkeeping issue absent from the Euclidean formulation. To use the ratio on $\Omega_x$, one must first remove a near-diagonal region so that the denominator $x_1-x_2$ has a controlled $p$-adic size. In the ultrametric setting this is handled using parent balls in the $p$-adic tree: after discarding pairs whose first coordinates lie in the same suitable parent ball, division by $x_1-x_2$ changes covering numbers only by a controlled factor. The complementary chart $\Omega_y$ is treated in the same way after swapping the two coordinates. This two-chart ratio mechanism is the $p$-adic substitute for the radial-projection input in Orponen-Shmerkin's proof.

The regular-set projection theorem proved in \Cref{sec: projections}, namely \Cref{thm: regular set projection}, is then used in \Cref{sec: Furstenberg regular set} to prove the discretized Furstenberg estimate for
upper-regular point sets in the sense of \Cref{def: upper regular set}. This gives the sharp Furstenberg lower bound in the regular case. To pass from the regular case to arbitrary Furstenberg configurations, we use the standard multiscale decomposition mechanism appearing in the Euclidean proofs of Orponen-Shmerkin and Ren-Wang: after passing to a uniform refinement, the relevant range of scales is partitioned into consecutive intervals, and on each interval the branching profile falls into one of two alternatives. In the semi-well-spaced alternative we apply the theorem of Ren-Shen \cite{ren2025high}, while in the regular alternative we apply the regular-case estimate proved in \Cref{sec: Furstenberg regular set}. Combining the contributions from these intervals yields the full $p$-adic Furstenberg bound given in \Cref{thm: sharp bound of Furstenberg}.

\subsection{Field dependence of Furstenberg and Kakeya phenomena}
It is worth mentioning that the inequality \eqref{eq: sharp bound in real plane} no longer holds if we replace $\R^2$ by $\C^2$. A counterexample is $E=\R^2$. In this case, $E$ is a $(\frac{1}{2},1)$-Furstenberg set, but it only has Hausdorff dimension $1$. The inherent reason is that they have different additive structures. Indeed, in the
complex setting, configurations essentially supported on the real subplane admit a large family of complex directions in which the projection becomes small, leading to the observed anomaly.

Likewise, for our non-archimedean setting, \Cref{thm: sharp bound of Furstenberg} does not hold if we replace the $p$-adic plane $\Q_p^2$ by $K^2$, where $K/\Q_p$ is an extension of degree $r>1$. To see this, one simply takes $E=\Q_p^2$, which is a $(\frac{1}{r},\frac{2}{r})$-Furstenberg set in $K^2$ of Hausdorff dimension $\frac{2}{r}$. Intrinsically, one can also see that the estimates in \Cref{thm: projection theorem} and \Cref{thm: exceptional set estimate} are no longer valid for such $K$.

Moreover, \Cref{thm: sharp bound of Furstenberg} also does not hold in the function field plane $\F_q((t))^2$ where $q$ is a power of a prime number $p$, since we can take $E=\F_q((t^r))^2$ for some integer $r>1$. In this case, $E$ is again a $(\frac{1}{r},\frac{2}{r})$-Furstenberg set in $\F_q((t))^2$ of Hausdorff dimension $\frac{2}{r}$. Therefore, among non-archimedean local fields, our theorem is specific to $\Q_p$ (with $p$ prime); heuristically, this reflects the fact that $\Q_p$ does not contain any proper closed subfield of finite index, whereas both nontrivial extensions of $\Q_p$ and $\F_q((t))$ contain large proper subfields that yield immediate counterexamples by restriction.

This field-dependence is specific to the Furstenberg-type estimates considered in this paper, and should be contrasted with the Kakeya problem. The finite field Kakeya conjecture was resolved by Dvir \cite{dvir2009size} using the polynomial method. Motivated by Dvir's argument, Arsovski \cite{arsovski2021p} developed a polynomial-method approach over discrete valuation rings and proved the Kakeya set conjecture over $\Q_p$; in fact, his result implies the conjecture over every finite extension of $\Q_p$. Salvatore \cite{salvatore2023kakeya} then adapted Arsovski's method to local fields of positive characteristic, proving the Kakeya conjecture over $\mathbb F_q((t))$. Thus Kakeya sets have full Hausdorff dimension over arbitrary non-archimedean local fields, in sharp contrast with the Furstenberg examples above.

Thus, unlike the Furstenberg problem, the Kakeya set conjecture remains true over arbitrary finite fields and non-archimedean local fields, including nontrivial finite extensions of $\Q_p$ and function fields such as $\mathbb F_q((t))$. The reason is that the Kakeya condition is much more rigid with respect to the ambient field: a Kakeya set must contain a line in every direction over the ambient field. A set supported on a proper closed subfield can contain many lines over the subfield, and hence can violate Furstenberg-type estimates, but it cannot contain lines in all directions of the larger field. This is why the subfield examples above obstruct the sharp Furstenberg bound but do not contradict the Kakeya conjecture.

It is also worth comparing this with the role of field-dependence in the recent proof of the Kakeya conjecture in $\R^3$ by Wang-Zahl \cite{wang2025volume}. In discussions of that proof, the Heisenberg group example is often described as a counterexample to a natural complex analogue in $\C^3$. More precisely, this example concerns a related tube configuration satisfying the analogue of the Wolff axioms, rather than a literal Kakeya set counterexample of the same kind as the subfield examples above. 

\subsection{Outline of the paper}

In \Cref{sec: Preliminaries}, we set up the basic discretized language used
throughout the paper. This includes $p$-adic $\delta$-cubes, dual
$\delta$-tubes, projection-adapted $\pi_\theta$-tubes, covering numbers, and the
non-concentration notions for discretized sets. We also record two multiscale
tools: a Lipschitz-function decomposition which will later be used to separate
regular and semi-well-spaced scale intervals, and a branching-number pigeonholing
lemma which allows us to pass to large uniform subsets.

\Cref{sec: sum-product problem} proves the $p$-adic ABC sum-product theorem,
\Cref{thm: ABC sum-product}. The logical direction here is important. The proof
does not use the regular projection theorem proved later in \Cref{sec: projections}.
Instead, it starts from a $p$-adic Bourgain projection input, in the streamlined
form due to O'Regan \cite{o2026bourgain}, and uses it to obtain the $p$-adic thin-tube estimate needed in the analogue of the Orponen-Shmerkin ABC argument. After proving the unit-coefficient weak version, where the coefficient set lies in $(\mathbb Z/p^n\mathbb Z)^\times$, we upgrade it to the restricted-graph statement needed later. In \Cref{sec: projections}, direction differences are first pigeonholed according to their $p$-adic size and then normalized by the selected separation scale. The resulting coefficients are units, so the unit-coefficient form of \Cref{thm: ABC sum-product} applies directly.

In \Cref{sec: projections}, we apply the ABC sum-product theorem from
\Cref{sec: sum-product problem} to prove a discretized projection theorem for
almost regular sets in $\mathbb Q_p^2$. Thus the projection theorem in this
section is an output of the ABC theorem, while the Bourgain/O'Regan projection
input used in \Cref{sec: sum-product problem} is an earlier external input. Since
the $p$-adic setting does not have a useful notion of orthogonal or radial
projection, the argument is formulated using the affine projections
$\pi_\theta(x,y)=x+\theta y$ and the corresponding $\pi_\theta$-tubes.

In \Cref{sec: Furstenberg regular set}, we use the projection theorem from
\Cref{sec: projections} inside a recursive construction to prove the sharp
Furstenberg lower bound in the regular case. This section is the $p$-adic
analogue of the regular-set part of the Orponen-Shmerkin strategy, with the
necessary modifications coming from the ultrametric geometry and from the use of
affine projections.

In \Cref{sec: General case}, we prove the full Furstenberg set estimate. Following
the multiscale strategy of Ren-Wang, we decompose the relevant range of scales
into intervals. On each interval, the configuration is either sufficiently
regular, in which case the result follows from \Cref{sec: Furstenberg regular set},
or semi-well-spaced, in which case we invoke the earlier semi-well-spaced theorem
of Ren-Shen. Combining these two alternatives yields
\Cref{thm: sharp bound of Furstenberg}.

Finally, \Cref{app: sharpness examples} records the sharpness constructions. We
give direct $p$-adic analogues of the standard examples corresponding to the
three terms in \Cref{thm: sharp bound of Furstenberg}. In particular, the middle term $\frac{3s+t}{2}$ is realized by a
$p$-adic Wolff-type grid construction, formulated at finite scale and then passed
to a limiting Moran construction.

\subsection{Disclosure on the use of AI tools}

During the preparation of this manuscript, the author used AI tools, including
ChatGPT, as an auxiliary writing and editing aid. These tools were used to help
improve exposition, reorganise preliminary drafts, check grammar, and suggest
alternative ways of presenting arguments. They were also occasionally used to
produce preliminary summaries of background material and to compare the structure
of this manuscript with related works in the literature.

The author is aware that AI-generated text can introduce imprecise or misleading
mathematical statements, especially when summarising technical arguments. For this
reason, all mathematical claims, definitions, proofs, and references appearing in
the paper were checked and revised by the author. The author takes
responsibility for the content of the manuscript and for any remaining errors. No
result in the paper relies on AI-generated output as a source of mathematical
justification.

\section{Preliminaries}\label{sec: Preliminaries}

We begin by fixing the basic metric and dimensional conventions used throughout the paper. Hausdorff dimension in $\Q_p^d$ is always taken with respect to the standard ultrametric
$$
||(x_1,\ldots,x_d)||_p:=\max_{1\le i\le d}||x_i||_p.
$$
All products of $p$-adic spaces below are equipped with the corresponding max metric.

We shall also use Hausdorff dimension for families of affine lines in $\Q_p^2$. This is defined in the same way as in the real setting, by using the natural finite collection of affine coordinate charts on the space of lines. More explicitly, every non-vertical affine line can be written uniquely in the form
$$
l_{a,b}:=\{(x,y)\in\Q_p^2:y=ax+b\},\qquad (a,b)\in\Q_p^2,
$$
while every non-horizontal affine line can be written uniquely in the form
$$
l'_{a,b}:=\{(x,y)\in\Q_p^2:x=ay+b\},\qquad (a,b)\in\Q_p^2.
$$
These two charts cover the space of affine lines. If $\mathcal L$ is a family of affine lines, we define $\dim_H\mathcal L$ to be the maximum of the Hausdorff dimensions of its parameter sets in these charts. Equivalently,
$$
\dim_H\mathcal L:=\max\left\{
\dim_H\{(a,b)\in\Q_p^2:l_{a,b}\in\mathcal L\},
\dim_H\{(a,b)\in\Q_p^2:l'_{a,b}\in\mathcal L\}
\right\}.
$$
This agrees with the usual convention in the Euclidean Furstenberg problem, where the space of affine lines is identified locally with slope-intercept coordinates. The use of two charts is only to avoid privileging one direction; replacing these charts by any other finite system of standard affine charts gives the same dimension, since the transition maps are locally bi-Lipschitz away from the usual chart boundaries.

With this convention, we say that a set $E\subset\Q_p^2$ is an $(s,t)$-Furstenberg set if there exists a family $\mathcal L$ of affine lines in $\Q_p^2$ with $\dim_H\mathcal L\ge t$ such that
$$
\dim_H(E\cap l)\ge s,\qquad l\in\mathcal L.
$$

In this paper, we will follow the same notation as in \cite{ren2025high}. We write $A\lesssim B$ if $A\leq CB$ for some constant $C>0$, and the notation $A\gtrsim B$ and $A\sim B$ is understood in a similar way. We write $A\lesssim_n B$ to emphasize that the implicit constant may depend on $n$. The notation $\gtrapprox_\delta$ means $\gtrsim_N \log(1/\delta)^{-C_N}$ for some number $N$ independent of $\delta$, and similarly for $\lessapprox_\delta$ and $\approx_\delta$.

We will also use the notation $0<a\ll b$ for parameters to mean that $a>0$ is chosen sufficiently small depending on $b$ and on the other fixed parameters in the ambient statement. Equivalently, after $b$ has been fixed, all subsequent arguments are valid provided $a\leq c(b)$ for a sufficiently small positive constant $c(b)$.

The notation $|A|$ denotes either the cardinality of $A$ when $A$ is a finite set, or the Haar measure when $A$ is measurable and infinite. We use the word fibre only in the elementary sense of a level set of one of the maps under consideration, such as a coordinate projection or an affine projection $\pi_\theta(x,y)=x+\theta y$; no fibration structure is intended.

\begin{defi}
($p$-adic $\delta$-cubes) Let $x=(a,b)\in\Z_p^2$, $\delta=p^{-n}$ where $n\in\N$. Then the \emph{$\delta$-cube} centered at $x$ is defined by
$$B(x,\delta):=(a+p^n\Z_p,b+p^n\Z_p)=\{(a_0,b_0): ||a-a_0||\le\delta,||b-b_0||\le\delta\}.$$
\end{defi}

In particular, we have $B(x,1)=B(0,1)=\Z_p^2$ for all $x\in\Z_p^2$. In contrast to the real setting, where $B(x,r)$ typically denotes a Euclidean ball, in the $p$-adic setting balls and cubes coincide due to the ultrametric structure. Hence we use the same notation $B(x,\delta)$. Moreover, any two $p$-adic cubes are either disjoint or one contains the other, reflecting the underlying tree-like structure of $\Z_p^2$, which often simplifies multiscale arguments.

For $n\in\Z_{\ge 0}$ and $\delta=p^{-n}$, denote by
$$\D_\delta(\Z_p^2):=\{B(x,\delta)\mid x\in\Z_p^2\}$$
the set of $\delta$-cubes in $\Z_p^2$. Via the natural quotient map $\Z_p\rightarrow \Z_p/p^k\Z_p$, we can also identify $\D_\delta(\Z_p^2)$ with $(\Z/p^n\Z)^2$. For distinct cubes $\mathscr{p}_1=B((a_1,b_1),\delta),\mathscr{p}_2=B((a_2,b_2),\delta)\in\D_\delta$, we define their distance by
$$\dist(\mathscr{p}_1,\mathscr{p}_2):=\max\{||a_1-a_2||,||b_1-b_2||\},$$
which takes values in $\{p^{-k}:0\le k\le n-1\}$.

For a set $P\subset\Z_p^2$, denote
$$\D_\delta(P):=\{\mathscr{p}\in\D_\delta(\Z_p^2): \mathscr{p}\cap P\ne\emptyset\}.$$
For brevity, we write $\P_\delta=\D_\delta(\Z_p^2)$. If $\delta<\Delta\in p^{-\N}$, and $\mathscr{p}=B(x,\delta)\in\P_\delta$, let $\mathscr{p}^\Delta=B(x,\Delta)$ denote the unique $\Delta$-cube $\mathbf{p}\in\P_\Delta$ that contains $\mathscr{p}$.

\begin{defi}\label{defi: covering number}
For any set $P\subset\Z_p^2$ and $\delta=p^{-n}\in p^{-\N}$, the $\delta$-covering number of $P$ is defined by
$$|P|_\delta:=|\D_\delta(P)|.$$
\end{defi}

If $\P\subset\P_\delta$ is a set of $\delta$-cubes, we say that $\P$ is a $(\delta,s,C)$-set if the union of $\delta$-cubes in $\P$ is a $(\delta,s,C)$-set in the above sense and we write $|\P|$ to denote the number of $\delta$-cubes in $\P$ and $|\P|_\Delta:=|\D_\Delta(\P)|$, so that $|\P|_\delta=|\P|$. 

\begin{defi}\label{defi: p-adic tubes}
($p$-adic $\delta$-tubes) Let $\delta=p^{-n}\in p^{-\N}$. A $\delta$-tube is a set of the form $T=\cup_{x\in\mathscr{p}}\mathbf{D}(x)$, where $\mathscr{p}\in\P_\delta$, and $\mathbf{D}$ is the point-line duality map
$$\mathbf{D}(a,b)=\{(x,y)\in\Z_p^2: y=ax+b\}\subset \Z_p^2$$
sending the point $(a,b)$ to the $p$-adic line with slope $a$ and intercept $b$. For brevity, we write $\mathbf{D}(\mathscr{p}):=\cup_{x\in\mathscr{p}}\mathbf{D}(x)$ as the $\delta$-tube that corresponds to $\mathscr{p}$. The collection of all $p$-adic $\delta$-tubes is denoted
$$\T_\delta:=\{\mathbf{D}(\mathscr{p}): \mathscr{p}\in\P_\delta\},$$
which may also be viewed as the set of lines in $(\Z/p^n\Z)^2$ with slopes in $\Z/p^n\Z$.
\end{defi}

A collection of $p$-adic $\delta$-tubes $\{\mathbf{D}(\mathscr{p})\}_{\mathscr{p}\in\P}$ is called a $(\delta,s,C)$-set if $\P$ is a $(\delta,s,C)$-set. For two distinct cubes $\mathscr{p}_1,\mathscr{p}_2\in\P_\delta$ with $$\dist(\mathscr{p}_1,\mathscr{p}_2)=p^{-k},\quad 0\le k\le n-1,$$ 
there are at most $p^k$ lines in $(\Z/p^n\Z)^2$ that pass through both of them.

If $\delta<\Delta\in p^{-\N}$ and $T\in\T_\delta$, denote by $T^\Delta$ the unique $p$-adic tube $\mathbf{T}\in\T_\delta$ containing $T$. For a set of $p$-adic $\delta$-tubes $\T$ and any scale $\Delta>\delta$, define the $p$-adic $\Delta$-covering number $|\T|_\Delta:=|\D_\Delta(\T)|$, where $\D_\Delta(\T):=\{\mathbf{T}\in\T_\Delta:\exists T\in\T,\mathbf{T}=T^\Delta\}$. We also denote $\T\cap\mathbf{T}:=\{T\in\T: T\subset\mathbf{T}\}$ for all $\mathbf{T}\in\D_\Delta(\T)$. When $\delta=\Delta$, we usually omit the subscript and simply write $|\T|=|\T|_\delta$.

We will use two related but distinct notions of tubes. The tubes in \Cref{defi: p-adic tubes} are dual tubes, parametrised by $\delta$-cubes in the line-parameter space $(a,b)$, and will be used for line families, incidences, and Furstenberg configurations. In contrast, the projection arguments below require tubes adapted to the fibres of
$\pi_\theta(x,y)=x+\theta y$. Thus, when studying $\pi_\theta(K)$, we will use the following $\pi_\theta$-tubes, which are inverse images of $\delta$-balls under $\pi_\theta$.

\begin{defi}[$\pi_\theta$-tubes]\label{defi: projection tubes}
Let $\delta=p^{-n}\in p^{-\mathbb N}$ and let $\theta\in \Z_p$. If $J\subset \Z_p$ is a $\delta$-ball, we define the associated
$\pi_\theta$-tube by
$$
T_{\theta,J}:=\pi_\theta^{-1}(J)\cap \Z_p^2
=\{(x,y)\in \Z_p^2:x+\theta y\in J\},
$$
where $\pi_\theta(x,y)=x+\theta y$. We denote by $\mathcal T_{\theta,\delta}$ the collection of all $\pi_\theta$-tubes of width
$\delta$.

In particular, when $\theta=0$, the tubes in $\mathcal T_{0,\delta}$ are the
vertical strips $J\times \Z_p$. These tubes will be used in the projection
arguments below. They should be distinguished from the dual tubes in
\Cref{defi: p-adic tubes}, which are parametrised by lines of the form
$y=ax+b$.
\end{defi}

\begin{defi}
(Slope set) The slope of a line $l=\mathbf{D}(a,b)$ is defined to be the number $a\in\Z_p$, and we write $\Dir(l):=a$. If $T=\mathbf{D}(p)$ is a $p$-adic $\delta$-tube with $\delta=p^{-n}$, we define the slope $\Dir(T)=\bigcup_{x\in\mathscr{p}}\Dir(\mathbf{D}(x))$, which we can regard as an element in $\Z/p^n\Z$. If $T(\mathscr{p})$ is a collection of $p$-adic $\delta$-tubes through a fixed point $\mathscr{p}$, we identify $\mathcal{T}(\mathscr{p})$ with $\Dir(\mathcal{T}(\mathscr{p}))$. 
\end{defi}

\subsection{Lemmas on Lipschitz functions}

In multiscale arguments in the Euclidean setting, Lipschitz functions often serve as a convenient way to encode scale-by-scale information and to control how geometric quantities evolve across scales. For instance, one may track the growth of covering numbers or branching patterns by a Lipschitz ``profile'', and then use elementary
properties of Lipschitz functions (together with pigeonholing) to locate scales on which the behaviour is essentially uniform, which is a basic mechanism underlying many inductive-on-scales proofs. Although we work in the $p$-adic setting, the same underlying multiscale mechanism remains in force: the geometry changes, but the bookkeeping of scale-to-scale growth and the uniformisation steps are conceptually the same. In this subsection, we introduce the Lipschitz functions that will be used throughout the paper as a systematic tool for this purpose.

\begin{defi}
(Interpolating slope) Given a function $f:[a,b]\rightarrow\R$, we define $s_f(a,b)$ to be the slope of the secant line through $a$ and $b$, i.e.,
$$s_f(a,b):=\frac{f(b)-f(a)}{b-a}.$$
\end{defi}

\begin{defi}
($\epsilon$-linear and superlinear functions). Given a function $f:[a,b]\rightarrow\R$ and $\epsilon>0,\sigma$, we say that $(f,a,b)$ is \emph{$(\sigma,\epsilon)$-superlinear} (or that $f$ is \emph{$(\sigma,\epsilon)$-superlinear} on the interval $[a,b]$) if
$$f(x)\ge f(a)+\sigma(x-a)-\epsilon(b-a),\quad x\in[a,b].$$
If $\sigma=s_f(a,b)$, then we simply say $f$ is $\epsilon$-linear.

We say that $(f,a,b)$ is $\epsilon$-linear or $f$ is $\epsilon$-linear on $[a,b]$ if both $(f,a,b)$ and $(-f,a,b)$ are $\epsilon$-superlinear. Equivalently,
$$|f(x)-f(a)-s_f(a,b)(x-a)|\le\epsilon(b-a),\quad x\in[a,b].$$
\end{defi}

\begin{lemma}\label{lem: interval partitioning}
Fix $0<s<t<u$ and $m>0$. For every $0<\epsilon<\min(u-s,\frac{1}{2})$, there is $\tau=\tau(\epsilon,s,u,t)>0$ such that the following holds: for every piecewise affine $2$-Lipschitz function $f:[0,m]\rightarrow\R$ such that 
$$f(x)\ge tx-\epsilon^2m\quad\forall x\in[0,m],\quad f(m)\le (t+\epsilon^2)m,$$
there exists a family of non-overlapping intervals $\{[c_j,d_j]\}_{j=1}^n$ contained in $[0,m]$ such that:
\begin{enumerate}
\item For each $j$, at least one of the following alternatives holds.
\begin{enumerate}
\item $(f,c_j,d_j)$ is $\epsilon$-linear with $s_f(c_j,d_j)\in[s,u]$; \label{item: epsilon linear invertal}
\item $(f,c_j,d_j)$ is $\epsilon$-superlinear with $s_f(c_j,d_j)\in[s,u]$ and
$$f(x)\ge\min\{f(c_j)+u(x-c_j),f(d_j)-s(d_j-x)\}-\epsilon(d_j-c_j).$$ \label{item: epsilon superlinear invertal}
\end{enumerate}
\item $d_j-c_j\ge\tau m$ for all $j$;
\item $[[0,m]\backslash[c_j,d_j]]\lesssim_{s,t,u}\epsilon m$.
\end{enumerate}
\end{lemma}

\begin{proof}
This is \cite[Proposition 6.4]{ren2023furstenberg}.
\end{proof}

\subsection{Branching numbers and uniform sets via pigeonholing}

\begin{defi}
Let $N\ge 1$, and
$$\delta=\Delta_N<\Delta_{N-1}<\cdots<\Delta_1<\Delta_0=1$$
be a sequence of $p$-adic scales. We say that a set $P\subset\Z_p^2$ is \emph{$\{\Delta_j\}_{j=1}^N$-uniform} if there is a sequence $\{N_j\}_{j=1}^N$ with $N_j\in p^{\N}$ and $|P\cap Q|_{\Delta_j}\in[N_j/p,N_j)$ for all $j\in\{1,2,\ldots,N\}$ and $Q\in\D_{\Delta_{j-1}}(\P)$.
\end{defi}

\begin{lemma}\label{lem: pigeonhole}
Let $P\in\Z_p^2$, $N,T\in\N$, and $\delta=p^{-NT}$. Also, let $\Delta_j:=p^{-jT}$ for all $0\le j\le N$. Then there exists a $\{\Delta_j\}_{j=1}^N$-uniform set $P'\subset P$ such that
$$|P'|_\delta\ge (pT)^{-N}|P|_\delta=p^{-\frac{\log_p T+1}{T}\cdot NT}|P|_\delta.$$
In particular, if $\epsilon>0$ and $T^{-1}(\log_p T+1)\le\epsilon$, then $|P'|_\delta\ge\delta^{\epsilon}|P|_\delta$.
\end{lemma}

\begin{defi}
(Branching function) Let $T\in \N$, and let $\mathcal{P}\subset\Z_p^d$ be a $\{\Delta_j\}_{j=1}^N$-uniform set with $\Delta_j:=p^{-jT}$, and let $\{N_j\}_{j=1}^n\subset\{1,\ldots,p^{dT}\}^n$ be the associated sequence. We define the branching function $f:[0,n]\rightarrow[0,dn]$ by setting $f(0)=0$, and
$$f(j):=\frac{\log|\mathcal{P}|_{p^{-jT}}}{T}=\frac{1}{T}\sum_{i=1}^j\log N_i,\quad i=\{1,2,\ldots,n\},$$
and then interpolating linearly.
\end{defi}

\begin{lemma}[Passing to coarse parents]\label{lem: passing to coarse parents}
Let $\delta<\Delta$ be $p$-adic scales, and let $P\subset \P_\delta$ be a
$(\delta,s,C)$-set. Assume that there exists an integer $N$ such that
$$
|P\cap Q|=N,\qquad Q\in \D_\Delta(P).
$$
Then $\D_\Delta(P)$ is a $(\Delta,s,O_p(C))$-set. More generally, if
$P'\subset P$ satisfies $|P'|\ge \lambda |P|$ and
$$
|P'\cap Q|=N',\qquad Q\in \D_\Delta(P'),
$$
then $\D_\Delta(P')$ is a $(\Delta,s,O_p(C\lambda^{-1}))$-set.
\end{lemma}

\begin{proof}
We prove the second statement, which contains the first. Let $R\in p^{-\N}$ with
$\Delta\le R\le 1$, and let $B$ be an $R$-ball in $\Z_p^2$. Since every
$\Delta$-cube in $\D_\Delta(P')$ contains exactly $N'$ elements of $P'$, we have
$$
|\D_\Delta(P')\cap B|\,N'
\le |P'\cap B|_\delta
\le |P\cap B|_\delta
\le C R^s |P|.
$$
On the other hand,
$$
|\D_\Delta(P')|\,N'=|P'|\ge \lambda |P|.
$$
Combining the two estimates gives
$$
|\D_\Delta(P')\cap B|
\le C\lambda^{-1}R^s|\D_\Delta(P')|.
$$
This is precisely the $(\Delta,s,O_p(C\lambda^{-1}))$-set condition, up to the
harmless constant coming from replacing balls by $\Delta$-cubes.
\end{proof}




\section{The ABC sum-product problem}\label{sec: sum-product problem}

In this section, we prove \Cref{thm: ABC sum-product}. The argument follows the
structure of \cite[Section 3]{orponen2023projections}. We first establish a weak
form in which the conclusion is required only for the full product
$G=A\times B$. We then pass to the restricted-graph formulation by the
additive-combinatorial reduction from
\cite[Section 5.1]{orponen2024discretised}. Since this reduction only uses
energy estimates, Balog--Szemerédi--Gowers, Plünnecke--Ruzsa, and
pigeonholing, it applies without change to the additive group
$\Z/p^n\Z$.

\begin{thm}[Weak ABC sum-product]\label{thm: weak ABC sum-product}
Let $0<\beta\le \alpha<1$, and let $\gamma\in(\alpha-\beta,1]$. Then there exist
$$
\chi=\chi(\alpha,\beta,\gamma)>0,
\qquad
\delta_0=\delta_0(\alpha,\beta,\gamma)\in(0,1/p]
$$
such that the following holds. Let $\delta=p^{-n}\in(0,\delta_0]$, and let
$A,B\subset \Z/p^n\Z$ satisfy
\begin{enumerate}
    \item $|A|\le \delta^{-\alpha}$;
    \item $B$ is a nonempty $(\delta,\beta,\delta^{-\chi})$-set.
\end{enumerate}
Let $C\subset(\Z/p^n\Z)^\times$ be a nonempty
$(\delta,\gamma,\delta^{-\chi})$-set. Then there exists $c\in C$ such that
$$
|A+cB|_\delta\ge \delta^{-\chi}|A|.
$$
Moreover, the constants $\chi,\delta_0$ stay bounded away from $0$ as $(\alpha,\beta,\gamma)$ ranges in a compact subset of
$$
\{(\alpha,\beta,\gamma):0<\beta\le \alpha< 1,\ \gamma\in(\alpha-\beta,1]\}.
$$
\end{thm}

The coefficient set is assumed to lie in
$(\Z/p^n\Z)^\times$. This is the natural normalized form of the theorem:
after the effective direction scale is selected in Section~4, the relevant
direction differences are divided by that scale and become units. Thus no
bounded-valuation parameter is needed in either the statement or the proof.

\subsection{Additive-combinatorial tools}\label{subsec: additive combinatorial tools}

We first recall the Plünnecke-Ruzsa inequality in the form needed below; see, for
example, Tao's exposition \cite[Theorem 1]{tao2009entropyPlunneckeRuzsa} on Plünnecke-Ruzsa inequalities. In the argument of
Orponen-Shmerkin this tool is applied in the real setting, but the Plünnecke-Ruzsa inequality is purely additive-combinatorial and holds in any additive group. Therefore it applies verbatim to the finite additive group
$\Z/p^n\Z$, giving the following $p$-adic version.

\begin{lemma}[Plünnecke-Ruzsa]\label{lem: plunnecke ruzsa}
Let $\delta=p^{-n}$, and let $A,B_1,\ldots,B_k\subset \Z/p^n\Z$. Assume that
$$
|A+B_i|_\delta\le K_i|A|_\delta,\qquad 1\le i\le k.
$$
Then there exists $A'\subset A$ with $|A'|_\delta\ge \frac12|A|_\delta$, such that
$$
|A'+B_1+\cdots+B_k|_\delta
\lesssim_k K_1\cdots K_k |A|_\delta.
$$
\end{lemma}

We next record the standard relation between small restricted sumsets and large
additive energy.

\begin{defi}[Discretised additive energy]\label{def: discretised additive energy}
Let $\delta=p^{-n}$ and let $A,B\subset \Z/p^n\Z$. Define
$$
\mathcal{E}_\delta(A,B):=|\{(a_1,a_2,b_1,b_2)\in A^2\times B^2:
a_1+b_1=a_2+b_2\}|.
$$
\end{defi}

\begin{lemma}\label{lem: energy restricted sumset}
Let $A,B\subset \Z/p^n\Z$ and $K\ge 1$.
\begin{enumerate}
    \item If there exists $G\subset A\times B$ such that $|G|\ge \frac{|A||B|}{K}$
    and
    $$
    |\{a+b:(a,b)\in G\}|_\delta\le K|A|,
    $$
    then $\mathcal{E}_\delta(A,B)\ge K^{-3}|A||B|^2$.
    \item Conversely, if $\mathcal{E}_\delta(A,B)\ge K^{-1}|A||B|^2$, then there exists $G\subset A\times B$ such that 
    $$
    |G|\ge (4K)^{-1}|A||B|,\quad|\{a+b:(a,b)\in G\}|_\delta\le 2K|A|.
    $$
\end{enumerate}
\end{lemma}

\begin{proof}
The first part follows from Cauchy-Schwarz:
$$
\mathcal{E}_\delta(A,B)\ge\sum_{h}
|\{(a,b)\in G:a+b=h\}|^2\ge\frac{|G|^2}{|\{a+b:(a,b)\in G\}|_\delta}\ge K^{-3}|A||B|^2.
$$
For the converse, define
$$
G:=\left\{(a_1,b_1)\in A\times B:
|\{(a_2,b_2)\in A\times B:a_1+b_1=a_2+b_2\}|
\ge \frac{|B|}{2K}\right\}.
$$
If $|G|<(4K)^{-1}|A||B|$, then the contribution to the energy from popular and
non-popular pairs is strictly smaller than $K^{-1}|A||B|^2$, a contradiction. Moreover,
$$
|\{a+b:(a,b)\in G\}|_\delta\le\frac{2K}{|B|}\sum_h |\{(a,b)\in A\times B:a+b=h\}|
\le2K|A|.
$$
\end{proof}

Finally, we recall the graph Balog-Szemerédi-Gowers theorem in the multiplier form needed below. This is just the usual graph Balog-Szemerédi-Gowers theorem applied with multiplicities; see, for example, Tao-Vu \cite[Theorem 2.29]{tao2006additive}. The point of the formulation below is that the multiplier is allowed to be a non-unit.

\begin{thm}[Asymmetric Balog-Szemerédi-Gowers, multiplier form]
\label{thm: asymmetric BSG}
For every $\eta>0$ there exist $\zeta=\zeta(\eta)>0$ and
$\delta_0=\delta_0(\eta)\in(0,1/p]$ such that the following holds for all
$\delta=p^{-n}\in(0,\delta_0]$. Let $A,C\subset \Z/p^n\Z$, and let
$d\in \Z/p^n\Z$ be arbitrary. Suppose that there is a set
$G\subset A\times C$ satisfying $|G|\ge \delta^{\zeta}|A||C|$ and
\begin{equation}\label{eq: small restricted sumset for aBSG}
|\{a+dc:(a,c)\in G\}|_{\delta}\le \delta^{-\zeta}|A|.
\end{equation}
Then there exist subsets $A'\subset A$ and $C'\subset C$ such that
$$
|A'||C'|\ge \delta^{\eta}|A||C|,\qquad
|A'+dC'|_{\delta}\le \delta^{-\eta}|A|.
$$
\end{thm}

The multiplier $d$ is not assumed to be a unit. This is important later, since
$d$ will be a difference such as $b_1-b_2$, which may be divisible by $p$. The
result is purely additive and does not require the map $c\mapsto dc$ to be
injective. The unit condition appears only in the ABC sum-product theorem, where
one needs multiplication by the coefficient $c$ to preserve $\delta$-scale
cardinalities.

\begin{proof}
Let
$$
\nu(x):=|\{c\in C:dc=x\}|,\qquad x\in \Z/p^n\Z,
$$
be the multiplicity function of the image of $C$ under the homomorphism
$c\mapsto dc$. Thus $\|\nu\|_1=|C|$. The graph $G\subset A\times C$ induces a
weighted graph between $A$ and $dC$ by
$$
\omega(a,x):=|\{c\in C:(a,c)\in G,\ dc=x\}|.
$$
Then $0\le \omega(a,x)\le \nu(x)$ and
$$
\sum_{a\in A}\sum_x \omega(a,x)=|G|\ge \delta^\zeta |A||C|.
$$
Moreover,
$$
\{a+x:\omega(a,x)>0\}
\subset
\{a+dc:(a,c)\in G\},
$$
and therefore \eqref{eq: small restricted sumset for aBSG} gives
$$
|\{a+x:\omega(a,x)>0\}|_\delta\le \delta^{-\zeta}|A|.
$$

We now apply the usual graph Balog-Szemerédi-Gowers theorem with multiplicities, as discussed in \cite[Section 2.5]{tao2006additive}. Equivalently, one may use the integer-weighted version obtained by replacing each $c\in C$ by a labelled copy of its image $dc$. Since the weighted graph has total size at least $\delta^\zeta |A||C|$ and its restricted sumset has $\delta$-covering number at most $\delta^{-\zeta}|A|$, the graph Balog-Szemerédi-Gowers theorem gives a subset $A'\subset A$ and a labelled subcollection
$C'\subset C$ such that
$$
|A'||C'|\ge \delta^{O(\zeta)}|A||C|,\quad
|A'+dC'|_\delta\le \delta^{-O(\zeta)}|A|.
$$
Choosing $\zeta>0$ sufficiently small depending on $\eta$ gives
$$
|A'||C'|\ge \delta^\eta |A||C|,
\qquad
|A'+dC'|_\delta\le \delta^{-\eta}|A|,
$$
as required.
\end{proof}

\subsection{From the weak version to the restricted-graph version}
\label{subsec: weak to restricted ABC}

We now reduce the restricted-graph statement in
\Cref{thm: ABC sum-product} to the weak theorem above.

\begin{prop}[Weak ABC implies restricted-graph ABC]
\label{prop: weak ABC implies restricted ABC}
Let $0<\beta\le\alpha<1$ and
$\gamma\in(\alpha-\beta,1]$. Choose
$\bar\beta<\beta$ sufficiently close to $\beta$ that
$\gamma>\alpha-\bar\beta$. If
\Cref{thm: weak ABC sum-product} holds with the parameters
$(\alpha,\bar\beta,\gamma)$, then
\Cref{thm: ABC sum-product} holds with the parameters
$(\alpha,\beta,\gamma)$.
\end{prop}

\begin{proof}
Let $\chi_{\mathrm w}>0$ and $\delta_{\mathrm w}>0$ be the constants from
\Cref{thm: weak ABC sum-product} for
$(\alpha,\bar\beta,\gamma)$. We choose
$$
0<\chi\ll_{\alpha,\beta,\bar\beta,\gamma}\chi_{\mathrm w},
\quad
\beta-\bar\beta,
\quad
\gamma-(\alpha-\bar\beta),
$$
and then take $\delta=p^{-n}>0$ sufficiently small.

Suppose, towards a contradiction, that the restricted-graph conclusion fails.
Then there exist sets $A,B\subset\Z/p^n\Z$ and a nonempty
$(\delta,\gamma,\delta^{-\chi})$-set
$C\subset(\Z/p^n\Z)^\times$ satisfying the hypotheses of
\Cref{thm: ABC sum-product}, such that for every $c\in C$ there is a graph
$G_c\subset A\times B$ with
$$
|G_c|\ge\delta^\chi|A||B|
$$
and
\begin{equation}
\label{eq: weak to restricted small restricted sumset}
|\{a+cb:(a,b)\in G_c\}|_\delta
<
\delta^{-\chi}|A|.
\end{equation}

Since $c$ is a unit, multiplication by $c$ is a bijection of
$\Z/p^n\Z$. Hence the graph $G_c$ may be viewed as a graph between
$A$ and $cB$, with the same cardinality. By Cauchy's inequality,
\eqref{eq: weak to restricted small restricted sumset} gives
\begin{equation}
\label{eq: weak to restricted graph energy}
\mathcal E_\delta(A,cB)
\ge
\delta^{O(\chi)}|A||B|^2,
\qquad c\in C.
\end{equation}
Applying the graph Balog--Szemerédi--Gowers theorem, followed by
\Cref{lem: plunnecke ruzsa}, we obtain for every $c\in C$ subsets
$A_c\subset A$ and $B_c\subset B$ such that
$$
|A_c|\ge\delta^{O(\chi)}|A|,
\qquad
|B_c|\ge\delta^{O(\chi)}|B|,
$$
and
\begin{equation}
\label{eq: weak to restricted structured pair}
|A_c+cB_c|_\delta
\le
\delta^{-O(\chi)}|A|.
\end{equation}

We now use the standard uniformization argument from
\cite[Section 5.1]{orponen2024discretised}. Pigeonholing the density profiles
of the sets $B_c$, and then the relevant translates of the sets $A_c$, yields
a subset $C'\subset C$ and sets $A'\subset\Z/p^n\Z$, $B'\subset B$ such that
\begin{equation}
\label{eq: weak to restricted uniformized sizes}
|C'|\ge\delta^{O(\chi)}|C|,
\qquad
|B'|\ge\delta^{O(\chi)}|B|,
\qquad
|A'|\le\delta^{-\alpha+O(\chi)},
\end{equation}
and
\begin{equation}
\label{eq: weak to restricted uniformized small sumsets}
|A'+cB'|_\delta
\le
\delta^{-O(\chi)}|A'|,
\qquad c\in C'.
\end{equation}
Moreover, after discarding another factor $\delta^{O(\chi)}$ if necessary,
$B'$ is a
$(\delta,\bar\beta,\delta^{-O(\chi)})$-set. Indeed, for every
$r\in[\delta,1]\cap p^{-\N}$ and every $r$-ball $J\subset\Z_p$,
$$
|B'\cap J|
\le
|B\cap J|
\le
\delta^{-\chi}r^\beta|B|
\le
\delta^{-O(\chi)}r^{\bar\beta}|B'|.
$$
Likewise, $C'$ remains a
$(\delta,\gamma,\delta^{-O(\chi)})$-set and is still contained in
$(\Z/p^n\Z)^\times$.

Choose $\chi>0$ sufficiently small that all losses
$\delta^{-O(\chi)}$ are bounded by $\delta^{-\chi_{\mathrm w}}$ and all
density losses are acceptable in the weak theorem. Then the sets
$A',B',C'$ satisfy the hypotheses of
\Cref{thm: weak ABC sum-product} with parameters
$(\alpha,\bar\beta,\gamma)$. Hence there exists $c\in C'$ such that
$$
|A'+cB'|_\delta
\ge
\delta^{-\chi_{\mathrm w}}|A'|,
$$
contradicting
\eqref{eq: weak to restricted uniformized small sumsets}. This proves the
restricted-graph theorem.
\end{proof}

For the rest of the section, we prove
\Cref{thm: weak ABC sum-product}.

\subsection{A discretised expansion estimate}\label{sec: discretised expansion estimate}

We first record the $p$-adic analogue of the expansion estimate used in \cite[Section 3.2]{orponen2023projections}. This is the only point where the projection-theoretic input enters the proof of the ABC theorem. The overall strategy is the same as in the real case: one removes a near-diagonal
part of the quadruple space, applies a two-chart ratio thin-tube estimate to the
remaining pairs, and then combines the resulting family of parameters with
Kaufman's projection theorem.

There are, however, two minor non-archimedean bookkeeping differences which are worth making explicit. First, the removal of the near-diagonal region is formulated using parent balls. If $I$ is a $p$-adic ball of radius $\rho$, we write $I^{+}$ for the unique ball of radius $p\rho$ containing $I$. This replaces the enlarged intervals used in the real proof. Its only purpose is to ensure that, after discarding the contribution with $a_1\in I$ and $a_2\in I^{+}$, the denominator
$a_1-a_2$ has a uniform lower bound of size $\rho$.

Second, unlike in the real setting, the quotient $\frac{b_1-b_2}{a_1-a_2}$ need not belong to $\Z_p$. As mentioned in \Cref{sec: Introduction}, we will work in two affine charts covering all non-diagonal pairs. Define
\begin{align}
\Omega_x&:=\left\{
(z_1,z_2)\in(\Q_p^2)^2:
||a_1-a_2||_p\ge ||b_1-b_2||_p,\
a_1\neq a_2\right\},\\
\Omega_y&:=\left\{
(z_1,z_2)\in(\Q_p^2)^2:
||b_1-b_2||_p>||a_1-a_2||_p\right\},
\end{align}
where $z_i=(a_i,b_i)$. On $\Omega_x$ we consider the ratio map
\begin{equation}\label{eq: ratio chart x}
\rho_x(z_1,z_2):=\frac{b_1-b_2}{a_1-a_2}\in\Z_p,
\end{equation}
while on $\Omega_y$ we interchange the two coordinates and use
\begin{equation}\label{eq: ratio chart y}
\rho_y(z_1,z_2):=\frac{a_1-a_2}{b_1-b_2}\in\Z_p.
\end{equation}

Since the ultrametric norm is given by the maximum of the coordinate norms,
every non-diagonal pair belongs to exactly one of the two charts. Thus
$\Omega_x\cup\Omega_y$ covers all non-diagonal pairs.

The fibres of these ratio maps play the role of Euclidean radial projections.
Indeed, after fixing one endpoint, the inverse image of a ball under either
$\rho_x$ or $\rho_y$ is contained in a bounded number of $p$-adic tubes of the
corresponding slope range. This observation replaces the radial-projection
machinery in \cite[Proposition~3.7]{orponen2023projections}. All losses arising
from the chart decomposition are absorbed into the
$\delta^{-O(\varepsilon)}$ factors appearing throughout the argument.

\begin{defi}[$(t,C)$-Frostman and $(\delta,t,C)$-Frostman measures]\label{defi: Frostman measure}
Let $t>0$ and $C\ge 1$. A Borel measure $\mu$ on $\Q_p^d$ is called a
$(t,C)$-Frostman measure if $$\mu(B(x,r))\le Cr^t$$ for all $x\in\Q_p^d$ and all
$r>0$. If $\delta\in p^{-\mathbb N}$ and the same inequality is assumed only for
$r\ge \delta$, then $\mu$ is called a $(\delta,t,C)$-Frostman measure.
\end{defi}

\begin{defi}[$p$-adic thin tubes]\label{defi: p-adic thin tubes}
Let $\mu,\nu$ be Borel probability measures on $\Z_p^2$. We say that the pair $(\mu,\nu)$ has $(\sigma,K,c)$-thin tubes if there exists a Borel set $\mathcal H\subset \Z_p^2\times\Z_p^2$ with
$$(\mu\times\nu)(\mathcal H)\ge c$$
such that, for every $x\in \Z_p^2$, every scale $r\in p^{-\mathbb N}$ with $\delta\le r\le 1$, and every $p$-adic $r$-tube $T\subset \Z_p^2$ with $x\in T$, one has
$$
\nu\{y\in T:(x,y)\in\mathcal H\}\le K r^{\sigma}.
$$
\end{defi}

The following proposition is the $p$-adic analogue of \cite[Proposition 3.6]{orponen2023projections}. We state it in terms of tubes rather than directions, in order to avoid introducing the projective line. 

\begin{prop}[Thin tubes from two-chart ratio projections]
\label{prop: p-adic thin tubes}
Let $\tau,\kappa\in(0,1]$ and $\sigma\in(0,\tau)$. Then there exist
$K=K(\tau,\kappa,\sigma)\ge 1$,
$\varepsilon_0=\varepsilon_0(\tau,\kappa,\sigma)>0$, and
$\delta_0=\delta_0(\tau,\kappa,\sigma)>0$ such that the following holds for all
$\delta=p^{-n}\in(0,\delta_0]$ and all
$\varepsilon\in(0,\varepsilon_0)$.

Let $\mu_1,\mu_2$ be $(\delta,\tau,\delta^{-\varepsilon})$-Frostman probability
measures on $\Z_p^2$. Assume that they satisfy the following tube
non-concentration condition: for every $p$-adic $r$-tube $T\subset\Z_p^2$ and
every $r\in p^{-\mathbb N}\cap[\delta,1]$,
\begin{equation}
\label{eq:p-adic-thin-tubes-tube-Frostman}
\mu_i(T)\le \delta^{-\varepsilon}r^{\kappa},
\qquad i=1,2.
\end{equation}
Assume also that their supports are separated:
\begin{equation}
\label{eq:p-adic-thin-tubes-separated-supports}
\dist(\spt\mu_1,\spt\mu_2)\ge \delta^{\varepsilon}.
\end{equation}
Then both pairs $(\mu_1,\mu_2)$ and $(\mu_2,\mu_1)$ have
$(\sigma,\delta^{-K\varepsilon},1-K\delta^{\varepsilon})$-thin tubes.
\end{prop}

We first record the $p$-adic analogue of Bourgain's discretized projection theorem,
which is the basic projection-theoretic input behind the argument. Instead of revisiting Bourgain's original argument
\cite{bourgain2010discretized}, we rely on the recent streamlined proof of a broad extension of Bourgain's theorem from O'Regan \cite[Theorem 1.3]{o2026bourgain}, covering both the real and $p$-adic settings. In the proof of \Cref{prop: p-adic thin tubes}, this result supplies the nontrivial projection expansion needed to obtain the thin-tube improvement.

\begin{lemma}[$p$-adic Bourgain projection theorem]
\label{lem: padic Bourgain projection}
Let $s\in(0,1)$ and $t\in(0,2)$. Then there exist constants $\eta_{\mathrm B}=\eta_{\mathrm B}(s,t)>0$,
$\epsilon_{\mathrm B}=\epsilon_{\mathrm B}(s,t)>0$, and $\delta_{\mathrm B}=\delta_{\mathrm B}(s,t)>0$ such that the following holds for all $\delta=p^{-n}\in(0,\delta_{\mathrm B}]$. Let $P\subset D_\delta(\Z_p^2)$ be a $(\delta,t,\delta^{-\epsilon_{\mathrm B}})$-set satisfying
$$
|P|\ge \delta^{-t+\epsilon_{\mathrm B}},
$$
and let $E\subset \Z_p$ be a $(\delta,s,\delta^{-\epsilon_{\mathrm B}})$-set. Then
there exists $\theta\in E$ such that, for every subset $P'\subset P$ with
$|P'|\ge \delta^{\epsilon_{\mathrm B}}|P|$, one has
$$
|\pi_\theta(P')|_\delta\ge \delta^{-t/2-\eta_{\mathrm B}}.
$$
\end{lemma}

\begin{proof}
We apply O'Regan's Bourgain projection theorem, more precisely
\cite[Theorem 1.3]{o2026bourgain}, in the special case where the normed
division algebra is $E_{\mathrm{OR}}=\Q_p$. To avoid confusion with the notation in
the present paper, let us first spell out the correspondence of parameters.

In O'Regan's notation, the theorem concerns the maps
$$
\pi_x(a,b)=a+xb,\qquad (a,b)\in E_{\mathrm{OR}}\times E_{\mathrm{OR}},\ x\in E_{\mathrm{OR}}.
$$
Thus, when $E_{\mathrm{OR}}=\Q_p$, these are exactly the affine projections used in
this paper. The ambient dimension in O'Regan's theorem is $d=1$. His planar set
$G\subset E_{\mathrm{OR}}\times E_{\mathrm{OR}}$ will be the set
$$
Q:=L_{\theta_0}(P)\subset \Q_p^2
$$
defined below. If $|Q|=\delta^{-\sigma_\delta}$, then O'Regan's parameter $\sigma$ is $\sigma_\delta$. The Frostman exponent of the planar set in O'Regan's theorem will be chosen as $s_{\mathrm{OR}}:=t-2\epsilon_{\mathrm B}$,
whereas the Frostman exponent of the parameter set $X$ in O'Regan's theorem is
the number $s$ appearing in the present lemma.

O'Regan's theorem also assumes that the parameter set $X$ strongly avoids proper
subalgebras. In the present case this condition has a very simple meaning. The
only proper subalgebra of $\Q_p$ is $\{0\}$, so avoiding proper subalgebras just
means that the projection parameters are not concentrated too close to $0$. We now
normalize the direction set so that this condition is automatic.

Fix $\theta_0\in E$ and set $r_0:=\delta^{4\epsilon_{\mathrm B}/s}$. By the $(\delta,s,\delta^{-\epsilon_{\mathrm B}})$-set condition,
$$
|E\cap B(\theta_0,r_0)|\le
\delta^{-\epsilon_{\mathrm B}}r_0^s|E|
=\delta^{3\epsilon_{\mathrm B}}|E|.
$$
Thus, after decreasing $\delta_{\mathrm B}$ if necessary, the set $E_0:=E\setminus B(\theta_0,r_0)$ satisfies $|E_0|\ge \frac12 |E|$. Consequently $E_0$ is still a
$(\delta,s,O(\delta^{-\epsilon_{\mathrm B}}))$-set. Define $X:=E_0-\theta_0$. Then $X\subset\Z_p$ is a $(\delta,s,O(\delta^{-\epsilon_{\mathrm B}}))$-set, and
every $x\in X$ satisfies $||x||_p\ge r_0$.
Since the only proper subalgebra of $\Q_p$ is $\{0\}$, this says that $X$ avoids
proper subalgebras with constant at worst $\delta^{-O(\epsilon_{\mathrm B})}$.
Moreover, it gives the strong version required by O'Regan: if
$X'\subset X$ has
$$
|X'|\ge \delta^{O(\epsilon_{\mathrm B})}|X|,
$$
then $X'$ is non-empty, and every point of $X'$ is still at distance at least
$r_0$ from $\{0\}$. Thus $X$ strongly avoids proper subalgebras, again with
avoidance constant $\delta^{-O(\epsilon_{\mathrm B})}$.

Next, we translate the direction parameter. Let $L_{\theta_0}(a,b):=(a+\theta_0 b,b)$.  This is an isometric automorphism of $\Z_p^2$. Set $Q:=L_{\theta_0}(P)$. Then $Q$ is a $(\delta,t,O(\delta^{-\epsilon_{\mathrm B}}))$-set and $|Q|=|P|$. Since $s_{\mathrm{OR}}=t-2\epsilon_{\mathrm B}<t$, it follows that $Q$ is also a $(\delta,s_{\mathrm{OR}},O(\delta^{-\epsilon_{\mathrm B}}))$-set. Indeed, for $0<r\le 1$ one has $r^t\le r^{s_{\mathrm{OR}}}$.

Moreover, for $x\in X$ and $(a,b)\in\Z_p^2$,
$$
\pi_x(L_{\theta_0}(a,b))
=
a+(\theta_0+x)b
=
\pi_{\theta_0+x}(a,b).
$$
Thus applying O'Regan's projection theorem to $Q$ and $X$ gives information about
the original projections of $P$ in the directions $\theta=\theta_0+x\in E$.

Now write
$|Q|=\delta^{-\sigma_\delta}$.
Since $|Q|=|P|\ge \delta^{-t+\epsilon_{\mathrm B}}$, we have
$\sigma_\delta\ge t-\epsilon_{\mathrm B}$.
On the other hand, $Q\subset D_\delta(\Z_p^2)$, so $\sigma_\delta\le 2$. Hence
$$
s_{\mathrm{OR}}=t-2\epsilon_{\mathrm B}
\le
\sigma_\delta
\le 2.
$$
Thus the condition $0<s_{\mathrm{OR}}\le \sigma_\delta\le 2d$ in O'Regan's theorem
is satisfied, since here $d=1$.

We now apply \cite[Theorem 1.3]{o2026bourgain} with the following
substitutions:
$$
E_{\mathrm{OR}}=\Q_p,\quad d=1,\quad
G_{\mathrm{OR}}=Q,\quad
\sigma_{\mathrm{OR}}=\sigma_\delta,\quad
s_{\mathrm{OR}}=t-2\epsilon_{\mathrm B},\quad
t_{\mathrm{OR}}=s,\quad
X_{\mathrm{OR}}=X.
$$
The theorem gives constants $c=c(s,t)>0$ and $C_{\mathrm{OR}}\ge 1$, after
choosing $\epsilon_{\mathrm B}$ sufficiently small and $\delta_{\mathrm B}$
sufficiently small, such that there is a set $Y\subset X$ with
$$
|Y|\ge (1-\delta^{O(\epsilon_{\mathrm B})})|X|
$$
and, for every $x\in Y$ and every subset $Q'\subset Q$ satisfying $|Q'|\ge \delta^{O(\epsilon_{\mathrm B})}|Q|$, one has
$$
|\pi_x(Q')|_\delta
\ge
\delta^{-c}|Q|^{1/2}.
$$
Here the implicit constants in $O(\epsilon_{\mathrm B})$ depend only on the fixed parameters $s,t$. Since $Y$ is non-empty, choose one $x\in Y$ and put
$\theta:=\theta_0+x\in E$.

We now choose $\epsilon_{\mathrm B}>0$ small enough so that every subset
$P'\subset P$ with $|P'|\ge \delta^{\epsilon_{\mathrm B}}|P|$
gives an admissible subset
$Q':=L_{\theta_0}(P')\subset Q$ for the preceding application of O'Regan's theorem. Since $L_{\theta_0}$ is bijective at scale $\delta$, we have $|Q'|=|P'|$ and $|Q|=|P|$. Therefore
$$
|\pi_\theta(P')|_\delta=|\pi_x(Q')|_\delta
\ge\delta^{-c}|Q|^{1/2}=\delta^{-c}|P|^{1/2}.
$$
Using $|P|\ge \delta^{-t+\epsilon_{\mathrm B}}$, this yields
$$
|\pi_\theta(P')|_\delta
\ge
\delta^{-c}\delta^{-t/2+\epsilon_{\mathrm B}/2}.
$$
Finally, after decreasing $\epsilon_{\mathrm B}$ if necessary, set for instance $\eta_{\mathrm B}:=\frac{c}{4}$. Then
$$
\delta^{-c}\delta^{-t/2+\epsilon_{\mathrm B}/2}
\ge
\delta^{-t/2-\eta_{\mathrm B}},
$$
for all sufficiently small $\delta$. This proves the desired estimate for this
choice of $\theta\in E$ and for every $P'\subset P$ with
$|P'|\ge \delta^{\epsilon_{\mathrm B}}|P|$.
\end{proof}

Before proving \Cref{prop: p-adic thin tubes}, we isolate the bootstrapping step
which upgrades a weak thin-tube exponent to a stronger one. This is the
$p$-adic counterpart of the self-improvement mechanism in
\cite[Lemma 2.9]{orponen2024kaufman}, which is also the input iterated in
\cite[Proposition 3.6]{orponen2023projections}. The only geometric difference is
that the Euclidean proof uses radial projections, whereas in the present
non-archimedean setting we use the two ratio charts together with the
$p$-adic Bourgain projection theorem from \Cref{lem: padic Bourgain projection}.
The rest of the argument is the same bootstrapping and two-ends/double-counting
scheme.

\begin{lemma}[p-adic thin-tube self-improvement]
\label{lem: padic thin tube self improvement}
Let $0<\kappa<\sigma<\tau<1$. There exists
$\eta=\eta(\tau,\kappa,\sigma)>0$ with the following property.

Let $\sigma_0\in[\kappa,\sigma)$ and let $K_0,L\ge 1$. Then there exist constants
$K_1=K_1(K_0,L,\tau,\kappa,\sigma)\ge 1$,
$\varepsilon_1=\varepsilon_1(K_0,L,\tau,\kappa,\sigma)>0$, and
$\delta_1=\delta_1(K_0,L,\tau,\kappa,\sigma)>0$ such that the following holds for
all $\delta\in(0,\delta_1]$ and all $\varepsilon\in(0,\varepsilon_1)$.

Let $\mu_1,\mu_2$ be probability measures on $\Z_p^2$ satisfying the Frostman and
tube non-concentration hypotheses of \Cref{prop: p-adic thin tubes}: for
$i=1,2$,
\begin{equation}
\label{eq: self improvement Frostman}
\mu_i(B(z,\rho))\le \delta^{-\varepsilon}\rho^\tau,
\qquad z\in\Z_p^2,\quad \rho\in p^{-\mathbb N}\cap[\delta,1],
\end{equation}
and
\begin{equation}
\label{eq: self improvement tube Frostman}
\mu_i(T)\le \delta^{-\varepsilon}\rho^\kappa
\end{equation}
for every $p$-adic $\rho$-tube $T\subset\Z_p^2$ and every
$\rho\in p^{-\mathbb N}\cap[\delta,1]$. Assume also that the supports are
separated as in \Cref{prop: p-adic thin tubes}.

If both ordered pairs $(\mu_1,\mu_2)$ and $(\mu_2,\mu_1)$ have
$(\sigma_0,\delta^{-K_0\varepsilon},1-L\delta^\varepsilon)$-thin tubes, then both
ordered pairs have
$(\min\{\sigma_0+\eta,\sigma\},\delta^{-K_1\varepsilon},
1-5L\delta^\varepsilon)$-thin tubes.
\end{lemma}

\begin{proof}
The proof follows the structure of \cite[Lemma 2.9]{orponen2024kaufman}. We spell out the reduction to make clear where the $p$-adic inputs enter.

Set $\sigma_1:=\min\{\sigma_0+\eta,\sigma\}$, where $\eta>0$ will be chosen sufficiently small depending only on $\tau,\kappa,\sigma$. We prove the assertion for $(\mu_1,\mu_2)$; the proof for $(\mu_2,\mu_1)$ is identical.

Let $H_0\subset\Z_p^2\times\Z_p^2$ be a good set witnessing that
$(\mu_1,\mu_2)$ has
$(\sigma_0,\delta^{-K_0\varepsilon},1-L\delta^\varepsilon)$-thin tubes. Thus
$(\mu_1\times\mu_2)(H_0)\ge 1-L\delta^\varepsilon$, and for every $x\in\Z_p^2$,
every $\rho\in p^{-\mathbb N}\cap[\delta,1]$, and every $p$-adic $\rho$-tube
$T$ through $x$,
\begin{equation}
\label{eq: old thin tube section estimate}
\mu_2(T\cap H_{0,x})
\le
\delta^{-K_0\varepsilon}\rho^{\sigma_0},
\end{equation}
where $H_{0,x}:=\{y\in\Z_p^2:(x,y)\in H_0\}$.

We claim that, after increasing $K_1$ and decreasing $\varepsilon_1,\delta_1$ if
necessary, there is a subset $H_1\subset H_0$ with
$(\mu_1\times\mu_2)(H_1)\ge 1-5L\delta^\varepsilon$ such that for every
$x\in\Z_p^2$, every $\rho\in p^{-\mathbb N}\cap[\delta,1]$, and every
$p$-adic $\rho$-tube $T$ through $x$,
\begin{equation}
\label{eq: improved section estimate}
\mu_2(T\cap H_{1,x})
\le
\delta^{-K_1\varepsilon}\rho^{\sigma_1}.
\end{equation}
This is exactly the desired improved thin-tube estimate for $(\mu_1,\mu_2)$.

Suppose, to the contrary, that such an $H_1$ cannot be found. Then the set of
pairs $(x,y)\in H_0$ for which \eqref{eq: improved section estimate} fails has
$(\mu_1\times\mu_2)$-measure at least a constant multiple of $L\delta^\varepsilon$.
Equivalently, after pigeonholing over the $O(\log(1/\delta))$ possible $p$-adic
scales, there exists a scale $r\in p^{-\mathbb N}\cap[\delta,1]$ and a set
$B_r\subset H_0$ satisfying
\begin{equation}
\label{eq: bad scale mass}
(\mu_1\times\mu_2)(B_r)\ge \delta^{O(\varepsilon)}
\end{equation}
such that for every $(x,y)\in B_r$ there is an $r$-tube $T$ through $x$, with
$y\in T$, for which
\begin{equation}
\label{eq: bad tube large mass}
\mu_2(T\cap H_{0,x})>
\delta^{-K_1\varepsilon}r^{\sigma_1}.
\end{equation}
Here and below the constants implicit in $O(\varepsilon)$ depend only on the fixed
parameters.

Let $X_r:=\{x:\mu_2(B_{r,x})\ge \delta^{O(\varepsilon)}\}$. By
\eqref{eq: bad scale mass}, after adjusting the implicit constant, we have
$\mu_1(X_r)\ge \delta^{O(\varepsilon)}$. For each $x\in X_r$, let
$\mathcal T_x$ be the family of all $r$-tubes $T$ through $x$ satisfying
\eqref{eq: bad tube large mass}. By discarding tubes which carry only a negligible
part of $B_{r,x}$, we may choose a set $Y_x\subset B_{r,x}$ such that
\begin{equation}
\label{eq: Yx positive mass}
\mu_2(Y_x)\ge \delta^{O(\varepsilon)}
\end{equation}
and $Y_x$ is covered by the tubes in $\mathcal T_x$. Moreover, for every
$T\in\mathcal T_x$,
\begin{equation}
\label{eq: tube section mass range}
\delta^{-K_1\varepsilon}r^{\sigma_1}
\le
\mu_2(T\cap Y_x)
\le
\delta^{-K_0\varepsilon}r^{\sigma_0}.
\end{equation}
The upper bound is \eqref{eq: old thin tube section estimate}.

The old thin-tube estimate also gives non-concentration for the tube families.
Indeed, if $T_0$ is a $p$-adic $\rho$-tube through $x$, with $r\le\rho\le 1$,
then
$$
\sum_{\substack{T\in\mathcal T_x\\ T\subset T_0}}
\mu_2(T\cap Y_x)
\le
\mu_2(T_0\cap H_{0,x})
\le
\delta^{-K_0\varepsilon}\rho^{\sigma_0}.
$$
Combining this with the lower bound in \eqref{eq: tube section mass range}, and
normalizing using \eqref{eq: Yx positive mass} and the upper bound in
\eqref{eq: tube section mass range}, gives
\begin{equation}
\label{eq: Tx nonconcentration}
|\{T\in\mathcal T_x:T\subset T_0\}|
\le
\delta^{-O(\varepsilon)}
\left(\frac{\rho}{r}\right)^{\sigma_0}
|\mathcal T_x|.
\end{equation}
Thus, after identifying $r$-tubes through $x$ with their slope parameters in an
affine chart, $\mathcal T_x$ is an
$(r,\sigma_0,\delta^{-O(\varepsilon)})$-set of directions.

The rest of the proof follows the two alternatives in \cite[Lemma 2.9]{orponen2024kaufman}. If, for a positive proportion of the pairs counted above, the sections $Y_x\cap T$ are not concentrated in smaller balls then the base points $x$ and the direction sets $\mathcal T_x$ form a discretized Furstenberg-type configuration. In the Euclidean argument the required expansion comes from the radial projection/Furstenberg input. Here it is supplied by \Cref{lem: padic Bourgain projection}, after rescaling the chosen ratio chart to unit scale. The conclusion gives a direction in the slope set of some
$\mathcal T_x$ for which every large subset of the corresponding portion of
$Y_x$ has projection covering number larger by a fixed power. Choosing $\eta=\eta(\tau,\kappa,\sigma)>0$ sufficiently small, and then choosing $K_1$
large depending on $K_0,L,\tau,\kappa,\sigma$, contradicts the tube-section bounds in \eqref{eq: tube section mass range} with $\sigma_1=\min\{\sigma_0+\eta,\sigma\}$.

In the remaining alternative, a positive proportion of the sections $Y_x\cap T$
are concentrated in smaller balls. The double-counting part of
\cite[Lemma 2.9]{orponen2024kaufman} is purely measure-theoretic and uses only
the lower bound \eqref{eq: Yx positive mass}, the tube-section bounds
\eqref{eq: tube section mass range}, and the tube non-concentration hypothesis
\eqref{eq: self improvement tube Frostman}. It therefore applies verbatim in the
$p$-adic setting. It gives a ball $B(y,\rho)$, with
$\delta\le\rho\le 1$, such that
$$\mu_2(B(y,\rho))>\delta^{-\varepsilon}\rho^\tau,$$
provided $\eta$ is sufficiently small in terms of $\tau,\kappa,\sigma$ and then
$\varepsilon_1,\delta_1$ are sufficiently small. This contradicts the Frostman
condition \eqref{eq: self improvement Frostman}.

Both alternatives are impossible. Hence $H_1$ exists, and $(\mu_1,\mu_2)$ has
$(\sigma_1,\delta^{-K_1\varepsilon},1-5L\delta^\varepsilon)$-thin tubes. Repeating
the same argument with $\mu_1$ and $\mu_2$ interchanged gives the same conclusion
for $(\mu_2,\mu_1)$. This proves the lemma.
\end{proof}

\begin{proof}[Proof of \Cref{prop: p-adic thin tubes}]
We first obtain the initial thin-tube estimate from the tube non-concentration
assumption. Let $x\in\Z_p^2$, let $r\in p^{-\mathbb N}\cap[\delta,1]$, and let
$T\subset\Z_p^2$ be a $p$-adic $r$-tube through $x$. By
\eqref{eq:p-adic-thin-tubes-tube-Frostman},
$$\mu_i(T)\le \delta^{-\varepsilon}r^\kappa,\qquad i=1,2.$$
Thus, taking the good set to be all of $\Z_p^2\times\Z_p^2$, both ordered pairs
$(\mu_1,\mu_2)$ and $(\mu_2,\mu_1)$ have
$(\kappa,\delta^{-\varepsilon},1)$-thin tubes.

If $\sigma\le\kappa$, this already implies the desired conclusion, after
increasing the constant $K$, since $r^\kappa\le r^\sigma$ for $0<r\le 1$.
Therefore we may assume that $\kappa<\sigma$.

Apply \Cref{lem: padic thin tube self improvement} with
$\sigma_0=\kappa$, $K_0=1$, and $L=1$. We obtain that both ordered pairs have
$$(\min\{\kappa+\eta,\sigma\},\delta^{-K_1\varepsilon},
1-5\delta^\varepsilon)\text{-thin tubes},$$
where $\eta=\eta(\tau,\kappa,\sigma)>0$.

If $\kappa+\eta\ge\sigma$, we are done. Otherwise, apply
\Cref{lem: padic thin tube self improvement} again, with the new value of
$\sigma_0$ and with the updated values of $K_0$ and $L$. Each application raises
the thin-tube exponent by at least $\eta$, until the target exponent $\sigma$ is
reached, at the cost of replacing the thin-tube constant by another power of
$\delta^{-\varepsilon}$ and increasing the exceptional mass by another bounded
multiple of $\delta^\varepsilon$.

Let $N$ be the least integer such that $\kappa+N\eta\ge\sigma$. Since
$N=O_{\tau,\kappa,\sigma}(1)$, after $N$ iterations we obtain
$$(\sigma,\delta^{-K\varepsilon},1-K\delta^\varepsilon)\text{-thin tubes}$$
for both ordered pairs, where $K=K(\tau,\kappa,\sigma)\ge 1$. Finally, choose $\varepsilon_0>0$ and $\delta_0>0$ sufficiently small in terms of the fixed
parameters so that every application of
\Cref{lem: padic thin tube self improvement} is legitimate. This proves the
proposition.
\end{proof}

\begin{prop}[$p$-adic Kaufman projection theorem]\label{prop: p-adic Kaufman projection}
Let $\delta=p^{-n}$, let $s\in(0,1)$, let $\alpha>0$, and let $C_1,C_2\ge 1$.
Let $P\subset(\Z/p^n\Z)^2$ be a nonempty $(\delta,s,C_1)$-set, and let $\nu$ be a
$(s,C_2)$-Frostman probability measure on $\Z_p$. Then there exists
$\theta\in\spt\nu$ such that, for every subset $P'\subset P$ with
$|P'|\ge \alpha |P|$, one has
$$
|\pi_{\theta}(P')|_{\delta}
\gtrsim_{s}
\frac{\alpha^2}{C_1C_2\log^2(1/\delta)}\delta^{-s}.
$$
\end{prop}

\begin{proof}
Throughout the proof, all implicit constants may depend on $s$. We write
$\pi_{\theta}\mu$ for the push-forward of a measure $\mu$ under the map
$\pi_{\theta}$.

Let $\mu$ be the normalized Haar measure on the union of the $\delta$-cosets in
$P$, namely
$$
\mu:=\frac{1}{|P|}\sum_{\mathscr p\in P}\delta^{-2}\mathbf 1_{\mathscr p}\,dz,
$$
where $dz$ denotes the Haar measure on $\Z_p^2$ normalized by $dz(\Z_p^2)=1$.
Since $P$ is a $(\delta,s,C_1)$-set, we have $\mu(B(z,r))\lesssim C_1r^s$ for all
$\delta\le r\le 1$. For $0<r<\delta$, the trivial estimate inside a single
$\delta$-coset gives $$\mu(B(z,r))\lesssim C_1\delta^{s-2}r^2.$$ Therefore, decomposing according to $p$-adic scales, we obtain
\begin{equation}\label{eq:kaufman-energy-mu}
\begin{aligned}
I_s(\mu)
&:=\int_{\Z_p^2}\int_{\Z_p^2}|z-z'|_p^{-s}\,d\mu(z')d\mu(z)\\
&\lesssim_{s}
\int_{\Z_p^2}
\left(
\sum_{\substack{r\in p^{-\mathbb N}\\ \delta\le r\le 1}}
r^{-s}\mu(B(z,r))
+
\sum_{\substack{r\in p^{-\mathbb N}\\ 0<r<\delta}}
r^{-s}\mu(B(z,r))
\right)d\mu(z)\\
&\lesssim_{s}
\int_{\Z_p^2}
\left(
\sum_{\substack{r\in p^{-\mathbb N}\\ \delta\le r\le 1}}
C_1
+
\sum_{\substack{r\in p^{-\mathbb N}\\ 0<r<\delta}}
C_1\delta^{s-2}r^{2-s}
\right)d\mu(z)\\
&\lesssim_{s} C_1\log(1/\delta).
\end{aligned}
\end{equation}

Since the conclusion is unchanged if $\theta$ is replaced by another element of
the same $\delta$-coset, we may first replace $\nu$ by its $\delta$-smoothed
version. Let $\mathcal D_{\delta}(\Z_p)$ be the partition of $\Z_p$ into
$\delta$-balls, and set
$$
\nu_{\delta}:=\sum_{J\in\mathcal D_{\delta}(\Z_p)}
\nu(J)\delta^{-1}\mathbf 1_J\,d\theta.
$$
If $\theta\in\spt\nu_{\delta}$, then the $\delta$-ball containing $\theta$ meets
$\spt\nu$. Thus, after proving the estimate for some $\theta\in\spt\nu_{\delta}$,
we may choose $\theta_0\in\spt\nu$ with $|\theta-\theta_0|_p\le \delta$. Since
$|\pi_{\theta}(z)-\pi_{\theta_0}(z)|_p\le \delta$ for every $z\in\Z_p^2$, the
$\delta$-covering numbers of $\pi_{\theta}(P')$ and $\pi_{\theta_0}(P')$ are
comparable. It is therefore enough to work with $\nu_{\delta}$.

The measure $\nu_{\delta}$ satisfies
$$
\nu_{\delta}(B(\theta,r))
\lesssim C_2
\begin{cases}
\delta^{s-1}r, & 0<r<\delta,\\
r^s, & \delta\le r\le 1,\\
1, & r\ge 1.
\end{cases}
$$
For $v=(v_1,v_2)\in\Z_p^2\setminus\{0\}$, write
$|v|_p:=\max\{|v_1|_p,|v_2|_p\}$. Define
$$
\Phi(v):=\int_{\Z_p}|\pi_{\theta}(v)|_p^{-s}\,d\nu_{\delta}(\theta).
$$
For $0<r\le |v|_p$, the set $\{\theta\in\Z_p:|\pi_{\theta}(v)|_p\le r\}$ is empty
if $|v_2|_p<|v|_p$ and $r<|v|_p$. If $|v_2|_p=|v|_p$, then this set is a
$p$-adic ball in $\theta$ of radius $r/|v|_p$. Hence
\begin{equation}\label{eq:kaufman-theta-incidence}
\nu_{\delta}\{\theta\in\Z_p:|\pi_{\theta}(v)|_p\le r\}
\lesssim C_2
\begin{cases}
\delta^{s-1}\dfrac{r}{|v|_p}, & 0<r<\delta |v|_p,\\
\left(\dfrac{r}{|v|_p}\right)^s, & \delta |v|_p\le r\le |v|_p,\\
1, & r>|v|_p.
\end{cases}
\end{equation}
Using \eqref{eq:kaufman-theta-incidence} and decomposing according to $p$-adic
scales, we obtain
\begin{equation}\label{eq:kaufman-pointwise-average}
\begin{aligned}
\Phi(v)
&\lesssim_{s}
\sum_{\substack{r\in p^{-\mathbb N}\\ 0<r<\delta |v|_p}}r^{-s}\nu_{\delta}\{\theta:|\pi_{\theta}(v)|_p\le r\}+\sum_{\substack{r\in p^{-\mathbb N}\\ \delta |v|_p\le r\le |v|_p}}
r^{-s}\nu_{\delta}\{\theta:|\pi_{\theta}(v)|_p\le r\}+\sum_{\substack{r\in p^{-\mathbb N}\\ |v|_p<r\le 1}}r^{-s}\\
&\lesssim_{s}
C_2\sum_{\substack{r\in p^{-\mathbb N}\\ 0<r<\delta |v|_p}}r^{-s}\delta^{s-1}\frac{r}{|v|_p}+C_2\sum_{\substack{r\in p^{-\mathbb N}\\ \delta |v|_p\le r\le |v|_p}}r^{-s}\left(\frac{r}{|v|_p}\right)^s+\sum_{\substack{r\in p^{-\mathbb N}\\ |v|_p<r\le 1}}r^{-s}\\
&\lesssim_{s}C_2|v|_p^{-s}+C_2\log(1/\delta)|v|_p^{-s}+|v|_p^{-s}\\
&\lesssim_{s}C_2\log(1/\delta)|v|_p^{-s}.
\end{aligned}
\end{equation}

By Fubini's theorem and \eqref{eq:kaufman-pointwise-average}, we get
\begin{equation}\label{eq:kaufman-averaged-energy}
\begin{aligned}
\int_{\Z_p} I_s(\pi_{\theta}\mu)\,d\nu_{\delta}(\theta)
&=\int_{\Z_p}\int_{\Z_p^2}\int_{\Z_p^2}
|\pi_{\theta}(z)-\pi_{\theta}(z')|_p^{-s}\,d\mu(z')d\mu(z)d\nu_{\delta}(\theta)\\
&=\int_{\Z_p^2}\int_{\Z_p^2}\left(\int_{\Z_p}
|\pi_{\theta}(z-z')|_p^{-s}\,d\nu_{\delta}(\theta)\right)d\mu(z')d\mu(z)\\
&=\int_{\Z_p^2}\int_{\Z_p^2}\Phi(z-z')\,d\mu(z')d\mu(z)\\
&\lesssim_{s}C_2\log(1/\delta)
\int_{\Z_p^2}\int_{\Z_p^2}
|z-z'|_p^{-s}\,d\mu(z')d\mu(z)\\
&=C_2\log(1/\delta)I_s(\mu)\\
&\lesssim_{s}C_1C_2\log^2(1/\delta),
\end{aligned}
\end{equation}
where the last inequality follows from \eqref{eq:kaufman-energy-mu}. Hence there
exists $\theta\in\spt\nu_{\delta}$ such that
\begin{equation}\label{eq:kaufman-good-theta}
I_s(\pi_{\theta}\mu)\lesssim_{s} C_1C_2\log^2(1/\delta).
\end{equation}

Fix $P'\subset P$ with $|P'|\ge \alpha |P|$, and let $\mu'$ be the normalized
restriction of $\mu$ to the union of the $\delta$-cosets in $P'$. Since
$\mu(P')\ge \alpha$, \eqref{eq:kaufman-good-theta} gives
$$
I_s(\pi_{\theta}\mu')
\lesssim_{s}
\alpha^{-2}C_1C_2\log^2(1/\delta).
$$
Let $N:=|\pi_{\theta}(P')|_{\delta}$. The measure $\pi_{\theta}\mu'$ is supported
on $N$ many $\delta$-balls in $\Z_p$. Therefore, by Cauchy-Schwarz,
$$
I_s(\pi_{\theta}\mu')
\gtrsim_{s}
\delta^{-s}\sum_q(\pi_{\theta}\mu'(q))^2
\gtrsim
\delta^{-s}N^{-1},
$$
where the sum is over the $\delta$-balls $q$ meeting $\pi_{\theta}(P')$. Combining
the last two estimates gives
$$
|\pi_{\theta}(P')|_{\delta}
\gtrsim_{s}
\frac{\alpha^2}{C_1C_2\log^2(1/\delta)}\delta^{-s}.
$$
Finally, choose $\theta_0\in\spt\nu$ with $|\theta-\theta_0|_p\le \delta$. As
explained above, replacing $\theta$ by $\theta_0$ changes $\pi_{\theta}(P')$ only by
$\delta$-errors and hence preserves the same $\delta$-covering lower bound, up to
an absolute constant. Renaming $\theta_0$ as $\theta$ completes the proof.
\end{proof}

\begin{prop}[Discretised expansion]\label{prop: discretised expansion}
Let $s,t\in(0,1)$ and $\sigma\in[0,\min\{s+t,1\})$. Then there exist
$$
\varepsilon=\varepsilon(s,t,\sigma)>0,
\qquad\delta_0=\delta_0(s,t,\sigma,\varepsilon)>0
$$
such that the following holds for all $\delta=p^{-n}\in(0,\delta_0]$. Let
$A_1,A_2\subset \Z/p^n\Z$ be $(\delta,s,\delta^{-\varepsilon})$-sets, and let
$B_1,B_2\subset \Z/p^n\Z$ be $(\delta,t,\delta^{-\varepsilon})$-sets. Let
$P\subset(\Z/p^n\Z)^2$ be a $(\delta,s+t,\delta^{-\varepsilon})$-set. Then there
exists a set $$\mathcal G\subset A_1\times A_2\times B_1\times B_2$$ with
$$
|(A_1\times A_2\times B_1\times B_2)\setminus \mathcal G|
\le
\delta^\varepsilon |A_1||A_2||B_1||B_2|,
$$
such that for every $(a_1,a_2,b_1,b_2)\in\mathcal G$ and every
$X\subset P$ satisfying $|X|\ge \delta^\varepsilon |P|$, we have
$$
\left|{(b_1\pm b_2)a+(a_1\pm a_2)b:(a,b)\in X}\right|_\delta
\ge
\delta^{-\sigma}.
$$
\end{prop}

We emphasize that the non-unit differences appearing in the above proposition play a different role from the coefficient set in the ABC theorem. In the present expansion estimate, quantities such as $b_1-b_2$ or $a_1-a_2$ arise as auxiliary multipliers inside an additive-combinatorial argument. They are allowed to be non-units, in agreement with \Cref{thm: asymmetric BSG},
whose multiplier is arbitrary in $\mathbb Z/p^n\mathbb Z$. The possible loss caused by such multipliers is controlled by the near-diagonal removal and by the use of two affine charts:
when $|b_1-b_2|_p\le |a_1-a_2|_p$ we work with
$$(b_1-b_2)/(a_1-a_2)\in\mathbb Z_p,$$
and otherwise we interchange the two coordinates. Multiplying back by the corresponding denominator only changes the covering estimates by a factor absorbed into $\delta^{-O(\epsilon)}$.

This should not be confused with the coefficient set in the ABC theorem.
There the coefficients have already been normalized at their effective
$p$-adic scale and therefore lie in $(\Z/p^n\Z)^\times$. By contrast, the
differences $b_1-b_2$ and $a_1-a_2$ above are internal multipliers in the
expansion argument and need not be units.

\begin{proof}[Proof of \Cref{prop: discretised expansion}]
We prove the assertion for the sign choice
$$(b_1-b_2)a+(a_1-a_2)b.$$
The other choices of signs are identical, after replacing one of the sets $A_i$ or $B_i$ by its negative.

We identify all subsets of $\Z/p^n\Z$ with their canonical lifts to unions of $\delta$-cosets in $\Z_p$. Let $\mu_i$ be the normalized Haar measure on the union of the $\delta$-cosets in $A_i$, and let $\nu_i$ be the normalized Haar measure on the union of the $\delta$-cosets in $B_i$, for $i=1,2$. Thus $\mu_i$ is a $(\delta,s,O(\delta^{-\varepsilon}))$-Frostman probability measure, and $\nu_i$ is a $(\delta,t,O(\delta^{-\varepsilon}))$-Frostman probability measure. Set $\rho_i:=\mu_i\times\nu_i$, viewed as a probability measure on $\Z_p^2$ with coordinates $z_i=(a_i,b_i)$.

It is enough to prove that the set of bad quadruples has $(\rho_1\times\rho_2)$-measure at most $\delta^{\varepsilon}$, after decreasing $\varepsilon>0$ in terms of $s,t,\sigma$. Indeed, since all measures are normalized Haar measures on unions of $\delta$-cosets, this gives the desired cardinality estimate for the complement of $\mathcal G$.

Suppose, towards a contradiction, that the set $\mathcal E\subset (A_1\times B_1)\times(A_2\times B_2)$ of bad quadruples satisfies
\begin{equation}\label{eq:expansion-bad-mass}
(\rho_1\times\rho_2)(\mathcal E)\ge \delta^{\varepsilon}.
\end{equation}
By definition, for every $((a_1,b_1),(a_2,b_2))\in\mathcal E$, there exists a set $X\subset P$ with $|X|\ge \delta^{\varepsilon}|P|$ such that
\begin{equation}\label{eq:expansion-bad-small-projection}
\left|\{(b_1-b_2)a+(a_1-a_2)b:(a,b)\in X\}\right|_{\delta}<\delta^{-\sigma}.
\end{equation}

We first separate the two $a$-coordinates. Choose $\rho\in p^{-\mathbb N}$ such that
$$\delta^{3\varepsilon/s}\le \rho < p\delta^{3\varepsilon/s}.$$
Let $\mathcal I$ be the partition of $\Z_p$ into $\rho$-balls. Since $\mu_i$ is $(\delta,s,O(\delta^{-\varepsilon}))$-Frostman, every $I\in\mathcal I$ satisfies
$$\mu_i(I)\lesssim \delta^{-\varepsilon}\rho^s\lesssim \delta^{2\varepsilon}.$$
For $I\in\mathcal I$, let $I^{+}$ denote the unique $p$-adic ball of radius $p\rho$ containing $I$. Equivalently, if $\rho=p^{-m}$ and $I=a+p^m\Z_p$, then $I^{+}=a+p^{m-1}\Z_p$. We remove the contribution with $a_1\in I$ and $a_2\in I^{+}$ in order to avoid the near-diagonal region where $|a_1-a_2|_p$ may be too small. We have
\begin{equation}\label{eq:expansion-near-diagonal}
\sum_{I\in\mathcal I}
(\mu_1|_I\times\nu_1\times\mu_2|_{I^{+}}\times\nu_2)(\mathcal E)\le\sum_{I\in\mathcal I}\mu_1(I)\mu_2(I^{+})\lesssim\delta^{2\varepsilon}.
\end{equation}
Indeed, $\mu_2(I^{+})\lesssim \delta^{-\varepsilon}(p\rho)^s\lesssim \delta^{2\varepsilon}$ by the Frostman non-concentrate condition, while the balls $I$ are disjoint and $\sum_{I\in\mathcal I}\mu_1(I)\le 1$.

Combining \eqref{eq:expansion-bad-mass} and \eqref{eq:expansion-near-diagonal}, and taking $\delta_0$ sufficiently small, we obtain
$$
\sum_{I\in\mathcal I}
(\mu_1|_I\times\nu_1\times\mu_2|_{(I^{+})^c}\times\nu_2)(\mathcal E)\gtrsim\delta^{\varepsilon}.
$$
Since $|\mathcal I|\lesssim \rho^{-1}\lesssim \delta^{-3\varepsilon/s}$, pigeonholing gives a ball $I_0\in\mathcal I$ such that
\begin{equation}\label{eq:expansion-separated-pigeonhole}
(\mu_1|_{I_0}\times\nu_1\times\mu_2|_{(I_0^{+})^c}\times\nu_2)(\mathcal E)\gtrsim\delta^{\varepsilon+3\varepsilon/s}.
\end{equation}
Let $\bar\mu_1$ and $\bar\mu_2$ be the normalized restrictions of $\mu_1$ and $\mu_2$ to $I_0$ and $(I_0^{+})^c$, respectively, and put $\bar\rho_i:=\bar\mu_i\times\nu_i$. Since the normalizing factors are at least $\delta^{O(\varepsilon)}$, the measures $\bar\mu_i$ are $(\delta,s,\delta^{-O(\varepsilon)})$-Frostman probability measures. Moreover, after normalizing \eqref{eq:expansion-separated-pigeonhole}, we get
\begin{equation}\label{eq:expansion-normalized-bad-mass}
(\bar\rho_1\times\bar\rho_2)(\mathcal E)\gtrsim \delta^{O(\varepsilon)}.
\end{equation}
The construction also gives the separation estimate
\begin{equation}\label{eq:expansion-separated-supports}
||a_1-a_2||_p\ge \rho\gtrsim \delta^{O(\varepsilon)}
\end{equation}
whenever $(a_1,b_1)\in\spt\bar\rho_1$ and $(a_2,b_2)\in\spt\bar\rho_2$.

The product measures $\bar\rho_i$ are $(\delta,s+t,\delta^{-O(\varepsilon)})$-Frostman probability measures on $\Z_p^2$. In addition, for every $p$-adic $r$-tube $T$ and every $\delta\le r\le 1$,
\begin{equation}\label{eq:expansion-product-tube}
\bar\rho_i(T)\lesssim \delta^{-O(\varepsilon)}r^{\min\{s,t\}},
\qquad i=1,2.
\end{equation}
Indeed, after fixing one coordinate, the intersection of a $p$-adic tube with the corresponding fibre is an $r$-ball in the other coordinate, and the estimate follows from the Frostman non-concentrate condition for $\bar\mu_i$ and $\nu_i$.

Let $\tau:=\min\{s+t,1\}$, let $\kappa:=\min\{s,t\}$, and choose
$$\sigma_0:=\frac{\sigma+\tau}{2}.$$
Then $\sigma<\sigma_0<\tau$. We apply \Cref{prop: p-adic thin tubes} to the two measures $\bar\rho_1,\bar\rho_2$, with exponent $\sigma_0$. Since \eqref{eq:expansion-separated-supports} gives separation and \eqref{eq:expansion-product-tube} gives tube non-concentration, after decreasing $\varepsilon>0$ if necessary we obtain that both pairs $(\bar\rho_1,\bar\rho_2)$ and $(\bar\rho_2,\bar\rho_1)$ have $$(\sigma_0,\delta^{-O(\varepsilon)},1-\delta^{O(\varepsilon)})$$-thin tubes.

We shall use the thin-tube estimate for the ordered pair $(\bar\rho_2,\bar\rho_1)$, because we want to fix the base point $z_2$ and vary $z_1$. Let $\mathcal H\subset\Z_p^2\times\Z_p^2$ be the corresponding good set, so that
\begin{equation}\label{eq:expansion-thin-tube-good-set}
(\bar\rho_2\times\bar\rho_1)(\mathcal H)\ge 1-\delta^{O(\varepsilon)}
\end{equation}
and, for every $z_2\in\Z_p^2$, every $p$-adic $r$-tube $T$ containing $z_2$, and every $\delta\le r\le 1$,
\begin{equation}\label{eq:expansion-thin-tube-bound}
\bar\rho_1\{z_1\in T:(z_2,z_1)\in\mathcal H\}
\le\delta^{-O(\varepsilon)}r^{\sigma_0}.
\end{equation}

Let $\mathcal E^{\operatorname{op}}:=\{(z_2,z_1):(z_1,z_2)\in\mathcal E\}$. By \eqref{eq:expansion-normalized-bad-mass} and \eqref{eq:expansion-thin-tube-good-set}, the set $\mathcal E^{\operatorname{op}}\cap\mathcal H$ has positive $(\bar\rho_2\times\bar\rho_1)$-measure. Hence there exists $z_2=(a_2,b_2)\in\spt\bar\rho_2$ such that the set
$$
F:=\{z_1=(a_1,b_1)\in\spt\bar\rho_1:(z_1,z_2)\in\mathcal E\text{ and }(z_2,z_1)\in\mathcal H\}
$$
satisfies
\begin{equation}\label{eq:expansion-F-large}
\bar\rho_1(F)\gtrsim \delta^{O(\varepsilon)}.
\end{equation}

We now pass from thin tubes through $z_2$ to a Frostman measure on slopes. At this point we depart slightly from the proof of the real setting. Since the quotient $(b_1-b_2)/(a_1-a_2)$ need not belong to $\Z_p$, we split into two affine charts, namely
$$
F_{\hor}:=\{(a_1,b_1)\in F:|b_1-b_2|_p\le |a_1-a_2|_p\}
$$
and $F_{\ver}:=F\setminus F_{\hor}$. At least one of these two sets has $\bar\rho_1$-measure at least $\bar\rho_1(F)/2$. We first assume that
\begin{equation}\label{eq:expansion-horizontal-chart-large}
\bar\rho_1(F_{\hor})\ge \frac{1}{2}\bar\rho_1(F).
\end{equation}
Let $\lambda$ be the probability measure on $\Z_p$ defined by
$$
\lambda(J):=\frac{\bar\rho_1\left(\left\{z_1=(a_1,b_1)\in F_{\hor}:
\frac{b_1-b_2}{a_1-a_2}\in J\right\}\right)}
{\bar\rho_1(F_{\hor})}
$$
for Borel sets $J\subset\Z_p$. The quotient belongs to $\Z_p$ on $F_{\hor}$, and the denominator is nonzero by \eqref{eq:expansion-horizontal-chart-large}. Moreover, $\lambda$ is a $(\delta,\sigma_0,\delta^{-O(\varepsilon)})$-Frostman probability measure. Indeed, let $J\subset\Z_p$ be a $p$-adic ball of radius $r$, where $\delta\le r\le 1$. The set of points $z_1=(a_1,b_1)$ with $\frac{b_1-b_2}{a_1-a_2}\in J$ is contained in a $r$-tube through $z_2$. Therefore \eqref{eq:expansion-thin-tube-bound} and \eqref{eq:expansion-F-large} imply
$$
\lambda(J)\lesssim \delta^{-O(\varepsilon)}r^{\sigma_0}.
$$

We apply \Cref{prop: p-adic Kaufman projection} to the set $P$ and to the measure $\lambda$, with $s$ in \Cref{prop: p-adic Kaufman projection} replaced by $\sigma_0$ and density parameter $\alpha=\delta^{\varepsilon}$. Since $P$ is a $(\delta,s+t,\delta^{-\varepsilon})$-set and $\sigma_0<s+t$, it is also a $(\delta,\sigma_0,\delta^{-O(\varepsilon)})$-set. Therefore there exists $c\in\spt\lambda$ such that every $X\subset P$ with $|X|\ge\delta^{\varepsilon}|P|$ satisfies
\begin{equation}\label{eq:expansion-kaufman-lower}
|\{b+ca:(a,b)\in X\}|_{\delta}
\ge
\delta^{-\sigma_0+O(\varepsilon)}.
\end{equation}
Since $c\in\spt\lambda$, there exists $z_1=(a_1,b_1)\in F_{\hor}$ with $c=(b_1-b_2)/(a_1-a_2)$. Since $z_1\in F$, the quadruple $((a_1,b_1),(a_2,b_2))$ belongs to $\mathcal E$. Hence, by the definition of $\mathcal E$, there exists $X\subset P$ with $|X|\ge\delta^{\varepsilon}|P|$ such that \eqref{eq:expansion-bad-small-projection} holds.

On the other hand, applying \eqref{eq:expansion-kaufman-lower} to this same set $X$ gives
$$
|\{b+ca:(a,b)\in X\}|_{\delta}
\ge
\delta^{-\sigma_0+O(\varepsilon)}.
$$
Multiplying by $a_1-a_2$ transforms the projected set into
$$
\{(a_1-a_2)b+(b_1-b_2)a:(a,b)\in X\}.
$$
By \eqref{eq:expansion-separated-supports}, multiplication by $a_1-a_2$ can decrease $\delta$-covering numbers by at most a factor $\delta^{-O(\varepsilon)}$. Thus
\begin{equation}\label{eq:expansion-contradiction-horizontal}
\left|\{(b_1-b_2)a+(a_1-a_2)b:(a,b)\in X\}\right|_{\delta}
\ge
\delta^{-\sigma_0+O(\varepsilon)}.
\end{equation}
Choosing $\varepsilon>0$ sufficiently small in terms of $s,t,\sigma$, we have $\sigma_0-O(\varepsilon)>\sigma$. This contradicts \eqref{eq:expansion-bad-small-projection}.

It remains to discuss the case $\bar\rho_1(F_{\ver})\ge \frac{1}{2}\bar\rho_1(F)$. In this case, we use the other affine chart. Define a probability measure $\lambda'$ on $\Z_p$ by
$$
\lambda'(J):=
\frac{\bar\rho_1\left(\left\{z_1=(a_1,b_1)\in F_{\ver}:
\frac{a_1-a_2}{b_1-b_2}\in J\right\}\right)}
{\bar\rho_1(F_{\ver})}.
$$
The same thin-tube estimate shows that $\lambda'$ is a $(\delta,\sigma_0,\delta^{-O(\varepsilon)})$-Frostman probability measure on $\Z_p$. Applying \Cref{prop: p-adic Kaufman projection} to the coordinate-swapped set $\{(b,a):(a,b)\in P\}$ gives a number $c'\in\spt\lambda'$ such that every $X\subset P$ with $|X|\ge\delta^{\varepsilon}|P|$ satisfies
$$
|\{a+c'b:(a,b)\in X\}|_{\delta}
\ge
\delta^{-\sigma_0+O(\varepsilon)}.
$$
Choose $z_1=(a_1,b_1)\in F_{\ver}$ with $c'=(a_1-a_2)/(b_1-b_2)$. Multiplying the last projected set by $b_1-b_2$, and using that $|b_1-b_2|_p\ge |a_1-a_2|_p\ge \rho\gtrsim \delta^{O(\varepsilon)}$, gives the same lower bound
$$
\left|\{(b_1-b_2)a+(a_1-a_2)b:(a,b)\in X\}\right|_{\delta}
\ge
\delta^{-\sigma_0+O(\varepsilon)}.
$$
This again contradicts \eqref{eq:expansion-bad-small-projection}.

We conclude that $(\rho_1\times\rho_2)(\mathcal E)\le \delta^{\varepsilon}$, after decreasing $\varepsilon>0$ if necessary. Taking $\mathcal G$ to be the complement of $\mathcal E$ in $A_1\times A_2\times B_1\times B_2$ gives
$$
|(A_1\times A_2\times B_1\times B_2)\setminus\mathcal G|
\le
\delta^{\varepsilon}|A_1||A_2||B_1||B_2|.
$$
This proves the proposition.
\end{proof}

\begin{rmk}\label{rem: expansion asymmetric}
The proof of \Cref{prop: discretised expansion} in fact gives the following
slightly more flexible statement. If $A_i$ are
$(\delta,s_i,\delta^{-\varepsilon})$-sets and $B_i$ are $(\delta,t_i,\delta^{-\varepsilon})$-sets, then the same conclusion holds, with $\varepsilon$ depending on all parameters, provided
$$\sigma<\min\{s_1+t_1,s_2+t_2,s+t,1\}.$$
Indeed, the only change is that $\bar\rho_i=\bar\mu_i\times\nu_i$ becomes a
$(\delta,s_i+t_i,\delta^{-O(\varepsilon)})$-Frostman measure, while the projection
theorem is applied to $P$, which remains a $(\delta,s+t,\delta^{-\varepsilon})$-set.
Thus the intermediate exponent $\sigma_0$ only needs to be chosen strictly below
all four relevant thresholds.
\end{rmk}

\subsection{A difference-set expansion lemma}

The following lemma is the $p$-adic analogue of the key difference-expansion step
in \cite[Section 3.3]{orponen2023projections}. It shows that, for most pairs $(b_1,b_2)\in B\times B$, the difference multiplier $b_1-b_2$ expands every
large product subset of $A\times C$.

\begin{lemma}\label{lem: difference expansion}
Let $0<\beta\le \alpha<1$, and let $\gamma\in(\alpha-\beta,1]$. Let
$A,B,C$ be as in \Cref{thm: weak ABC sum-product}. Then, for a sufficiently
small $\chi>0$ in terms of $\alpha,\beta,\gamma$, there exists an exceptional set $E\subset B\times B$ with $|E|\le \delta^\chi |B|^2$, such that for every $(b_1,b_2)\in (B\times B)\setminus E$ and every pair of subsets $A'\subset A$, $C'\subset C$ satisfying
$$
|A'||C'|\ge \delta^\chi |A||C|,
$$
we have
$$
|A'+(b_1-b_2)C'|_\delta\ge\delta^{-(\beta+\gamma-\alpha)/4}|A|.
$$
\end{lemma}

No unit condition is imposed on the differences $b_1-b_2$ in this lemma. This is
consistent with \Cref{thm: asymmetric BSG}, whose multiplier is arbitrary.

\begin{proof}[Proof of \Cref{lem: difference expansion}]
Let $\theta:=\frac{\beta+\gamma-\alpha}{4}>0$. We choose the parameter $\chi>0$ sufficiently small in terms of $\alpha,\beta,\gamma$ and
also sufficiently small compared with the auxiliary parameter in
\Cref{prop: discretised expansion} and \Cref{rem: expansion asymmetric}. All losses of the
form $\delta^{-O(\chi)}$ below will be absorbed into the constants allowed there.

Suppose that the conclusion fails. Then there exists a set
$$
P\subset B\times B,\qquad |P|\ge \delta^\chi |B|^2,
$$
such that for every $(b_1,b_2)\in P$ one can find subsets $A_{b_1,b_2}\subset A,C_{b_1,b_2}\subset C$ satisfying
\begin{equation}\label{eq: bad-pair-large-subsets}
|A_{b_1,b_2}||C_{b_1,b_2}|\ge \delta^\chi |A||C|
\end{equation}
and
\begin{equation}\label{eq: bad-pair-small-sumset}
|A_{b_1,b_2}+(b_1-b_2)C_{b_1,b_2}|_\delta
<\delta^{-\theta}|A|.
\end{equation}

We first record the non-concentration inherited by the bad-pair set $P$. Since
$B$ is a $(\delta,\beta,\delta^{-\chi})$-set, the product set $B\times B$ is a
$(\delta,2\beta,\delta^{-O(\chi)})$-set in $\mathbb Z_p^2$. Indeed, for every
$r\in[\delta,1]$ and every $r$-ball $Q\subset \mathbb Z_p^2$, one has
$$
|(B\times B)\cap Q|
\le
\delta^{-O(\chi)}r^{2\beta}|B|^2.
$$
Since $P\subset B\times B$ and $|P|\ge \delta^\chi |B|^2$, it follows that
$$
|P\cap Q|
\le
\delta^{-O(\chi)}r^{2\beta}|P|,
\qquad r\in[\delta,1].
$$
Thus $P$ itself is a $(\delta,2\beta,\delta^{-O(\chi)})$-set. Equivalently, if desired, one may
first pass to a subset $P'\subset P$ with $|P'|\ge \delta^{O(\chi)}|P|$ and the same
$(\delta,2\beta,\delta^{-O(\chi)})$ non-concentration property; this only changes the powers
of $\delta^{O(\chi)}$ and will not affect the argument. We keep the notation $P$.

Similarly, $C$ remains a $(\delta,\gamma,\delta^{-\chi})$-set, hence also a
$(\delta,\gamma,\delta^{-O(\chi)})$-set. Therefore the pair consisting of the bad-pair set
$P\subset B\times B$ and the coefficient set $C$ satisfies the Frostman hypotheses needed for
\Cref{prop: discretised expansion}, in the asymmetric form of
\Cref{rem: expansion asymmetric}, after replacing the small parameter there by a fixed multiple
of $\chi$.

We apply \Cref{prop: discretised expansion} with target exponent $\sigma=\alpha+\theta$.
This choice is admissible because
$$
\alpha+\theta=\alpha+\frac{\beta+\gamma-\alpha}{4}<\beta+\gamma,
$$
and $\theta$ was chosen strictly smaller than the available gap $\beta+\gamma-\alpha$. Thus,
after taking $\chi$ sufficiently small, the asymmetric expansion statement gives a pair
$(b_1,b_2)\in P$ such that for every subset $X\subset A\times C$ with $|X|\ge \delta^{O(\chi)}|A||C|$, one has
\begin{equation}\label{eq: good-difference-expansion-from-prop}
|\{a+(b_1-b_2)c:(a,c)\in X\}|_\delta
\ge\delta^{-\alpha-\theta}.
\end{equation}
Here no unit condition is imposed on $b_1-b_2$; this is exactly why we use the multiplier form of the discretised expansion proposition.

Now take $X=A_{b_1,b_2}\times C_{b_1,b_2}$. By \eqref{eq: bad-pair-large-subsets}, this set is large enough for
\eqref{eq: good-difference-expansion-from-prop}, provided $\chi$ was chosen smaller than the
auxiliary expansion parameter. Hence
$$
|A_{b_1,b_2}+(b_1-b_2)C_{b_1,b_2}|_\delta
\ge
\delta^{-\alpha-\theta}.
$$
Since $|A|\le \delta^{-\alpha}$, we have
$\delta^{-\alpha-\theta}\ge \delta^{-\theta}|A|$. Therefore
$$
|A_{b_1,b_2}+(b_1-b_2)C_{b_1,b_2}|_\delta
\ge
\delta^{-\theta}|A|,
$$
which contradicts \eqref{eq: bad-pair-small-sumset}. This contradiction proves the lemma.
\end{proof}

\subsection{Proof of the weak ABC theorem}

\begin{proof}[Proof of \Cref{thm: weak ABC sum-product}]
We prove the theorem in the range $0<\beta\le \alpha<1$. Put
$$
\chi_0:=\frac{\beta+\gamma-\alpha}{4}>0.
$$
Let $\zeta=\zeta(\chi_0/2)>0$ be the constant supplied by
\Cref{thm: asymmetric BSG} with parameter $\eta=\chi_0/2$. We choose
$$
0<\chi\ll \min\{\chi_0,\zeta\},
$$
and then choose $\delta_0>0$ sufficiently small depending on
$\alpha,\beta,\gamma,\chi_0,\zeta,\chi$. Throughout the proof, the implicit
constants in $O(\chi)$ depend only on $\alpha,\beta,\gamma$.

Suppose, towards a contradiction, that
\begin{equation}\label{eq:weak-ABC-counterassumption}
|A+cB|_{\delta}<\delta^{-\chi}|A|,\qquad c\in C.
\end{equation}
Since $\chi\ll \chi_0$, we may apply \Cref{lem: difference expansion} with
$\chi_0$ in place of the small parameter. Thus there exists an exceptional set
$E\subset B\times B$ with
\begin{equation}\label{eq:weak-ABC-exceptional-set}
|E|\le \delta^{\chi_0}|B|^2,
\end{equation}
such that, for every $(b_1,b_2)\in (B\times B)\setminus E$ and every pair of
subsets $A'\subset A$, $C'\subset C$ satisfying
\begin{equation}\label{eq:weak-ABC-large-subsets-condition}
|A'||C'|\ge \delta^{\chi_0}|A||C|,
\end{equation}
one has
\begin{equation}\label{eq:weak-ABC-difference-expansion}
|A'+(b_2-b_1)C'|_{\delta}\ge \delta^{-\chi_0}|A|.
\end{equation}

We next use the counterassumption to obtain large additive energy. This is the
only point in the proof where the unit condition on the coefficient set $C$ is
used. Since $C\subset(\Z/p^n\Z)^\times$, multiplication by every $c\in C$ is a
bijection on $\Z/p^n\Z$. Hence $|cB|=|B|$. Applying \Cref{lem: energy restricted sumset} to the pair $(A,cB)$ and to the
full graph $A\times cB$, the estimate \eqref{eq:weak-ABC-counterassumption}
gives
\begin{equation}\label{eq:weak-ABC-energy-for-fixed-c}
E_{\delta}(A,cB)\ge \delta^{O(\chi)}|A||B|^2,\qquad c\in C.
\end{equation}
Equivalently, for every $c\in C$,
$$
|\{(a,a',b_1,b_2)\in A^2\times B^2:a+cb_1=a'+cb_2\}|
\ge \delta^{O(\chi)}|A||B|^2.
$$
Summing this estimate over $c\in C$ and rearranging the equation $a+cb_1=a'+cb_2$ as
$a=a'+c(b_2-b_1)$, we obtain 
\begin{equation}\label{eq:weak-ABC-rearranged-energy}
\sum_{b_1,b_2\in B}
|\{(a',c)\in A\times C:a'+c(b_2-b_1)\in A\}|
\ge
\delta^{O(\chi)}|A||B|^2|C|.
\end{equation}

The contribution of the exceptional pairs in $E$ to the left hand side of
\eqref{eq:weak-ABC-rearranged-energy} is at most
$$
|E||A||C|\le \delta^{\chi_0}|A||B|^2|C|.
$$
Since $\chi\ll\chi_0$, this is negligible compared with the lower bound in
\eqref{eq:weak-ABC-rearranged-energy}, after decreasing $\delta_0$ if necessary.
Therefore
\begin{equation}\label{eq:weak-ABC-nonexceptional-energy}
\sum_{(b_1,b_2)\in (B\times B)\setminus E}
|\{(a',c)\in A\times C:a'+c(b_2-b_1)\in A\}|
\ge\delta^{O(\chi)}|A||B|^2|C|.
\end{equation}
Consequently, there exists a pair
$(b_1,b_2)\in (B\times B)\setminus E$ such that, writing $d:=b_2-b_1$, the set
$$
G:=\{(a',c)\in A\times C:a'+dc\in A\}
$$
satisfies
\begin{equation}\label{eq:weak-ABC-large-graph}
|G|\ge \delta^{O(\chi)}|A||C|.
\end{equation}
Moreover,
\begin{equation}\label{eq:weak-ABC-small-restricted-sumset}
|\{a'+dc:(a',c)\in G\}|_{\delta}\le |A|.
\end{equation}

We now apply \Cref{thm: asymmetric BSG} with multiplier $d=b_2-b_1$ and parameter $\eta=\chi_0/2$. Notice that $d$ is not required to be a unit; this is
precisely why \Cref{thm: asymmetric BSG} was stated in multiplier form. Since
$\chi\ll\zeta$, \eqref{eq:weak-ABC-large-graph} implies $|G|\ge \delta^\zeta |A||C|$
for $\delta$ sufficiently small, and
\eqref{eq:weak-ABC-small-restricted-sumset} trivially implies
$$
|\{a'+dc:(a',c)\in G\}|_{\delta}\le \delta^{-\zeta}|A|.
$$
Thus \Cref{thm: asymmetric BSG} yields subsets $A'\subset A$ and
$C'\subset C$ such that
\begin{equation}\label{eq:weak-ABC-BSG-large-subsets}
|A'||C'|\ge \delta^{\chi_0/2}|A||C|
\end{equation}
and
\begin{equation}\label{eq:weak-ABC-BSG-small-sumset}
|A'+dC'|_{\delta}\le \delta^{-\chi_0/2}|A|.
\end{equation}

Since $\delta^{\chi_0/2}\ge \delta^{\chi_0}$, the pair $A',C'$ satisfies
\eqref{eq:weak-ABC-large-subsets-condition}. Since
$(b_1,b_2)\notin E$, the difference expansion estimate
\eqref{eq:weak-ABC-difference-expansion} gives $|A'+dC'|_{\delta}\ge \delta^{-\chi_0}|A|$. This contradicts \eqref{eq:weak-ABC-BSG-small-sumset}, because
$\delta^{-\chi_0/2}<\delta^{-\chi_0}$. The contradiction proves that there exists $c\in C$ such that $|A+cB|_{\delta}\ge \delta^{-\chi}|A|$. This proves the weak ABC theorem.

Finally, the uniformity statement follows from the same choice of parameters. If
$(\alpha,\beta,\gamma)$ ranges in a compact subset of the admissible region
$$
\{(\alpha,\beta,\gamma):0<\beta\le \alpha<1,\ \gamma>\alpha-\beta\},
$$
then $\chi_0=(\beta+\gamma-\alpha)/4$ is bounded from below by a positive
constant. Consequently $\zeta(\chi_0/2)$ is also bounded from below, and the
choices of $\chi$ and $\delta_0$ can be made uniformly on the compact set.
\end{proof}

\section{Projection of regular sets}
\label{sec: projections}

In this section, we prove the projection estimate for almost regular sets. The proof follows the strategy of Orponen-Shmerkin \cite[Section 4]{orponen2023projections}, with Euclidean dyadic cubes and orthogonal projections replaced by $p$-adic balls and the projection map defined in \Cref{defi: projection map}. 

\begin{thm}
\label{thm: regular set projection}
Let $s\in(0,1)$, $t\in(s,2-s)$, and $u\in[0,(s+t)/2)$. Then there exist
$$
\epsilon=\epsilon(s,t,u)>0,\qquad\delta_0=\delta_0(s,t,u)>0
$$
such that the following holds for all $\delta=p^{-n}\in(0,\delta_0]$. Let
$\mathcal P\subset\mathcal D_\delta(\Z_p^2)$ be a $(\delta,t,\delta^{-\epsilon})$-regular set, and let $E\subset\Z_p$ be a $(\delta,s,\delta^{-\epsilon})$-set. Then there exists $\theta\in E$ such that
\begin{equation}
\label{eq: regular projection conclusion}
|\pi_\theta(\mathcal P')|_\delta\ge \delta^{-u}
\end{equation}
for every subset $\mathcal P'\subset\mathcal P$ satisfying $|\mathcal P'|\ge\delta^\epsilon|\mathcal P|$.
\end{thm}

\begin{rmk}\label{rmk: many good directions}
As in \cite[Remark 4.11]{orponen2023projections}, \Cref{thm: regular set projection} also yields the following many-directions version, after decreasing $\epsilon>0$. There exists a subset $E'\subset E$ with $|E'|\ge |E|/2$ such that, for every $\theta\in E'$ and every $P'\subset P$ with $|P'|\ge \delta^\epsilon |P|$, one has
$$
|\pi_\theta(P')|_\delta\ge \delta^{-u}.
$$
Indeed, if no such subset $E'$ existed, then the exceptional set $E_{\mathrm{bad}}\subset E$ would satisfy $|E_{\mathrm{bad}}|\ge |E|/2$. Since $E$ is a $(\delta,s,\delta^{-\epsilon})$-set, the subset $E_{\mathrm{bad}}$ is a $(\delta,s,2\delta^{-\epsilon})$-set, hence a $(\delta,s,\delta^{-2\epsilon})$-set
for $\delta>0$ small enough. Taking $\epsilon$ sufficiently small, we may apply
\Cref{thm: regular set projection} to $P$ and $E_{\mathrm{bad}}$, which gives a
direction in $E_{\mathrm{bad}}$ satisfying the desired projection estimate for all
large subsets of $P$, a contradiction.
\end{rmk}

\subsection{Regular measures and high multiplicity}
\label{subsec: regular measures and high multiplicity}

We first introduce the measure-theoretic language used in the proof. This is slightly more flexible than working only with regular sets, since the argument repeatedly rescales balls.

\begin{defi}
\label{defi: regular measure}
Let $t>0$ and $C\ge 1$. A compactly supported Borel probability measure $\mu$
on $\Q_p^d$ is called $(t,C)$-regular if, writing $K:=\spt\mu$, one has
\begin{equation}
\label{eq: regular measure two-sided mass}
C^{-1}r^t\le \mu(B(x,r))\le Cr^t
\end{equation}
for every $x\in K$ and every $0<r\le \diam(K)$.

If $\delta\in p^{-\N}$, we say that $\mu$ is $(\delta,t,C)$-regular if
\eqref{eq: regular measure two-sided mass} is required only for radii
$r\in p^{-\N}$ satisfying
$$
\delta\le r\le \diam(K),
$$
and if, in addition, $K=\spt\mu$ is a union of $\delta$-balls.
\end{defi}



\begin{defi}
\label{defi: multiplicity number}
Let $K\subset\Q_p^2$, let $0<r\le R\le\infty$ be powers of $p$, and let $x\in K$. For $\theta\in\Z_p$, define the multiplicity number
$$
\mathfrak m_{K,\theta}(x\mid[r,R]):=
|B(x,R)\cap K_r\cap\pi_\theta^{-1}(\pi_\theta(x))|_r.
$$
Here $K_r$ denotes the $r$-neighbourhood of $K$, and the right-hand side is the $r$-covering number of the intersection. When the set $K$ is clear from context, we abbreviate $\mathfrak m_{K,\theta}=:\mathfrak m_\theta$.
\end{defi}

\begin{defi}
\label{defi: high multiplicity set}
Let $0<r\le R\le\infty$ be powers of $p$, let $M>0$, and let $\theta\in\Z_p$. For $K\subset\Q_p^2$, define the high multiplicity set
$$
H_\theta(K,M,[r,R]):=\{x\in K:\mathfrak m_{K,\theta}(x\mid[r,R])\ge M\}.
$$
\end{defi}

For $z_0\in\Q_p^2$ and $r_0\in p^\Z$, write
$T_{z_0,r_0}(z):=\frac{z-z_0}{r_0}$ for the homothety mapping $B(z_0,r_0)$ onto $\Z_p^2$.

\begin{lemma}
\label{lem: high multiplicity scaling}
Let $K\subset\Q_p^2$ be arbitrary, let $0<r\le R\le\infty$ be powers of $p$, let $M>0$, and let $\theta\in\Z_p$. Then, for every $z_0\in\Q_p^2$ and $r_0\in p^\Z$,
\begin{equation}
\label{eq: high multiplicity scaling}
T_{z_0,r_0}\left(H_\theta(K,M,[r,R])\cap B(z_0,r_0)\right)\subset H_\theta\left(
T_{z_0,r_0}(K),M,\left[\frac r{r_0},\frac R{r_0}\right]\right).
\end{equation}
Conversely, if the right-hand side is intersected with $\Z_p^2=T_{z_0,r_0}(B(z_0,r_0))$, then the reverse inclusion also holds.
\end{lemma}

\begin{proof}
The homothety $T_{z_0,r_0}$ maps $r$-balls to $r/r_0$-balls and $R$-balls to $R/r_0$-balls. Moreover, for $x=(x_1,x_2)$ and $w=(w_1,w_2)$, the condition $\pi_\theta(w)=\pi_\theta(x)$ is equivalent to
$$
\pi_\theta(T_{z_0,r_0}(w))=\pi_\theta(T_{z_0,r_0}(x)).
$$
Thus the relevant fibre intersections are mapped bijectively to the corresponding rescaled fibre intersections, and their covering numbers are unchanged. 
\end{proof}

The following theorem is the main multiplicity estimate behind \Cref{thm: regular set projection}.

\begin{thm}
\label{thm: exceptional measure estimate}
Let $s\in(0,1)$, $t\in(s,2-s)$, and $\sigma>\frac{t-s}{2}$. Then there exist
$$
\epsilon=\epsilon(s,t,\sigma)>0,\qquad\delta_0=\delta_0(s,t,\sigma)>0
$$
such that the following holds for all $\delta=p^{-n}\in(0,\delta_0]$. Let $\mu$ be a $(\delta,t,\delta^{-\epsilon})$-regular measure on $\Q_p^2$, with $K:=\spt\mu$, and let $E\subset\Z_p$ be a $(\delta,s,\delta^{-\epsilon})$-set. Then there exists $\theta\in E$ such that
\begin{equation}
\label{eq: existence of direction with small multiplicity}
\mu\bigl(\Z_p^2\cap H_\theta(K,\delta^{-\sigma},[\delta,1])\bigr)\le\delta^\epsilon.
\end{equation}
\end{thm}

We next show that this multiplicity estimate implies the projection theorem.

\begin{proof}[Proof of \Cref{thm: regular set projection}, assuming \Cref{thm: exceptional measure estimate}]
Fix $s\in(0,1)$, $t\in(s,2-s)$, and $u<(s+t)/2$. Choose numbers $\sigma'>\sigma>\frac{t-s}{2}$ so that $u<t-\sigma'$. Let $\epsilon>0$ be sufficiently small compared to the constant supplied by \Cref{thm: exceptional measure estimate} with parameter $\sigma$, and let $\delta>0$ be sufficiently small.

Let $\mathcal P\subset\mathcal D_\delta(\Z_p^2)$ be a $(\delta,t,\delta^{-\epsilon})$-regular set. Define $
P:=\bigcup_{\mathscr p\in\mathcal P}\mathscr p$, and let $\mu$ be the probability measure obtained by normalizing Haar measure on $P$. Then $\mu$ is a $(\delta,t,\delta^{-O(\epsilon)})$-regular measure. Applying \Cref{thm: exceptional measure estimate}, and decreasing $\epsilon$ if necessary, we find $\theta\in E$ such that
\begin{equation}
\label{eq: measure high multiplicity small}
\mu\bigl(H_\theta(P,\delta^{-\sigma},[\delta,1])\bigr)\le\delta^{O(\epsilon)}.
\end{equation}
Equivalently,
\begin{equation}
\label{eq: cube high multiplicity small}
\left|\{\mathscr p\in\mathcal P:\mathscr p\cap H_\theta(P,\delta^{-\sigma},[\delta,1])\neq\varnothing\}\right|\le\delta^{O(\epsilon)}|\mathcal P|.
\end{equation}

Let $\mathcal P'\subset\mathcal P$ satisfy $
|\mathcal P'|\ge\delta^\epsilon|\mathcal P|$. Removing from $\mathcal P'$ the cubes which meet the high multiplicity set, we obtain a subcollection $\mathcal G\subset\mathcal P'$ with
\begin{equation}
\label{eq: good P prime large}
|\mathcal G|\gtrsim \delta^\epsilon|\mathcal P|
\end{equation}
and such that for every $\mathscr p\in\mathcal G$ and every $x\in \mathscr p$,
\begin{equation}
\label{eq: good fibre multiplicity bound}
\mathfrak m_{P,\theta}(x\mid[\delta,1])\le \delta^{-\sigma}.
\end{equation}
This implies that each $\pi_\theta$-fibre intersects at most $\lesssim\delta^{-\sigma}$ cubes in $\mathcal G$. Hence
\begin{equation}
\label{eq: projection lower by multiplicity}
|\pi_\theta(\mathcal P')|_\delta\ge|\pi_\theta(\mathcal G)|_\delta\gtrsim\delta^\sigma|\mathcal G|.
\end{equation}
Since $\mathcal P$ is $(\delta,t,\delta^{-\epsilon})$-regular, we have $|\mathcal P|\gtrsim \delta^{-t+O(\epsilon)}$. Combining this with \eqref{eq: good P prime large} and \eqref{eq: projection lower by multiplicity}, we obtain
$$
|\pi_\theta(\mathcal P')|_\delta\gtrsim\delta^{-t+\sigma+O(\epsilon)}.
$$
By choosing $\epsilon>0$ sufficiently small, one can ensure $t-\sigma-O(\epsilon)>u$. Therefore
$$
|\pi_\theta(\mathcal P')|_\delta\ge \delta^{-u}
$$
for all sufficiently small $\delta$, as required.
\end{proof}

\subsection{An inductive scheme for the exceptional estimate}
\label{subsec: inductive scheme projection}

We now set up the inductive scheme used to prove \Cref{thm: exceptional measure estimate}. The scheme is phrased in terms of the following statement.

\begin{defi}
\label{defi: Proj statement}
Let $s,\sigma\in(0,1)$ and $t\in(0,2)$. We say that $\Proj(s,\sigma,t)$
holds if the conclusion of \Cref{thm: exceptional measure estimate} holds with the parameters $s,\sigma,t$. More explicitly, this means that there exist constants
$$
\epsilon=\epsilon(s,\sigma,t)>0,\qquad\delta_0=\delta_0(s,\sigma,t)>0,
$$
such that whenever $\delta=p^{-n}\in(0,\delta_0]$, $\mu$ is a
$(\delta,t,\delta^{-\epsilon})$-regular measure on $\Q_p^2$, with $K:=\spt\mu$, and
$E\subset\Z_p$ is a $(\delta,s,\delta^{-\epsilon})$-set, there exists $\theta\in E$ such that
\begin{equation}
\label{eq: Proj conclusion}
\mu\bigl(\Z_p^2\cap H_\theta(K,\delta^{-\sigma},[\delta,1])\bigr)\le\delta^\epsilon.
\end{equation}
\end{defi}

If $\sigma'\ge \sigma$, then it is clear that
$$
\Proj(s,\sigma,t)\Longrightarrow \Proj(s,\sigma',t).
$$
Indeed, $H_\theta(K,\delta^{-\sigma'},[\delta,1])\subset H_\theta(K,\delta^{-\sigma},[\delta,1])$ since $\delta^{-\sigma'}\ge \delta^{-\sigma}$ for $\delta\in(0,1)$.

The proof of \Cref{thm: exceptional measure estimate} will be based on the following iteration step.

\begin{prop}
\label{prop: iterating projection}
Let $s\in(0,1)$, $t\in(s,2-s)$, and $\frac{t-s}{2}<\sigma<\frac t2$. Then there exists
$\zeta=\zeta(s,\sigma,t)>0$ such that
$$
\Proj(s,\sigma,t)\Longrightarrow \Proj(s,\sigma-\zeta,t).
$$
Moreover, $\zeta(s,\sigma,t)$ stays bounded away from $0$ as long as $(s,\sigma,t)$ ranges over a compact subset of the parameter region
$$
\left\{(s,\sigma,t):\ s\in(0,1),\ t\in(s,2-s),\ \frac{t-s}{2}<\sigma<\frac t2\right\}.
$$
\end{prop}

For later use, we also record the more quantitative form of the implication in
\Cref{prop: iterating projection}. Suppose that $\Proj(s,\sigma,t)$ holds with constants
$\epsilon_0,\Delta_0>0$. Thus, whenever $\Delta\in(0,\Delta_0]$, $\mu$ is a
$(\Delta,t,\Delta^{-\epsilon_0})$-regular measure and $E\subset\Z_p$ is a
$(\Delta,s,\Delta^{-\epsilon_0})$-set, there exists $\theta\in E$ satisfying
$$
\mu\bigl(\Z_p^2\cap H_\theta(\spt\mu,\Delta^{-\sigma},[\Delta,1])\bigr)\le\Delta^{\epsilon_0}.
$$
Then \Cref{prop: iterating projection} asserts that there exist constants
$$
\epsilon=\epsilon(\epsilon_0,s,\sigma,t)>0,
\qquad
\delta_0=\delta_0(\epsilon,\Delta_0,s,\sigma,t)>0
$$
such that, whenever $\delta\in(0,\delta_0]$, $\mu$ is a
$(\delta,t,\delta^{-\epsilon})$-regular measure and $E\subset\Z_p$ is a
$(\delta,s,\delta^{-\epsilon})$-set, there exists $\theta\in E$ satisfying
\begin{equation}
\label{eq: iterated Proj conclusion}
\mu\bigl(\Z_p^2\cap H_\theta(\spt\mu,\delta^{-(\sigma-\zeta)},[\delta,1])\bigr)\le\delta^\epsilon.
\end{equation}

We next record the base case of the iteration.

\begin{prop}
\label{prop: base case of projection}
Let $s\in(0,1)$ and $t\in(s,2-s)$. Then there exists $\eta=\eta(s,t)>0$ such that
$\Proj\left(s,\frac t2-\eta,t\right)
$ holds.
\end{prop}

\begin{proof} 
Choose $\eta:=\eta_{\mathrm{B}}/10$, and then choose $\varepsilon>0$ sufficiently small in terms of $s,t,\eta$. We will show that Projection$(s,t/2-\eta,t)$ holds with this choice of $\eta$.

Let $\delta\in(0,\delta_0]$, let $\mu$ be a $(\delta,t,\delta^{-\varepsilon})$-regular measure, let $K:=\spt\mu$, and let $E\subset\Z_p$ be a $(\delta,s,\delta^{-\varepsilon})$-set. We need to find $\theta\in E$ such that
$$
\mu\left(\Z_p^2\cap H_{\theta}\left(K,\delta^{-t/2+\eta},[\delta,1]\right)\right)\le\delta^{\varepsilon}.
$$
Suppose, to the contrary, that for every $\theta\in E$ one has
\begin{equation}\label{eq:base-bad-for-all-theta}
\mu(B_{\theta})>\delta^{\varepsilon},
\qquad B_{\theta}:=\Z_p^2\cap H_{\theta}\left(K,\delta^{-t/2+\eta},[\delta,1]\right).
\end{equation}

Let $P:=\mathcal D_{\delta}(K\cap\Z_p^2)$. By the $(\delta,t,\delta^{-\varepsilon})$-regularity of $\mu$, after increasing the implicit constants and choosing $\varepsilon>0$ sufficiently
small, the set $P$ is a $(\delta,t,\delta^{-\varepsilon_{\mathrm{B}}})$-set and
\begin{equation}\label{eq:base-P-size}
|P|\le \delta^{-t-O(\varepsilon)}.
\end{equation}
Moreover, \eqref{eq:base-bad-for-all-theta} implies $\mu(K\cap\Z_p^2)>\delta^{\varepsilon}$,
so the regularity of $\mu$ also gives
\begin{equation}\label{eq:base-P-lower-size}
|P|\ge \delta^{-t+O(\varepsilon)}.
\end{equation}
In particular, for $\delta_0$ and $\varepsilon$ sufficiently small, $P$ satisfies the size hypothesis in \Cref{lem: padic Bourgain projection}.

For each $\theta\in E$, let $P_{\theta}:=\mathcal D_{\delta}(B_{\theta})$ be the collection of $\delta$-cubes meeting $B_{\theta}$. Since every $\delta$-cube has $\mu$-mass at most $\delta^{t-O(\varepsilon)}$, \eqref{eq:base-bad-for-all-theta} gives
\begin{equation}\label{eq:base-P-theta-large}
|P_{\theta}|\ge\delta^{-t+O(\varepsilon)}\ge
\delta^{O(\varepsilon)}|P|.
\end{equation}
After decreasing $\varepsilon>0$ if necessary, this implies $|P_{\theta}|\ge \delta^{\varepsilon_{\mathrm{B}}}|P|$ for every $\theta\in E$.

We next show that every $P_{\theta}$ has small projection. Fix $\theta\in E$. If
$q\in\pi_{\theta}(P_{\theta})$ is a $\delta$-ball in $\Z_p$, then there exists
$x\in B_{\theta}$ whose $\pi_{\theta}$-value lies in $q$. Since
$x\in H_{\theta}(K,\delta^{-t/2+\eta},[\delta,1])$, the $\pi_{\theta}$-fibre through $x$, at scale $\delta$, meets at least $\delta^{-t/2+\eta}$ many $\delta$-cubes of $P$, up to an absolute constant. Therefore each $\delta$-ball in
$\pi_{\theta}(P_{\theta})$ accounts for at least $\gtrsim \delta^{-t/2+\eta}$ cubes of
$P$. Distinct $\delta$-balls in the projection correspond to disjoint families of
$\delta$-cubes. Hence
\begin{equation}\label{eq:base-small-projection}
|\pi_{\theta}(P_{\theta})|_{\delta}
\lesssim\delta^{t/2-\eta}|P|\le\delta^{-t/2-2\eta},
\end{equation}
where the last inequality follows from \eqref{eq:base-P-size} and from the choice
$\varepsilon\ll\eta$.

Now apply \Cref{lem: padic Bourgain projection} to the set $P$ and the direction set $E$. We obtain a direction $\theta\in E$ such that every subset $P'\subset P$ with
$|P'|\ge\delta^{\varepsilon_{\mathrm{B}}}|P|$ satisfies 
$$|\pi_\theta(P')|_\delta\ge \delta^{-t/2-\eta_B}.$$
Applying this with $P'=P_{\theta}$, which is allowed by \eqref{eq:base-P-theta-large}, gives
$$
|\pi_{\theta}(P_{\theta})|_{\delta}\ge
\delta^{-t/2-\eta_{\mathrm{B}}}.
$$
This contradicts \eqref{eq:base-small-projection}, since $\eta_{\mathrm{B}}=10\eta$ and $\delta$ is sufficiently small.

The contradiction shows that there exists $\theta\in E$ such that
$$
\mu\left(\Z_p^2\cap H_{\theta}\left(K,\delta^{-t/2+\eta},[\delta,1]\right)\right)\le\delta^{\varepsilon}.
$$
This is precisely Projection$(s,t/2-\eta,t)$.
\end{proof}

We now show that the base case and the iteration step imply
\Cref{thm: exceptional measure estimate}.

\begin{proof}[Proof of \Cref{thm: exceptional measure estimate}, assuming \Cref{prop: iterating projection}]
Fix $s\in(0,1)$ and $t\in(s,2-s)$. Let $\Sigma=\Sigma(s,t)$ be the infimum of the numbers $\sigma>(t-s)/2$ for which $\Proj(s,\sigma,t)$ holds. This means that $\Proj(s,\sigma,t)$ holds for every
$\sigma>\Sigma$.

By \Cref{prop: base case of projection}, there exists $\eta=\eta(s,t)>0$ such that
$\Proj\left(s,\frac t2-\eta,t\right)$ holds. Consequently, $\Sigma\le \frac t2-\eta<\frac t2$. We claim that $\Sigma=\frac{t-s}{2}$. Suppose, to the contrary, that $\Sigma>\frac{t-s}{2}$. Choose $\sigma>\Sigma$ so close to $\Sigma$ that
\begin{equation}
\label{eq: sigma close to Sigma}
\sigma-\zeta(s,\sigma,t)<\Sigma,
\end{equation}
where $\zeta(s,\sigma,t)>0$ is the constant from \Cref{prop: iterating projection}. Such a choice is possible because, under the counter assumption $\Sigma>(t-s)/2$, the parameter $\sigma$ may be restricted to a compact subinterval of $\left(\frac{t-s}{2},\frac t2\right)$, and hence $\zeta(s,\sigma,t)$ is bounded from below there.

Since $\sigma>\Sigma$, the statement $\Proj(s,\sigma,t)$ holds. By \Cref{prop: iterating projection}, this implies $\Proj(s,\sigma-\zeta(s,\sigma,t),t)$. But \eqref{eq: sigma close to Sigma} contradicts the definition of $\Sigma$. Hence $\Sigma=\frac{t-s}{2}$. Now let $
\sigma>\frac{t-s}{2}$. Then $\sigma>\Sigma$, so $\Proj(s,\sigma,t)$ holds. This is precisely \Cref{thm: exceptional measure estimate}.
\end{proof}

For the rest of this section, we prove \Cref{prop: iterating projection}.

\subsection{Small slices imply sparse slices}
\label{subsec: small slices imply sparse slices}

In this subsection, we record a local version of the projection statement $\Proj(s,\sigma,t)$. Roughly speaking, $\Proj(s,\sigma,t)$ says that, for a suitable
direction $\theta$, only a small portion of the measure can lie on $\pi_\theta$-fibres
of multiplicity at least $\delta^{-\sigma}$ between the scales $\delta$ and $1$. The result below shows that the same conclusion can be made simultaneously on all short scale intervals $[r,R]$ with $r/R$ sufficiently small.

\begin{defi}
\label{defi: local high multiplicity set}
Let $\delta\in(0,1/p]$, let $\rho\in p^{-\N}$, let $\sigma\in(0,1]$, and let
$\theta\in\Z_p$. If $K\subset\Q_p^2$, define the local high multiplicity set
$$
H_{\theta,\loc}(K,\sigma,\delta,\rho):=\bigcup_{\delta\le r\le R\le 1}H_\theta\left(K,\left(\frac Rr\right)^\sigma,[r,R]\right),
$$
where the union ranges over all $r,R\in p^{-\N}\cap[\delta,1]$ satisfying $\frac rR\le \rho$.
\end{defi}

Equivalently, $x\in H_{\theta,\loc}(K,\sigma,\delta,\rho)$ if and only if there exist $p$-adic radii $r,R\in[\delta,1]$ with $r\le \rho R$ such that
$$
\mathfrak m_{K,\theta}(x\mid[r,R])\ge\left(\frac Rr\right)^\sigma.
$$
In contrast to the Euclidean setting, no additional constant enlargement is needed
here: the $p$-adic scales are nested exactly, and every $\Delta_{j+1}$-cube is
contained in a unique $\Delta_j$-cube. Notice also that if $\rho\le\delta$, then
\begin{equation}
\label{eq: local equals global small rho}
H_{\theta,\loc}(K,\sigma,\delta,\rho)
\subset H_\theta(K,C_\sigma\delta^{-\sigma},[\delta,1])
\end{equation}
for a constant $C_\sigma\ge 1$. Indeed, the condition $r/R\le\rho\le\delta$ and
$r\ge\delta$ forces $R\sim 1$ in the $p$-adic scale range, up to a bounded change
in constants.

\begin{thm}
\label{thm: small slices imply sparse slices}
Let $s,\sigma\in(0,1)$ and let $t\in(s,2-s)$. Assume that $\Proj(s,\sigma,t)$ holds
with constants $\Delta_0,\epsilon_0>0$. In other words, whenever $\Delta=p^{-m}\in(0,\Delta_0]$, $\nu$ is a
$(\Delta,t,\Delta^{-\epsilon_0})$-regular measure on $\Q_p^2$, with $L:=\spt\nu$, and $F\subset\Z_p$ is a $(\Delta,s,\Delta^{-\epsilon_0})$-set, there exists $\theta\in F$ such that
$$
\nu\bigl(\Z_p^2\cap H_\theta(L,\Delta^{-\sigma},[\Delta,1])\bigr)\le\Delta^{\epsilon_0}.
$$
Then, for every $\eta\in(0,1]$, there exist
$$
\epsilon=\epsilon(\eta,\epsilon_0)>0,
\qquad\delta_0=\delta_0(\Delta_0,\eta,\epsilon)>0
$$
such that the following holds for all $\delta=p^{-n}\in(0,\delta_0]$.

Let $\mu$ be a $(\delta,t,\delta^{-\epsilon})$-regular measure on $\Q_p^2$, with
$K:=\spt\mu$, and let $E\subset\Z_p$ be a $(\delta,s,\delta^{-\epsilon})$-set. Then
there exists $\theta\in E$ such that
\begin{equation}
\label{eq: local sparse slices conclusion}
\mu\bigl(\Z_p^2\cap H_{\theta,\loc}(K,\sigma,\delta,\delta^\eta)\bigr)\le
\delta^\epsilon.
\end{equation}
\end{thm}

\begin{proof}
We argue by contradiction. Fix $\eta\in(0,1]$, and choose $0<\epsilon\ll \eta\epsilon_0$ sufficiently small. Suppose that the conclusion fails for arbitrarily small $\delta=p^{-n}$. Then there exist a $(\delta,t,\delta^{-\epsilon})$-regular measure $\mu$, with $K:=\spt\mu$, and a $(\delta,s,\delta^{-\epsilon})$-set $E\subset\Z_p$ such that
\begin{equation}
\label{eq: local sparse contradiction}
\mu\bigl(H_{\theta,\loc}(K,\sigma,\delta,\delta^\eta)\bigr)>
\delta^\epsilon,\qquad \theta\in E.
\end{equation}

For each $\theta\in E$, by the definition of the local high multiplicity set, there exist
$p$-adic radii
$$
\delta\le r_\theta\le R_\theta\le 1,
\qquad\frac{r_\theta}{R_\theta}\le\delta^\eta,
$$
such that
$$
K_\theta:=
K\cap H_\theta\left(K,\left(\frac{R_\theta}{r_\theta}\right)^\sigma,
[r_\theta,R_\theta]\right)
$$
satisfies $\mu(K_\theta)\gtrapprox_\delta\delta^\epsilon$. The loss hidden in $\gtrapprox_\delta$ comes only from pigeonholing over the
$O(\log^2(1/\delta))$ possible pairs of $p$-adic radii.

Pigeonholing again over the possible pairs $(r_\theta,R_\theta)$, we find radii
$r,R\in p^{-\N}\cap[\delta,1]$ with $r/R\le\delta^\eta$ and a subset $E'\subset E$ such that
\begin{equation}
\label{eq: E prime large}
|E'|_\delta\gtrapprox_\delta |E|_\delta,
\end{equation}
and for every $\theta\in E'$,
\begin{equation}
\label{eq: Ktheta fixed scales}
\mu\left(K\cap H_\theta\left(K,\left(\frac Rr\right)^\sigma,[r,R]\right)\right)
\gtrapprox_\delta\delta^\epsilon .
\end{equation}
Put $\Delta:=r/R$. Then $\Delta\le\delta^\eta$, and hence, for $\delta_0$ sufficiently small,
$\Delta\le\Delta_0$.

We next localize the measure to an $R$-ball, but we do so with an absolute mass lower
bound. Let
$$
K_{\theta,R}:=K\cap H_\theta\left(K,\Delta^{-\sigma},[r,R]\right),\qquad \theta\in E'.
$$
Cover $K$ by $R$-balls. Since $\mu$ is $(\delta,t,\delta^{-\epsilon})$-regular, every $R$-ball
$B$ meeting $K$ satisfies
\begin{equation}
\label{eq: R ball mass comparable}
\delta^\epsilon R^t\le \mu(B)\le \delta^{-\epsilon}R^t .
\end{equation}
Moreover, the number of such $R$-balls is $\lesssim \delta^{-\epsilon}R^{-t}$. For each
$\theta\in E'$, the lower bound \eqref{eq: Ktheta fixed scales} and the upper bound in
\eqref{eq: R ball mass comparable} imply, by pigeonholing over the $R$-balls, that there is an $R$-ball $B$ with
$$
\mu(K_{\theta,R}\cap B)\gtrapprox_\delta \delta^{O(\epsilon)}R^t .
$$
A further pigeonholing over the possible $R$-balls gives an $R$-ball $B:=B(z_0,R)$ and a
subset $E''\subset E'$ such that $|E''|_\delta\gtrapprox_\delta |E|_\delta$, and for every $\theta\in E''$,
\begin{equation}
\label{eq: localized bad mass absolute}
\mu(K_{\theta,R}\cap B)\gtrapprox_\delta \delta^{O(\epsilon)}R^t .
\end{equation}
In particular, $B\cap K\neq\emptyset$, so \eqref{eq: R ball mass comparable} gives
\begin{equation}
\label{eq: selected B mass comparable}
\delta^\epsilon R^t\le \mu(B)\le \delta^{-\epsilon}R^t .
\end{equation}
Dividing \eqref{eq: localized bad mass absolute} by \eqref{eq: selected B mass comparable},
we also have, for every $\theta\in E''$,
\begin{equation}
\label{eq: localized bad mass relative corrected}
\frac{\mu(K_{\theta,R}\cap B)}{\mu(B)}
\gtrapprox_\delta \delta^{O(\epsilon)} .
\end{equation}

Now rescale $B$ to $\Z_p^2$. Let
$$
T_B(z):=\frac{z-z_0}{R},\qquad\nu:=\frac{1}{\mu(B)}T_{B\#}(\mu|_B),
$$
and let $L:=\spt\nu\subset\Z_p^2$. We claim that $\nu$ is a $(\Delta,t,\Delta^{-\epsilon_0})$-regular measure, provided $\epsilon>0$ is sufficiently small
in terms of $\eta,\epsilon_0$. Indeed, for every $u\in L$ and every $\Delta\le \rho\le 1$,
the ball $B(u,\rho)$ in the rescaled coordinates corresponds to a ball of radius $R\rho$ in the original coordinates. Using the regularity of $\mu$ and \eqref{eq: selected B mass comparable},
we obtain
$$
\nu(B(u,\rho))\le\frac{\delta^{-\epsilon}(R\rho)^t}{\delta^\epsilon R^t}\le\delta^{-2\epsilon}\rho^t
$$
and similarly, since $u\in L$,
$$
\nu(B(u,\rho))\ge\frac{\delta^\epsilon(R\rho)^t}{\delta^{-\epsilon}R^t}\ge\delta^{2\epsilon}\rho^t .
$$
Since $\Delta\le\delta^\eta$, choosing $\epsilon\le \eta\epsilon_0/10$ gives
$$
\delta^{-2\epsilon}\le \Delta^{-\epsilon_0},
\qquad\delta^{2\epsilon}\ge \Delta^{\epsilon_0},
$$
after harmless adjustment of constants. Hence $\nu$ is
$(\Delta,t,\Delta^{-\epsilon_0})$-regular.

By \Cref{lem: high multiplicity scaling}, \eqref{eq: localized bad mass relative corrected}
implies that for every $\theta\in E''$,
\begin{equation}
\label{eq: rescaled high multiplicity large}
\nu\bigl(H_\theta(L,\Delta^{-\sigma},[\Delta,1])\bigr)
\gtrapprox_\delta \delta^{O(\epsilon)} .
\end{equation}
After decreasing $\epsilon>0$ and taking $\delta>0$ sufficiently small, the right-hand side is larger than $\Delta^{\epsilon_0}$.

It remains to pass the direction set from scale $\delta$ to scale $\Delta$. Since
$E''\subset E$ and $|E''|_\delta\gtrapprox_\delta |E|_\delta$, $E''$ is a $(\delta,s,\delta^{-O(\epsilon)})$-set. Passing to a maximal $\Delta$-separated subset and using the $p$-adic covering tree, we obtain a set $F\subset E''$ such that $F$ is a
$(\Delta,s,\Delta^{-\epsilon_0})$-set and
$|F|_\Delta\gtrsim \Delta^{-s+\epsilon_0}$. Here we used again $\Delta\le\delta^\eta$ and $\epsilon\ll \eta\epsilon_0$.

Now apply $\Proj(s,\sigma,t)$ at scale $\Delta$ to the measure $\nu$ and the direction set $F$. We obtain a direction $\theta\in F$ such that
$$
\nu\bigl(H_\theta(L,\Delta^{-\sigma},[\Delta,1])\bigr)\le \Delta^{\epsilon_0}.
$$
This contradicts \eqref{eq: rescaled high multiplicity large}. Therefore the assumed failure of the local sparse-slice conclusion is impossible, and the proof is complete.
\end{proof}

\begin{rmk}
\label{rmk: local sparse constants}
The constants in \Cref{thm: small slices imply sparse slices} may be chosen uniformly as long as $(s,\sigma,t)$ stays in a compact subset of the parameter range
$$
s\in(0,1),\qquad t\in(s,2-s),\qquad \sigma\in(0,1).
$$
This follows from the corresponding uniformity in the constants for $\Proj(s,\sigma,t)$ and from the elementary nature of the pigeonholing and rescaling steps above.
\end{rmk}

\subsection{Fixing parameters}
\label{subsec: fixing parameters projection}

We now begin the proof of \Cref{prop: iterating projection}. Fix $s\in(0,1),t\in(s,2-s)$, and let $\frac{t-s}{2}<\sigma<\frac t2$. The purpose of this subsection is to choose the parameters used in the proof and to reduce \Cref{prop: iterating projection} to a contradiction argument.

For the sake of keeping track of uniformity, it is useful to work with a compact subinterval of the admissible range for $\sigma$. Thus, fix numbers
$$
\frac{t-s}{2}<\sigma_0\le \sigma\le \sigma_1<\frac t2.
$$
All constants below are allowed to depend on $s,t,\sigma_0,\sigma_1$, but not on the particular choice of $\sigma\in[\sigma_0,\sigma_1]$. At the end of the proof this will imply that the improvement parameter in \Cref{prop: iterating projection} stays bounded away from zero as long as $\sigma$ stays bounded away from the endpoints $(t-s)/2$ and $t/2$.

We first record the parameter margins needed for the $p$-adic ABC sum-product theorem. The natural exponents are
$$
\alpha_1:=t-\sigma,\qquad \beta_1:=\sigma,\qquad \gamma_1:=s.
$$
They satisfy the admissibility conditions for \Cref{thm: ABC sum-product}. Indeed,
$0<\beta_1<\alpha_1$ because $\sigma<t/2$, and $\alpha_1<1$ because
$$
t-\sigma<t-\frac{t-s}{2}=\frac{s+t}{2}<1,
$$
using $t<2-s$. Finally, $\gamma_1>\alpha_1-\beta_1$ is equivalent to $s>t-2\sigma$, which follows from $\sigma>(t-s)/2$.

Since these inequalities are strict uniformly for $\sigma\in[\sigma_0,\sigma_1]$, there exists
$$
\zeta_0=\zeta_0(s,t,\sigma_0,\sigma_1)>0
$$
such that the perturbed parameters
\begin{equation}
\label{eq: ABC perturbed parameters}
\alpha:=t-\sigma+\zeta_0,\qquad
\beta:=\sigma-\zeta_0,\qquad
\gamma:=s-\zeta_0
\end{equation}
still satisfy
\begin{equation}
\label{eq: ABC parameter inequalities}
0<\beta\le \alpha<1,\qquad\gamma>\alpha-\beta.
\end{equation}
After decreasing $\zeta_0$ if necessary, we may also assume
\begin{equation}
\label{eq: zeta0 small}
0<\zeta_0<\frac{1}{100}
\min\left\{s,\,\sigma_0-\frac{t-s}{2},\,
\frac t2-\sigma_1\right\}.
\end{equation}

Apply \Cref{thm: ABC sum-product} with the parameters $(\alpha,\beta,\gamma)$. We obtain constants
$$
\chi=\chi(\alpha,\beta,\gamma)>0,
\qquad\delta_{\ABC}=\delta_{\ABC}(\alpha,\beta,\gamma)>0
$$
such that the following holds for all $\Delta=p^{-m}\le\delta_{\ABC}$. If $A,B\subset\Z/p^m\Z$ satisfy $|A|\le \Delta^{-\alpha}$, if $B$ is a nonempty $(\Delta,\beta,\Delta^{-\chi})$-set, and if
$C\subset (\Z/p^m\Z)^\times$ is a nonempty $(\Delta,\gamma,\Delta^{-\chi})$-set, then there exists $c\in C$ such that
\begin{equation}
\label{eq: ABC conclusion fixed parameters}
|\{a+cb:(a,b)\in G\}|_\Delta\ge\Delta^{-\chi}|A|
\end{equation}
for every $G\subset A\times B$ satisfying $
|G|\ge \Delta^\chi |A||B|$. Since $\sigma\in[\sigma_0,\sigma_1]$, the triples in \eqref{eq: ABC perturbed parameters} range over a compact subset of the admissible ABC parameter region. After decreasing $\chi$ and $\delta_{\ABC}$ if necessary, we may choose them uniformly for all such $\sigma$.

We now choose the improvement parameter in \Cref{prop: iterating projection}. Let
\begin{equation}
\label{eq: zeta choice projection}
0<\zeta\ll \min\{\zeta_0,\chi\}.
\end{equation}
More explicitly, $\zeta>0$ is chosen sufficiently small so that all losses of size
$O(\zeta)$, $O(\sqrt{\zeta})$, and $O_\zeta(\epsilon)$ appearing later can be absorbed into the margins in \eqref{eq: ABC parameter inequalities}. In particular, we will use
\begin{equation}
\label{eq: ABC margin after losses}
s-10\sqrt{\zeta}>\bigl(t-\sigma+10\zeta\bigr)-\bigl(\sigma-10\sqrt{\zeta}\bigr),
\end{equation}
and
\begin{equation}
\label{eq: beta alpha margins after losses}
0<\sigma-10\sqrt{\zeta}\le
t-\sigma+10\zeta<1.
\end{equation}
The choice is possible by \eqref{eq: ABC parameter inequalities} and
\eqref{eq: zeta0 small}.

Assume that $\Proj(s,\sigma,t)$ holds with constants $\epsilon_0>0,\qquad \Delta_0>0$. We shall prove that $\Proj(s,\sigma-\zeta,t)$ holds. Let $0<\eta\ll \zeta$ be a parameter to be fixed later. Applying \Cref{thm: small slices imply sparse slices} with this $\eta$, and with the constants $\epsilon_0,\Delta_0$ from $\Proj(s,\sigma,t)$, gives
$$\epsilon_1=\epsilon_1(\eta,\epsilon_0)>0,\quad\delta_1=\delta_1(\Delta_0,\eta,\epsilon_1)>0$$
such that the local sparse-slice conclusion \eqref{eq: local sparse slices conclusion} holds for all $\delta\le\delta_1$.

We now choose $0<\epsilon\ll_{\zeta,\eta,\chi}\epsilon_1$ sufficiently small. Finally, choose
$$
\delta_0=\delta_0(s,t,\sigma,\zeta,\eta,\epsilon,\delta_{\ABC},\delta_1)>0
$$
sufficiently small so that all subsequent applications of uniformization, pigeonholing, rescaling, sparse-slice extraction, and \Cref{thm: ABC sum-product} are valid. In particular, whenever a local scale $\Delta$ appears later, we will have $
\Delta\le \min\{\delta_{\ABC},\Delta_0\}$.

We now make the counter assumption. Suppose that $\Proj(s,\sigma-\zeta,t)$ fails with the
constant $\epsilon>0$ chosen above. Then, for arbitrarily small $\delta=p^{-n}\in(0,\delta_0]$, there exist:
\begin{itemize}
    \item a $(\delta,t,\delta^{-\epsilon})$-regular measure $\mu$ on $\Q_p^2$, with
    $K:=\spt\mu\subset\Z_p^2$;
    \item a $(\delta,s,\delta^{-\epsilon})$-set $E\subset\Z_p$;
\end{itemize}
such that, for every $\theta\in E$, $\mu\bigl(H_\theta(K,\delta^{-(\sigma-\zeta)},[\delta,1])\bigr)>\delta^\epsilon$.

On the other hand, since $\Proj(s,\sigma,t)$ holds, \Cref{thm: small slices imply sparse slices} applied to $\mu$ and $E$ gives a direction $\theta_0\in E$ such that
\begin{equation}
\label{eq: sparse slices theta0}
\mu\bigl(H_{\theta_0,\loc}(K,\sigma,\delta,\delta^\eta)\bigr)\le\delta^{\epsilon_1}.
\end{equation}
After decreasing $\epsilon$ if necessary, we may replace the right-hand side by $\delta^{O(\epsilon)}$ in later estimates.

We now normalize this sparse direction to be $0$. Define the affine map
$$
L_{\theta_0}(x,y):=(x+\theta_0 y,y).
$$
Then $\pi_0(L_{\theta_0}(x,y))=\pi_{\theta_0}(x,y)$, and, more generally, $\pi_{\theta-\theta_0}(L_{\theta_0}(x,y))=\pi_\theta(x,y)$. The map $L_{\theta_0}$ is bi-Lipschitz on $\Z_p^2$, and in fact preserves $p$-adic balls up to absolute constants since $\theta_0\in\Z_p$. Thus it preserves all regularity and multiplicity estimates up to harmless
constants. Replacing $K,\mu,E$ by
$$
L_{\theta_0}(K),\quad (L_{\theta_0})_\#\mu,\quad E-\theta_0,
$$
where $(L_{\theta_0})_\#\mu$ denotes the pushforward of $\mu$ under
$L_{\theta_0}$. We may assume from now on that $0\in E$ and that the sparse local estimate holds in the direction $0$:
\begin{equation}
\label{eq: sparse slices zero}
\mu\bigl(H_{0,\loc}(K,\sigma,\delta,\delta^\eta)
\bigr)\le\delta^{O(\epsilon)}.
\end{equation}
The counter assumption becomes
\begin{equation}
\label{eq: counter assumption normalized}
\mu\bigl(H_\theta(K,\delta^{-(\sigma-\zeta)},[\delta,1])\bigr)>\delta^{O(\epsilon)},
\qquad \theta\in E.
\end{equation}
Moreover, $E$ remains a $(\delta,s,\delta^{-O(\epsilon)})$-set. By replacing $\epsilon$ with a sufficiently small multiple of itself, we will continue to write $\delta^{-\epsilon}$ and $\delta^\epsilon$ for these harmless $O(\epsilon)$ losses.

The rest of the proof will derive a contradiction to \Cref{thm: ABC sum-product} from \eqref{eq: sparse slices zero} and \eqref{eq: counter assumption normalized}. The next step is to find a scale at which the direction set $E$ has good branching properties.

\subsection{Finding a branching scale for $E$}
\label{subsec: branching scale for E}

We continue under the assumptions and notation of \Cref{subsec: fixing parameters projection}. Thus $E\subset\Z_p$ is a $(\delta,s,\delta^{-\epsilon})$-set, the sparse direction has been normalized to $0$, and we have
\begin{equation}
\label{eq: sparse slices zero before branching}
\mu\bigl(H_{0,\loc}(K,\sigma,\delta,\delta^\eta)\bigr)\le\delta^{O(\epsilon)}
\end{equation}
and
\begin{equation}
\label{eq: counter assumption before branching}
\mu\bigl(H_\theta(K,\delta^{-(\sigma-\zeta)},[\delta,1])\bigr)>\delta^{O(\epsilon)},
\qquad \theta\in E.
\end{equation}
The purpose of this subsection is to replace $\delta$ by a slightly larger scale $\bar\delta$ so that the direction set has good local branching at scale $\bar\Delta:=\bar\delta^{1/2}$.

We first make $E$ uniform. Choose an integer $T\ge 1$ sufficiently large, depending on $\epsilon,\zeta$, and write $\Delta_0:=p^{-T}$. After decreasing $\delta$ by a harmless factor, we may assume $\delta=\Delta_0^m=p^{-Tm}$ for some integer $m$. By the uniformization lemma from \Cref{sec: Preliminaries}, we may pass to a subset $E'\subset E$, which is $\{\Delta_0^j\}_{j=0}^m$-uniform and satisfies
$|E'|_\delta\ge\delta^{O(\epsilon)}|E|_\delta$. In particular, $E'$ is a $(\delta,s,\delta^{-O(\epsilon)})$-set. Since \eqref{eq: counter assumption before branching} holds for every $\theta\in E$, it also holds for every $\theta\in E'$. Replacing $E$ by $E'$, we assume from now on that $E$ itself is $\{\Delta_0^j\}_{j=0}^m$-uniform.

Let $N_j:=|E\cap I|_{\Delta_0^j}$ for $I\in\mathcal D_{\Delta_0^{j-1}}(E)$, which is independent of $I$ by uniformity. Define the branching function $f:[0,m]\to[0,m]$ by
$$
f(0)=0,\qquad f(j):=\frac{\log |E|_{\Delta_0^j}}{\log \Delta_0^{-1}}
=\sum_{i=1}^j \frac{\log N_i}{\log \Delta_0^{-1}}, \qquad 0\le j\le m,
$$
and extend $f$ linearly on each interval $[j,j+1]$. Since $E\subset\Z_p$, the function $f$ is non-decreasing and $1$-Lipschitz. Since $E$ is a $(\delta,s,\delta^{-O(\epsilon)})$-set, the branching function satisfies
\begin{equation}
\label{eq: E branching superlinear}
f(y)-f(x)\ge s(y-x)-O(\epsilon)m-O_T(1),
\qquad 0\le x\le y\le m.
\end{equation}
Indeed, for every $I\in\mathcal D_{\Delta_0^x}(E)$, the definition of a $(\delta,s,\delta^{-O(\epsilon)})$-set gives
$$
|E\cap I|_\delta\le\delta^{-O(\epsilon)}(\Delta_0^x)^s |E|_\delta,
$$
and uniformity converts this into \eqref{eq: E branching superlinear}.

We shall find an integer $\bar m$ with
\begin{equation}
\label{eq: mbar range}
\frac{\sqrt{\zeta}}{12}m\le \bar m\le \frac m2
\end{equation}
such that $f$ is sufficiently superlinear on all intervals of the form $[\bar m,2\bar m]$ relative to the descendants of every $\Delta_0^{\bar m}$-parent. The following elementary consequence of \eqref{eq: E branching superlinear} is the key point.

\begin{lemma}
\label{lem: good branching scale for E}
There exists an integer $\bar m$ satisfying \eqref{eq: mbar range} such that, for every $
0\le k\le \bar m$, one has
\begin{equation}
\label{eq: good branching inequality}
f(\bar m+k)-f(\bar m)\ge(s-\sqrt{\zeta})k-O_\zeta(\epsilon)\bar m-O_T(1).
\end{equation}
\end{lemma}

\begin{proof}
Let us write explicitly the branching information which will be used below. Recall that, after the uniformization step, the direction set $E$ is $\{\Delta_0^j\}_{j=0}^m$-uniform, and
the branching function $f:[0,m]\to[0,m]$ is defined by
$$
f(j)=\frac{\log |E|_{\Delta_0^j}}{\log \Delta_0^{-1}},
\qquad 0\le j\le m,
$$
with linear interpolation between consecutive integers. Since $E\subset \mathbb Z_p$, the function $f$ is non-decreasing and $1$-Lipschitz. Moreover, the Frostman upper bound for $E$ gives the global lower branching estimate
\begin{equation}
\label{eq: E branching superlinear expanded}
f(y)-f(x)\ge s(y-x)-A\epsilon m-O_T(1),
\qquad 0\le x\le y\le m,
\end{equation}
for some absolute constant $A>0$. This is the estimate denoted by \eqref{eq: E branching superlinear}.

Assume that no such $\bar m$ exists. Then for every integer $\bar m\in\left[\frac{\sqrt{\zeta}}{12}m,\frac m2\right]$ there exists an integer $0\le k_{\bar m}\le \bar m$ such that
\begin{equation}
\label{eq: failure good branching expanded}
f(\bar m+k_{\bar m})-f(\bar m)<(s-\sqrt{\zeta})k_{\bar m}-C_\zeta\epsilon \bar m.
\end{equation}
Combining \eqref{eq: failure good branching expanded} with
\eqref{eq: E branching superlinear expanded}, and using
$\bar m\ge(\sqrt{\zeta}/12)m$, we obtain
$$
\sqrt{\zeta}\, k_{\bar m}
\le A\epsilon m+O_T(1)-C_\zeta\epsilon\bar m
\lesssim_\zeta \epsilon\bar m+O_T(1).
$$
After choosing $\epsilon>0$ sufficiently small in terms of $\zeta$, and then taking
$\delta=\Delta_0^m$ sufficiently small in terms of $T$ and $\zeta$, this implies
\begin{equation}
\label{eq: failure only short intervals}
k_{\bar m}\le \frac{\sqrt{\zeta}}{100}\bar m.
\end{equation}
Thus any possible failure of the desired branching inequality can only occur on a very short interval starting from $\bar m$.

We now use the Lipschitz selection lemma to find a long interval on which such short failures are impossible. Apply the Lipschitz-function lemma from \Cref{sec: Preliminaries} to the function $f$ on $[0,m]$, with error parameter $\sqrt{\zeta}/100$. The hypotheses are satisfied because $f$ is $1$-Lipschitz and because \eqref{eq: E branching superlinear expanded} gives a global lower slope $s$ up to an $O(\epsilon)m+O_T(1)$ error. Taking $\epsilon$ sufficiently small in terms of $\zeta$ and $m$ sufficiently large, the lemma gives an interval $J=[a,b]\subset[0,m]$ with $b-a\gtrsim_\zeta m$
such that, after losing another harmless constant multiple of $\sqrt{\zeta}$, we have
\begin{equation}
\label{eq: local superlinear block}
f(y)-f(x)\ge (s-\sqrt{\zeta}/2)(y-x)-O_\zeta(\epsilon)m-O_T(1),
\qquad a\le x\le y\le b.
\end{equation}
Since $b-a\gtrsim_\zeta m$, we may choose an integer $\bar m\in\left[\frac{\sqrt{\zeta}}{12}m,\frac m2\right]$ such that $
[\bar m,\bar m+(\sqrt{\zeta}/100)\bar m]\subset J$. For this value of $\bar m$, suppose that the desired estimate fails, and let $k_{\bar m}$ be as in \eqref{eq: failure good branching expanded}. By
\eqref{eq: failure only short intervals}, we have $k_{\bar m}\le(\sqrt{\zeta}/100)\bar m$, hence both $\bar m$ and $\bar m+k_{\bar m}$ lie in $J$. Applying \eqref{eq: local superlinear block} with $x=\bar m$ and $y=\bar m+k_{\bar m}$ gives
$$
f(\bar m+k_{\bar m})-f(\bar m)
\ge (s-\sqrt{\zeta}/2)k_{\bar m}-O_\zeta(\epsilon)m-O_T(1).
$$
Using again $\bar m\ge(\sqrt{\zeta}/12)m$, and increasing $C_\zeta$ if necessary, the error term satisfies
$$
O_\zeta(\epsilon)m+O_T(1)\le C_\zeta\epsilon\bar m
$$
for $\epsilon$ sufficiently small and $\delta$ sufficiently small. Therefore
$$
f(\bar m+k_{\bar m})-f(\bar m)
\ge (s-\sqrt{\zeta})k_{\bar m}-C_\zeta\epsilon\bar m,
$$
which contradicts \eqref{eq: failure good branching expanded}. Hence the desired
$\bar m$ exists.

Finally, let us spell out how this branching inequality will be used. For every
$\Delta_0^{\bar m}$-ball $I$ meeting $E$, the number of descendants from scale
$\Delta_0^{\bar m}$ to scale $\Delta_0^{\bar m+k}$ is controlled by $f(\bar m+k)-f(\bar m)$. Thus the inequality just proved says precisely that the rescaled direction set inside $I$ has lower branching exponent at least $s-\sqrt{\zeta}$, up to the loss $\Delta_0^{-C_\zeta\epsilon \bar m}$. This is the form needed in the next step.
\end{proof}

Define $\bar\Delta:=\Delta_0^{\bar m},\bar\delta:=\bar\Delta^2=\Delta_0^{2\bar m}$. By \eqref{eq: mbar range},
\begin{equation}
\label{eq: bardelta range}
\delta\le \bar\delta\le \delta^{\sqrt{\zeta}/6}.
\end{equation}
The exact power is unimportant; by decreasing $\zeta$ one may replace $\sqrt{\zeta}/6$ by $\sqrt{\zeta}/12$ if desired.

Let $E_{\bar\delta}$ be a maximal $\bar\delta$-separated subset of $E$. Since $E$ is uniform and $(\delta,s,\delta^{-O(\epsilon)})$, the set $E_{\bar\delta}$ is a $(\bar\delta,s,\bar\delta^{-O_\zeta(\epsilon)})$-set and satisfies $|E_{\bar\delta}|_{\bar\delta}\gtrapprox_\delta |E|_{\bar\delta}$. Moreover, for every $I\in\mathcal D_{\bar\Delta}(E_{\bar\delta})$, the rescaled set
$$
(E_{\bar\delta})_I:=\bar\Delta\bigl(E_{\bar\delta}\cap I\bigr)\subset\Z_p
$$
is a $(\bar\Delta,s-\sqrt{\zeta},\bar\Delta^{-O_\zeta(\epsilon)})$-set. Indeed, \eqref{eq: good branching inequality} says precisely that the number of descendants of any $\bar\Delta$-interval down to scale $\bar\delta$ has $(s-\sqrt{\zeta})$-dimensional lower branching, while the upper Frostman non-concentrate condition follows from the fact that $E_{\bar\delta}\subset E$ and $E$ is a $(\bar\delta,s,\bar\delta^{-O_\zeta(\epsilon)})$-set. 

\begin{rmk}
Here and for the rest of this work, our scale notation differs slightly from that in Orponen-Shmerkin \cite{orponen2023projections}. In the Euclidean setting, a set at scale $\Delta$ is normalized to unit scale by multiplication by $\Delta^{-1}$, and hence they write $\Delta^{-1}(E\cap I)$. Since the $p$-adic absolute value reverses the usual size of powers of $p$, multiplication by $\Delta$ is the corresponding expanding normalization in our convention. We therefore write $\Delta(E\cap I)$ instead.

The same convention applies to balls and cosets. If $\rho=p^{-k}$ is a geometric scale, then a ball of radius $\rho$ in $\Z_p$ is a coset of $p^k\Z_p=\rho^{-1}\Z_p$, and is therefore written as $x+\rho^{-1}\Z_p$, rather than $x+\rho\Z_p$. Similar reversals will be used throughout whenever a geometric scale is converted into a $p$-adic dilation or an additive subgroup.
\end{rmk}

We next check that the counter assumption and the sparse-slice estimate survive after replacing $\delta$ by $\bar\delta$.

\begin{lemma}
\label{lem: counter survives bardelta}
After possibly decreasing $\epsilon>0$ in terms of $\zeta$, for every $\theta\in E_{\bar\delta}$,
\begin{equation}
\label{eq: counter survives bardelta}
\mu\bigl(H_\theta(K,\bar\delta^{-(\sigma-2\zeta)},[\bar\delta,1])\bigr)>\bar\delta^{O_\zeta(\epsilon)}.
\end{equation}
\end{lemma}

\begin{proof}
Fix $\theta\in E_{\bar\delta}\subset E$. By \eqref{eq: counter assumption before branching},
there is a set
$$
K_\theta\subset H_\theta(K,\delta^{-(\sigma-\zeta)},[\delta,1])
$$
with $\mu(K_\theta)>\delta^{O(\epsilon)}$. For $x\in K_\theta$, the fibre $K_\delta\cap\pi_\theta^{-1}(\pi_\theta(x))$
contains at least $\delta^{-(\sigma-\zeta)}$ many $\delta$-balls. Group these $\delta$-balls into $\bar\delta$-balls. Since $\mu$ is $(\delta,t,\delta^{-\epsilon})$-regular and $\bar\delta\ge\delta$, a single $\bar\delta$-ball can contain at most
$\left(\frac{\bar\delta}{\delta}\right)^{1+O(\epsilon)}$ relevant $\delta$-balls on a fixed $\pi_\theta$-fibre. Therefore the same fibre contains at least $\delta^{-(\sigma-\zeta)}
\left(\frac{\delta}{\bar\delta}\right)^{1+O(\epsilon)}$ distinct $\bar\delta$-balls. Using \eqref{eq: bardelta range} and choosing $\zeta$ sufficiently small, this lower bound is at least $\bar\delta^{-(\sigma-2\zeta)}$
after absorbing the $O(\epsilon)$ losses. Hence
$$
x\in H_\theta(K,\bar\delta^{-(\sigma-2\zeta)},[\bar\delta,1]).
$$
The measure lower bound follows from $
\delta^{O(\epsilon)}\ge\bar\delta^{O_\zeta(\epsilon)}$.
\end{proof}

\begin{lemma}
\label{lem: sparse survives bardelta}
We have
\begin{equation}
\label{eq: sparse survives bardelta}
\mu\bigl(H_{0,\loc}(K,\sigma,\bar\delta,\bar\delta^\eta)
\bigr)\le\bar\delta^{O_\zeta(\epsilon)}.
\end{equation}
\end{lemma}

\begin{proof}
Let $\eta_0>0$ denote the exponent appearing in the definition of
$H_{0,\loc}(K,\sigma,\delta,\delta^{\eta_0})$ before the choice of the branching scale. We now choose the final exponent $\eta>0$ sufficiently small in terms of $\eta_0$ and $\zeta$. More precisely, using
\eqref{eq: bardelta range}, we may assume that $\bar\delta^\eta \le \delta^{\eta_0}$. Indeed, \eqref{eq: bardelta range} gives a lower bound of the form
$\bar\delta\ge \delta^{C_\zeta}$ for some constant $C_\zeta\ge 1$, and hence it
suffices to take $\eta\le \eta_0/C_\zeta$.

With this choice of $\eta$, the relevant local high-multiplicity condition at
scale $\bar\delta$ is stronger than the corresponding condition at scale $\delta$.
More precisely, if
$$
\bar\delta\le r\le R\le 1,
\qquad r/R\le \bar\delta^\eta,
$$
then also
$$
\delta\le r\le R\le 1,
\qquad r/R\le \delta^{\eta_0}.
$$
Consequently,
\begin{equation}\label{eq: local high multiplicity inclusion}
H_{0,\loc}(K,\sigma,\bar\delta,\bar\delta^\eta)
\subset H_{0,\loc}(K,\sigma,\delta,\delta^{\eta_0}).
\end{equation}
Applying \eqref{eq: sparse slices zero before branching} with the exponent
$\eta_0$, and then using \eqref{eq: local high multiplicity inclusion}, gives the
desired sparse-slice estimate at scale $\bar\delta$. Finally, since
$\bar\delta\ge \delta^{C_\zeta}$, every loss of the form $\delta^{O(\epsilon)}$
may be rewritten as a loss of the form $\bar\delta^{O_\zeta(\epsilon)}$. This
proves \eqref{eq: sparse survives bardelta}, after decreasing $\eta$ once more if
necessary.
\end{proof}

Combining the preceding discussion, we have proved the following proposition.

\begin{prop}
\label{prop: branching scale for E}
There exist a scale
$$
\bar\delta\in p^{-\N},
\qquad
\delta\le \bar\delta\le \delta^{\sqrt{\zeta}/12}
$$
and a $\bar\delta$-separated subset $
E_{\bar\delta}\subset E$ with the following properties.
\begin{enumerate}
    \item $E_{\bar\delta}$ is a $(\bar\delta,s,\bar\delta^{-O_\zeta(\epsilon)})$-set.
    \item If $\bar\Delta:=\bar\delta^{1/2}$ and $I\in\mathcal D_{\bar\Delta}(E_{\bar\delta})$, then $(E_{\bar\delta})_I:=\bar\Delta(E_{\bar\delta}\cap I)$ is a $(\bar\Delta,s-\sqrt{\zeta},\bar\Delta^{-O_\zeta(\epsilon)})$-set.
    \item For every $\theta\in E_{\bar\delta}$, $
    \mu\bigl(H_\theta(K,\bar\delta^{-(\sigma-2\zeta)},[\bar\delta,1])\bigr)>\bar\delta^{O_\zeta(\epsilon)}$.
    \item The sparse-slice estimate in the direction $0$ holds at the new scale:
    $$
    \mu\bigl(H_{0,\loc}(K,\sigma,\bar\delta,\bar\delta^\eta)
    \bigr)\le\bar\delta^{O_\zeta(\epsilon)}.
    $$
\end{enumerate}
\end{prop}

For the rest of the proof of \Cref{prop: iterating projection}, we replace $\delta,E$ by $\bar\delta,E_{\bar\delta}$. We also replace $\sigma-\zeta$ by $\sigma-2\zeta$ in the counter assumption. To avoid changing notation, we continue to write $\delta$ and $E$ for the new objects. Thus, from now on, with $\Delta:=\delta^{1/2}$, we may assume:
\begin{enumerate}
    \item $E$ is a $(\delta,s,\delta^{-O_\zeta(\epsilon)})$-set.
    \item For every $I\in\mathcal D_\Delta(E)$, $E_I:=\Delta(E\cap I)$ is a $(\Delta,s-\sqrt{\zeta},\Delta^{-O_\zeta(\epsilon)})$-set.
    \item For every $\theta\in E$,
    $\mu\bigl(H_\theta(K,\delta^{-(\sigma-2\zeta)},[\delta,1])\bigr)>\delta^{O_\zeta(\epsilon)}$.
    \item $\mu\bigl(H_{0,\loc}(K,\sigma,\delta,\delta^\eta)
    \bigr)\le\delta^{O_\zeta(\epsilon)}$.
\end{enumerate}
These are the assumptions used in the next subsection to extract a quasi-product structure from the counterexample.

\subsection{Defining the sets $K_\theta$}
\label{subsec: defining K theta}

We continue the proof of \Cref{prop: iterating projection}. The goal of this subsection is to refine the direction set $E$ once more and, for each remaining direction $\theta\in E$, define a large set $K_\theta$ of points which have high multiplicity in the direction $\theta$, but which do not lie in a locally high multiplicity set in the same direction.

Recall from \Cref{subsec: branching scale for E} that, after replacing the original scale by a slightly larger scale and renaming it as $\delta$, we may assume the following. Let $\Delta:=\delta^{1/2}$.
Then for every $I\in\mathcal D_\Delta(E)$, the rescaled set $E_I:=\Delta(E\cap I)$
is a
\begin{equation}
\label{eq: 46 local E property}
(\Delta,s-\sqrt{\zeta},\Delta^{-O_\zeta(\epsilon)})\text{-set}.
\end{equation}
Moreover, the counter assumption survives in the form
\begin{equation}
\label{eq: 46 high multiplicity all theta}
\mu\bigl(\Z_p^2\cap H_\theta(K,\delta^{-(\sigma-\zeta)},[\delta,1])\bigr)\ge\delta^{C_\zeta\epsilon},
\qquad \theta\in E,
\end{equation}
where $C_\zeta\ge 1$ is a constant depending only on $\zeta$. Here and below $K:=\spt\mu$.
By increasing $C_\zeta$ if necessary, we may absorb all harmless losses of the form
$O_\zeta(\epsilon)$ into the exponent $C_\zeta\epsilon$.

We now use the hypothesis that $\Proj(s,\sigma,t)$ holds. More precisely, we use \Cref{thm: small slices imply sparse slices}, with the parameter $\eta:=\sqrt{\zeta}$. If $\epsilon>0$ is chosen sufficiently small in terms of $\epsilon_0,\eta$, then the theorem
implies the following statement: for every subset $E'\subset E$
which remains a $(\delta,s,\delta^{-\epsilon_0})$-set, there exists $\theta\in E'$ such that
\begin{equation}
\label{eq: 46 local sparse one direction}
\mu\bigl(\Z_p^2\cap H_{\theta,\loc}(K,\sigma,\delta,\delta^\eta)\bigr)\le
\delta^{2C_\zeta\epsilon}.
\end{equation}
Consequently, by repeatedly removing directions satisfying
\eqref{eq: 46 local sparse one direction}, we may find a subset $E_{\mathrm{sp}}\subset E$ with $|E_{\mathrm{sp}}|_\delta\ge \frac12 |E|_\delta$ such that
\begin{equation}
\label{eq: 46 local sparse all remaining theta}
\mu\bigl(\Z_p^2\cap H_{\theta,\loc}(K,\sigma,\delta,\delta^\eta)
\bigr)\le\delta^{2C_\zeta\epsilon},
\qquad \theta\in E_{\mathrm{sp}}.
\end{equation}
Indeed, if fewer than half of the $\delta$-separated directions satisfied
\eqref{eq: 46 local sparse all remaining theta}, then the complementary subset would still be a $(\delta,s,\delta^{-O(\epsilon)})$-set after decreasing $\epsilon$, and \Cref{thm: small slices imply sparse slices} could be applied to that complementary subset, yielding one more direction satisfying \eqref{eq: 46 local sparse one direction}, a contradiction.

We replace $E$ by $E_{\mathrm{sp}}$ and keep the notation $E$. Since
$|E_{\mathrm{sp}}|_\delta\ge |E|_\delta/2$, the set $E$ remains a $(\delta,s,\delta^{-O(\epsilon)})$-set. However, the local property \eqref{eq: 46 local E property} may have been lost after this refinement. To restore it, we apply the same uniformization argument used in \Cref{subsec: branching scale for E} to the new set $E$. Passing to a further subset, and then renaming it again as $E$, we may assume simultaneously that:
\begin{enumerate}
    \item $E$ is a $(\delta,s,\delta^{-O_\zeta(\epsilon)})$-set.
    \item For every $I\in\mathcal D_\Delta(E)$, 
    \begin{equation}
    \label{eq: 46 local E restored}
    E_I:=\Delta(E\cap I)
    \text{ is a }
    (\Delta,s-\sqrt{\zeta},\Delta^{-O_\zeta(\epsilon)})\text{-set}.
    \end{equation}
    \item For every $\theta\in E$,
    \begin{equation}
    \label{eq: 46 high multiplicity restored}
    \mu\bigl(\Z_p^2\cap H_\theta(K,\delta^{-(\sigma-\zeta)},[\delta,1])
    \bigr)\ge\delta^{C_\zeta\epsilon}.
    \end{equation}
    \item For every $\theta\in E$,
    \begin{equation}
    \label{eq: 46 sparse local restored}
    \mu\bigl(\Z_p^2\cap H_{\theta,\loc}(K,\sigma,\delta,\delta^\eta)
    \bigr)\le\delta^{2C_\zeta\epsilon}.
    \end{equation}
\end{enumerate}
The third item is inherited from the counter assumption, since all refinements of $E$ are
subsets of the original direction set. The fourth item is inherited from \eqref{eq: 46 local sparse all remaining theta}. The first two items follow from uniformization, at the cost of replacing the exponent $O(\epsilon)$ by $O_\zeta(\epsilon)$.

We are now ready to define the main objects of this subsection. For every $\theta\in E$, set
\begin{equation}
\label{eq: 46 K theta definition}
K_\theta:=\left(\Z_p^2\cap H_\theta(K,\delta^{-(\sigma-\zeta)},[\delta,1])\right)\setminus
H_{\theta,\loc}(K,\sigma,\delta,\delta^\eta).
\end{equation}

\begin{lemma}
\label{lem: 46 Ktheta large}
For every $\theta\in E$, we have $\mu(K_\theta)\ge\frac12\delta^{C_\zeta\epsilon}$. In particular, after increasing $C_\zeta$ by a factor depending only on $\zeta$, we may simply write $\mu(K_\theta)\ge \delta^{C_\zeta\epsilon}$.
\end{lemma}

\begin{proof}
By \eqref{eq: 46 high multiplicity restored},
$$
\mu\bigl(\Z_p^2\cap H_\theta(K,\delta^{-(\sigma-\zeta)},[\delta,1])\bigr)\ge\delta^{C_\zeta\epsilon}.
$$
On the other hand, by \eqref{eq: 46 sparse local restored},
$$
\mu\bigl(\Z_p^2\cap H_{\theta,\loc}(K,\sigma,\delta,\delta^\eta)
\bigr)\le\delta^{2C_\zeta\epsilon}.
$$
Since $\delta>0$ is sufficiently small, we have $\delta^{2C_\zeta\epsilon}\le \frac12\delta^{C_\zeta\epsilon}$. Subtracting the second estimate from the first completes the proof.
\end{proof}

We record two immediate consequences of the definition of $K_\theta$.

\begin{lemma}
\label{lem: 46 Ktheta high and sparse}
Let $\theta\in E$ and $x\in K_\theta$. Then the following holds.
\begin{enumerate}
    \item the $\theta$-fibre through $x$ has high global multiplicity:
    \begin{equation}
    \label{eq: 46 Ktheta high fibre}
    \mathfrak m_{K,\theta}(x\mid[\delta,1])
    \ge \delta^{-(\sigma-\zeta)}.
    \end{equation}
    \item the same point has no locally high multiplicity in the direction $\theta$: for all $p$-adic radii $r,R\in p^{-\N}\cap[\delta,1]$ satisfying
    $r/R\le \delta^\eta$, one has
    \begin{equation}
    \label{eq: 46 Ktheta sparse local}
    \mathfrak m_{K,\theta}(x\mid[r,R])<\left(\frac Rr\right)^\sigma.
    \end{equation}
\end{enumerate}
\end{lemma}

\begin{proof}
The first assertion follows from $x\in H_\theta(K,\delta^{-(\sigma-\zeta)},[\delta,1])$. The second assertion follows from $x\notin H_{\theta,\loc}(K,\sigma,\delta,\delta^\eta)$ and the definition of the local high multiplicity set.
\end{proof}

The point of the construction is the following. The lower bound \eqref{eq: 46 Ktheta high fibre} says that every $x\in K_\theta$ lies on a $\pi_\theta$-fibre which meets $K_\delta$ in at least $\delta^{-(\sigma-\zeta)}$ many $\delta$-balls between the scales $\delta$ and $1$. On the other hand, \eqref{eq: 46 Ktheta sparse local} says that this multiplicity cannot concentrate inside any short range of scales. This tension will be exploited in the next subsection to find, for many directions $\theta$, a common intermediate scale and a common spatial location where the $\theta$-fibres contain a well-distributed set of descendants.

\subsection{Projecting the sets $K_\theta$}
\label{subsec: projecting K theta}

In this subsection we record two projection estimates for the sets $K_\theta$ constructed in \Cref{subsec: defining K theta}. The first estimate says that, inside any ball of an intermediate radius, the projection of $K_\theta$ cannot be too small. The second estimate says that, inside a tube parallel to the fibres of $\pi_\theta$, the projection of $K_\theta$ is small. These two estimates are the $p$-adic analogues of the corresponding estimates in \cite[Section 4.7]{orponen2023projections}.

Throughout this subsection, we keep the notation $\Delta:=\delta^{1/2},
\eta:=\sqrt{\zeta}$. We assume, as in the previous subsection, that $\zeta>0$ and then $\epsilon>0$ have been chosen
sufficiently small. We shall use the following terminology. If $\rho\in p^{-\N}$ and $\theta\in\Z_p$, a $\rho$-tube parallel to $\ker\pi_\theta$ is a set of the form
$\mathbf T=\pi_\theta^{-1}(J)\cap \Z_p^2$,
where $J\subset\Z_p$ is a $p$-adic ball of radius $\rho$. Equivalently, $\mathbf T$ is the $\rho$-neighbourhood of one fibre of $\pi_\theta$. If $\rho=\delta$, such tubes are the $\delta$-tubes associated with the projection $\pi_\theta$.

\begin{lemma}
\label{lem: 47 projection lower ball}
Let $\theta\in E$, and let $\mathbf B\subset\Z_p^2$ be a $p$-adic ball of radius
$\rho\in p^{-\N}$ satisfying $\delta^{1-\eta}\le \rho\le 1$. Then
\begin{equation}
\label{eq: 47 projection lower ball}
|\pi_\theta(\mathbf B\cap K_\theta)|_\delta
\gtrsim\delta^{-t+O_\zeta(\epsilon)}
\left(\frac{\delta}{\rho}\right)^\sigma
\mu(\mathbf B\cap K_\theta).
\end{equation}
In particular, for $\rho=\Delta$,
\begin{equation}
\label{eq: 47 projection lower Delta ball}
|\pi_\theta(\mathbf B\cap K_\theta)|_\delta
\gtrsim
\delta^{-t+O_\zeta(\epsilon)}
\left(\frac{\delta}{\Delta}\right)^\sigma
\mu(\mathbf B\cap K_\theta).
\end{equation}
\end{lemma}

\begin{proof}
If $\mathbf B\cap K_\theta=\varnothing$, there is nothing to prove. Let
$\mathcal T_{\theta,\delta}$ be the family of $\pi_\theta$-tubes of width $\delta$
which meet $\mathbf B\cap K_\theta$. Then
\begin{equation}
\label{eq: 47 projection comparable to tubes}
|\pi_\theta(\mathbf B\cap K_\theta)|_\delta
\sim
|\mathcal T_{\theta,\delta}|.
\end{equation}

We claim that every $\mathbf T\in\mathcal T_{\theta,\delta}$ satisfies
\begin{equation}
\label{eq: 47 tube sparse in ball}
|(\mathbf B\cap K_\theta)\cap \mathbf T|_\delta
\le
|(\mathbf B\cap K)\cap \mathbf T|_\delta
\lesssim
\left(\frac{\rho}{\delta}\right)^\sigma .
\end{equation}
Fix such a tube $\mathbf T$ and choose
$x_0\in \mathbf B\cap K_\theta\cap \mathbf T$. Since $x_0\in K_\theta$, the local
sparseness estimate in \Cref{lem: 46 Ktheta high and sparse} gives, for every
$r,R\in p^{-\N}\cap[\delta,1]$ with $r/R\le\delta^\eta$,
$$
\mathfrak m_{K,\theta}(x_0\mid[r,R])
<
\left(\frac Rr\right)^\sigma .
$$
We apply this with $r=\delta$ and $R=\rho$. This is allowed because
$\delta/\rho\le\delta^\eta$ by the assumption $\rho\ge\delta^{1-\eta}$. Hence
\begin{equation}
\label{eq: 47 local sparse at delta rho}
\mathfrak m_{K,\theta}(x_0\mid[\delta,\rho])
<
\left(\frac{\rho}{\delta}\right)^\sigma .
\end{equation}

It remains only to relate the $\delta$-tube count in
\eqref{eq: 47 tube sparse in ball} to this exact-fibre multiplicity. Since
$\mathbf B$ is a $p$-adic ball of radius $\rho$ containing $x_0$, we have
$\mathbf B=B(x_0,\rho)$. Also, because $\mathbf T$ is a width $\delta$ tube
parallel to $\ker\pi_\theta$ and contains $x_0$, it has the form $\mathbf T=\pi_\theta^{-1}(\pi_\theta(x_0)+\delta\Z_p)\cap \Z_p^2$.
Let $\mathbf q$ be a $\delta$-ball contained in $(\mathbf B\cap K)\cap\mathbf T$.
Choose $z=(a,b)\in K\cap\mathbf q$. Since $z\in\mathbf T$, $\pi_\theta(z)-\pi_\theta(x_0)\in\delta\Z_p$. Set
$z':=\bigl(a-(\pi_\theta(z)-\pi_\theta(x_0)),\,b\bigr)$. Then $||z-z'||_p\le\delta$, and $\pi_\theta(z')=\pi_\theta(x_0)$. Thus $z'\in K_\delta\cap\pi_\theta^{-1}(\pi_\theta(x_0))$. Moreover, since
$z\in B(x_0,\rho)$ and $\delta\le\rho$, the ultrametric inequality gives
$z'\in B(x_0,\rho)$. Therefore, every $\delta$-ball counted by
$|(\mathbf B\cap K)\cap\mathbf T|_\delta$ is accounted for by
$$
B(x_0,\rho)\cap K_\delta\cap\pi_\theta^{-1}(\pi_\theta(x_0)).
$$
Consequently
$$
|(\mathbf B\cap K)\cap \mathbf T|_\delta
\lesssim\mathfrak m_{K,\theta}(x_0\mid[\delta,\rho])
<\left(\frac{\rho}{\delta}\right)^\sigma,
$$
which proves \eqref{eq: 47 tube sparse in ball}.

Finally, by the $(\delta,t,\delta^{-O_\zeta(\epsilon)})$-regularity of $\mu$, every
$\delta$-ball has $\mu$-mass at most $\delta^{t-O_\zeta(\epsilon)}$. Hence
\begin{equation}
\label{eq: 47 mass to counting}
|\mathbf B\cap K_\theta|_\delta
\gtrsim
\delta^{-t+O_\zeta(\epsilon)}
\mu(\mathbf B\cap K_\theta).
\end{equation}
On the other hand, by \eqref{eq: 47 tube sparse in ball},
$$
|\mathbf B\cap K_\theta|_\delta\le\sum_{\mathbf T\in\mathcal T_{\theta,\delta}}
|(\mathbf B\cap K_\theta)\cap\mathbf T|_\delta\lesssim|\mathcal T_{\theta,\delta}|
\left(\frac{\rho}{\delta}\right)^\sigma .
$$
Combining this with \eqref{eq: 47 projection comparable to tubes} and \eqref{eq: 47 mass to counting} gives
$$
|\pi_\theta(\mathbf B\cap K_\theta)|_\delta
\gtrsim\delta^{-t+O_\zeta(\epsilon)}\left(\frac{\delta}{\rho}\right)^\sigma\mu(\mathbf B\cap K_\theta),
$$
as claimed.
\end{proof}

The second estimate is a converse type statement for the intersection of $K_\theta$ with a tube parallel to the fibres of $\pi_\theta$. It will later be applied with $\rho=\Delta$.

\begin{lemma}
\label{lem: 47 projection upper tube}
Let $\theta\in E$, and let $\mathbf T\subset\Z_p^2$ be a $\rho$-tube parallel to
$\ker\pi_\theta$, where $\delta^{1-\eta}\le \rho\le \delta^\eta$. Then
\begin{equation}
\label{eq: 47 projection upper tube}
|\pi_\theta(K_\theta\cap\mathbf T)|_\delta
\lesssim
\delta^{-O_\zeta(\epsilon)-\zeta}
\left(\frac{\rho}{\delta}\right)^{t-\sigma}.
\end{equation}
In particular, for $\rho=\Delta=\delta^{1/2}$, after taking $\epsilon>0$ sufficiently small depending on $\zeta$,
\begin{equation}
\label{eq: 47 projection upper Delta tube simplified}
|\pi_\theta(K_\theta\cap\mathbf T)|_\delta
\le\Delta^{\sigma-t-3\zeta}.
\end{equation}
\end{lemma}

\begin{proof}
We may assume that $K_\theta\cap\mathbf T\neq\varnothing$. Let
$\mathcal S_\delta$ be the family of $\delta$-tubes parallel to $\ker\pi_\theta$
which are contained in $\mathbf T$ and meet $K_\theta$. Then
\begin{equation}
\label{eq: 47 tube projection equals number}
|\pi_\theta(K_\theta\cap\mathbf T)|_\delta
\sim
|\mathcal S_\delta|.
\end{equation}

We first obtain a lower bound for the number of $\delta$-cubes of $K_\delta$ contained
in the union of these $\delta$-tubes. For each $\mathbf S\in\mathcal S_\delta$,
choose a point $x_{\mathbf S}\in K_\theta\cap\mathbf S$. By the high global
multiplicity estimate in \Cref{lem: 46 Ktheta high and sparse},
\begin{equation}
\label{eq: 47 high global for tube}
\mathfrak m_{K,\theta}(x_{\mathbf S}\mid[\delta,1])
\ge
\delta^{-(\sigma-\zeta)}.
\end{equation}
The exact fibre $\pi_\theta^{-1}(\pi_\theta(x_{\mathbf S}))$ is contained in the $\delta$-tube $\mathbf S$, since $\mathbf S$ is the unique width-$\delta$ $\pi_\theta$-tube containing $x_{\mathbf S}$. Hence
\eqref{eq: 47 high global for tube} gives
$|K_\delta\cap\mathbf S|_\delta\gtrsim\delta^{-(\sigma-\zeta)}$. The tubes in $\mathcal S_\delta$ are disjoint at scale $\delta$, and therefore
\begin{equation}
\label{eq: 47 lower union tube}
|K_\delta\cap\mathbf T|_\delta
\gtrsim |\mathcal S_\delta|\delta^{-(\sigma-\zeta)}.
\end{equation}

We next prove the corresponding upper bound for $|K_\delta\cap\mathbf T|_\delta$.
Fix one point $x_0\in K_\theta\cap\mathbf T$. Since $x_0\in K_\theta$, it does not
belong to the local high-multiplicity set
$H_{\theta,\loc}(K,\sigma,\delta,\delta^\eta)$. We apply the local sparseness estimate
from \Cref{lem: 46 Ktheta high and sparse} with $r=\rho$ and $R=1$. This is allowed
because $\rho\le\delta^\eta$. Thus
\begin{equation}
\label{eq: 47 local sparse rho one}
\mathfrak m_{K,\theta}(x_0\mid[\rho,1])<\rho^{-\sigma}.
\end{equation}

We now relate the $\rho$-tube $\mathbf T$ to the exact fibre through $x_0$. Since
$\mathbf T$ is a $\rho$-tube parallel to $\ker\pi_\theta$ containing $x_0$, it has the form
$$
\mathbf T=\pi_\theta^{-1}(\pi_\theta(x_0)+\rho^{-1}\Z_p)\cap\Z_p^2.
$$
Let $\mathbf q$ be a $\rho$-ball which meets $K_\delta\cap\mathbf T$. Choose
$z=(a,b)\in K_\delta\cap\mathbf T\cap\mathbf q$. Since $z\in\mathbf T$, we have $\pi_\theta(z)-\pi_\theta(x_0)\in\rho^{-1}\Z_p$. Define $z':=\bigl(a-(\pi_\theta(z)-\pi_\theta(x_0)),\,b\bigr)$. Then $||z-z'||_p\le\rho$ and $\pi_\theta(z')=\pi_\theta(x_0)$. By the ultrametric property, $z'\in\mathbf q$. Moreover, since $z\in K_\delta$ and
$\delta\le\rho$, we also have $z'\in K_\rho$. Hence every $\rho$-ball meeting
$K_\delta\cap\mathbf T$ also meets $K_\rho\cap\pi_\theta^{-1}(\pi_\theta(x_0))$.
Therefore \eqref{eq: 47 local sparse rho one} implies that $K_\delta\cap\mathbf T$
is contained in the union of at most $\rho^{-\sigma}$ many $\rho$-balls.

By the $(\delta,t,\delta^{-O_\zeta(\epsilon)})$-regularity of $K$, each $\rho$-ball contains at most $\delta^{-O_\zeta(\epsilon)}\left(\frac{\rho}{\delta}\right)^t$ many $\delta$-balls meeting $K$. Hence
\begin{equation}
\label{eq: 47 upper union tube}
|K_\delta\cap\mathbf T|_\delta
\lesssim\delta^{-O_\zeta(\epsilon)}\rho^{-\sigma}\left(\frac{\rho}{\delta}\right)^t.
\end{equation}

Combining \eqref{eq: 47 lower union tube} and
\eqref{eq: 47 upper union tube}, we obtain
$$
|\mathcal S_\delta|\lesssim\delta^{-O_\zeta(\epsilon)}\delta^{\sigma-\zeta}\rho^{-\sigma}\left(\frac{\rho}{\delta}\right)^t=\delta^{-O_\zeta(\epsilon)-\zeta}
\left(\frac{\rho}{\delta}\right)^{t-\sigma}.
$$
Together with \eqref{eq: 47 tube projection equals number}, this proves
\eqref{eq: 47 projection upper tube}.

Finally, taking $\rho=\Delta=\delta^{1/2}$ gives
$$
|\pi_\theta(K_\theta\cap\mathbf T)|_\delta
\lesssim
\delta^{-O_\zeta(\epsilon)-\zeta}
\left(\frac{\Delta}{\delta}\right)^{t-\sigma}.
$$
Since $\Delta=\delta^{1/2}$, after choosing $\epsilon>0$ sufficiently small depending
on $\zeta$ and then taking $\delta>0$ sufficiently small, the right-hand side is at most $\Delta^{\sigma-t-3\zeta}$. This proves \eqref{eq: 47 projection upper Delta tube simplified}.
\end{proof}

\subsection{Choosing a good tube at the effective direction scale}
\label{subsec: choosing good Delta tube}

We continue from \Cref{subsec: projecting K theta}. We keep the notation
$\Delta=\delta^{1/2}$ and $\eta=\sqrt{\zeta}$. Recall from
\Cref{lem: 46 Ktheta large} that
$\mu(K_\theta)\ge \delta^{C_\zeta\epsilon}$ for every $\theta\in E$. Moreover,
by \eqref{eq: 46 local E restored}, for every
$I\in\mathcal D_\Delta(E)$ and every $\theta_I\in I$, the normalized direction
set $E_I:=\Delta(E\cap I-\theta_I)$
is a $(\Delta,s-\sqrt{\zeta},\Delta^{-O_\zeta(\epsilon)})$-set.

Fix a ball $I\in\mathcal D_\Delta(E)$ which survives the preceding
uniformization. Write $F:=E\cap I$. By the local branching lower bound,
\begin{equation}
\label{eq: 48 local direction cardinality}
|F|\ge \Delta^{-s+O(\sqrt{\zeta})+O_\zeta(\epsilon)}.
\end{equation}
For $\theta_0\in F$, define
$$
A(\theta_0):=\sum_{\theta\in F}\mu(K_{\theta_0}\cap K_\theta).
$$
Cauchy's inequality gives
$$\sum_{\theta_0\in F}A(\theta_0)\ge\int\left(\sum_{\theta\in F}\mathbf 1_{K_\theta}(x)\right)^2d\mu(x)\ge
\left(\sum_{\theta\in F}\mu(K_\theta)\right)^2\ge\delta^{O_\zeta(\epsilon)}|F|^2.$$
Hence there exists $\theta_0\in F$ such that
\begin{equation}
\label{eq: 48 reference direction overlap}
A(\theta_0)\ge \delta^{O_\zeta(\epsilon)}|F|.
\end{equation}

We next select the scale at which the directions contributing to
\eqref{eq: 48 reference direction overlap} separate from $\theta_0$. Let
$$
F_{\mathrm{near}}:=
F\cap\bigl(\theta_0+\delta^{-(1-\eta)}\Z_p\bigr).
$$
Since $F$ becomes a
$(\Delta,s-\sqrt{\zeta},\Delta^{-O_\zeta(\epsilon)})$-set after normalization
by $\Delta$, we have
$$
|F_{\mathrm{near}}|
\le
\Delta^{-O_\zeta(\epsilon)}
\left(\frac{\delta^{1-\eta}}{\Delta}\right)^{s-\sqrt{\zeta}}
|F|.
$$
Recall that $\Delta=\delta^{1/2}$ and $\eta=\sqrt{\zeta}$. After choosing
$\epsilon>0$ sufficiently small in terms of $s$ and $\zeta$, the right hand
side is at most $\frac12\delta^{O_\zeta(\epsilon)}|F|$. Since every overlap has mass at most
$1$, \eqref{eq: 48 reference direction overlap} therefore implies
\begin{equation}
\label{eq: 48 overlap away from terminal direction scale}
\sum_{\theta\in F\setminus F_{\mathrm{near}}}
\mu(K_{\theta_0}\cap K_\theta)\ge\delta^{O_\zeta(\epsilon)}|F|.
\end{equation}

The set $F\setminus F_{\mathrm{near}}$ is the disjoint union of the valuation
shells
$$
F_j:=\left\{\theta\in F:
\Delta(\theta-\theta_0)\in
p^j\Z_p^\times\right\},
$$
where $j$ ranges over the integers for which
$\delta^{1-\eta}\le \Delta p^{-j}\le\Delta$. There are
$O(\log(1/\delta))$ such integers. Hence, after absorbing this logarithmic
loss into $\delta^{O_\zeta(\epsilon)}$, there exists an integer $j\ge0$ such
that, with $\rho:=\Delta p^{-j}$, one has
\begin{equation}
\label{eq: 48 one valuation shell overlap}
\delta^{1-\eta}\le\rho\le\Delta
\quad\text{and}\quad
\sum_{\theta\in F_j}
\mu(K_{\theta_0}\cap K_\theta)
\ge
\delta^{O_\zeta(\epsilon)}|F|.
\end{equation}
Thus every $\theta\in F_j$ has a unique representation
$$
\theta=\theta_0+\rho^{-1}c_\theta,
\qquad c_\theta\in\Z_p^\times.
$$

A final thresholding produces pointwise overlap. There exists
$\mathcal E\subset F_j$ such that
\begin{equation}
\label{eq: 48 effective direction slice size}
|\mathcal E|\ge\delta^{O_\zeta(\epsilon)}|F|
\end{equation}
and
\begin{equation}
\label{eq: 48 effective direction pointwise overlap}
\mu(K_{\theta_0}\cap K_\theta)
\ge\delta^{O_\zeta(\epsilon)},
\qquad \theta\in\mathcal E.
\end{equation}
Indeed, this follows from \eqref{eq: 48 one valuation shell overlap}, since
each summand is at most $1$.

Set
$$
\tau:=\frac{\delta}{\rho}
\quad\text{and}\quad
C_{\theta_0,\rho}:=\rho(\mathcal E-\theta_0).
$$
Since $\delta^{1-\eta}\le\rho\le\Delta$, we have
$\Delta\le\tau\le\delta^\eta$. Moreover,
$C_{\theta_0,\rho}\subset\Z_p^\times$. The non-concentration of $F$, together
with \eqref{eq: 48 effective direction slice size}, gives
\begin{equation}
\label{eq: 48 effective coefficient set}
C_{\theta_0,\rho}
\text{ is a }
(\tau,s-O(\sqrt{\zeta}),\tau^{-O_\zeta(\epsilon)})\text{-set}.
\end{equation}
To verify this, let $\tau\le r\le1$ and let $J\subset\Z_p$ be an $r$-ball.
Then $\theta_0+\rho^{-1}J$ is a ball of radius $\rho r$ contained in $I$, and
hence
$$
|\mathcal E\cap(\theta_0+\rho^{-1}J)|
\le
\Delta^{-O_\zeta(\epsilon)}
\left(\frac{\rho r}{\Delta}\right)^{s-\sqrt{\zeta}}
|F|
\le
\tau^{-O_\zeta(\epsilon)}r^{s-O(\sqrt{\zeta})}|\mathcal E|.
$$
This proves \eqref{eq: 48 effective coefficient set}.

For the rest of the proof, the reference direction $\theta_0$, the effective
direction scale $\rho$, and the direction slice $\mathcal E$ are fixed. We
write $K_0:=K_{\theta_0}$. After applying the affine change of variables
$(x,y)\mapsto(x+\theta_0y,y)$, we may regard $\pi_{\theta_0}$ as the reference
projection. We keep the same notation for $K$ and $\mu$.

Let $\mathcal B_\rho$ be the family of $\rho$-balls meeting $K$. By the
$(\delta,t,\delta^{-O_\zeta(\epsilon)})$-regularity of $\mu$,
\begin{equation}
\label{eq: 48 number rho balls}
|\mathcal B_\rho|
\le
\delta^{-O_\zeta(\epsilon)}\rho^{-t}.
\end{equation}
From \eqref{eq: 48 effective direction pointwise overlap},
\begin{equation}
\label{eq: 48 sum over rho balls}
\sum_{\theta\in\mathcal E}
\sum_{\mathbf B\in\mathcal B_\rho}
\mu(K_0\cap K_\theta\cap\mathbf B)
\ge
\delta^{O_\zeta(\epsilon)}|\mathcal E|.
\end{equation}

We say that $\mathbf B\in\mathcal B_\rho$ is light if
$$
\frac1{|\mathcal E|}
\sum_{\theta\in\mathcal E}
\mu(K_0\cap K_\theta\cap\mathbf B)
\le
\rho^{t+C_\zeta\epsilon},
$$
where $C_\zeta$ is a sufficiently large constant. Let
$\mathcal B_\rho^{\mathrm{heavy}}$ be the complementary family. By
\eqref{eq: 48 number rho balls}, the contribution of the light balls is at
most
$$
\delta^{-O_\zeta(\epsilon)}\rho^{C_\zeta\epsilon}|\mathcal E|.
$$
Since $\rho\ge\delta^{1-\eta}$, this is at most half of the right hand side of
\eqref{eq: 48 sum over rho balls}, after taking $C_\zeta$ sufficiently large
and then $\epsilon>0$ sufficiently small. Therefore
\begin{equation}
\label{eq: 48 heavy rho ball contribution}
\sum_{\theta\in\mathcal E}
\sum_{\mathbf B\in\mathcal B_\rho^{\mathrm{heavy}}}
\mu(K_0\cap K_\theta\cap\mathbf B)
\ge
\delta^{O_\zeta(\epsilon)}|\mathcal E|.
\end{equation}
Every $\mathbf B\in\mathcal B_\rho^{\mathrm{heavy}}$ also satisfies
\begin{equation}
\label{eq: 48 heavy rho ball K0 mass}
\mu(K_0\cap\mathbf B)\ge\rho^{t+C_\zeta\epsilon}.
\end{equation}

Applying \Cref{lem: 47 projection lower ball} with $\theta=\theta_0$ and scale $\rho$, we obtain
\begin{equation}\label{eq: 48 heavy rho ball projection lower}
|\pi_{\theta_0}(\mathbf B\cap K_0)|_\delta\gtrsim\delta^{-t+O_\zeta(\epsilon)}\left(\frac{\delta}{\rho}\right)^\sigma\mu(\mathbf B\cap K_0)\ge
\delta^{\sigma-t+O_\zeta(\epsilon)}\rho^{t-\sigma},\qquad
\mathbf B\in\mathcal B_\rho^{\mathrm{heavy}}.
\end{equation}

Let $\mathcal T_\rho$ be the minimal family of $\pi_{\theta_0}$-tubes of width $\rho$ covering all balls in
$\mathcal B_\rho^{\mathrm{heavy}}$. For every
$\mathbf T\in\mathcal T_\rho$, choose a ball
$\mathbf B_{\mathbf T}\in\mathcal B_\rho^{\mathrm{heavy}}$
with $\mathbf B_{\mathbf T}\cap\mathbf T\neq\varnothing$.
The projection images
$\pi_{\theta_0}(\mathbf B_{\mathbf T})$ lie in distinct $\rho$-balls as $\mathbf T$ ranges over $\mathcal T_\rho$.
Hence \eqref{eq: 48 heavy rho ball projection lower} gives
\begin{equation}
\label{eq: 48 projection K0 lower by rho tubes}
|\pi_{\theta_0}(K_0)|_\delta
\gtrsim
\delta^{\sigma-t+O_\zeta(\epsilon)}
|\mathcal T_\rho|\rho^{t-\sigma}.
\end{equation}
On the other hand, the global high-multiplicity property defining $K_0$
implies
\begin{equation}
\label{eq: 48 projection K0 upper variable scale}
|\pi_{\theta_0}(K_0)|_\delta\lesssim\delta^{\sigma-t-\zeta+O_\zeta(\epsilon)}.
\end{equation}
Comparing the last two displays, and decreasing $\epsilon$ in terms of
$\zeta$, yields
\begin{equation}
\label{eq: 48 number rho tubes}
|\mathcal T_\rho|
\le
\delta^{-\zeta-O_\zeta(\epsilon)}\rho^{\sigma-t}.
\end{equation}

For $\mathbf T\in\mathcal T_\rho$, let
$$
\mathcal B(\mathbf T)
:=
\left\{
\mathbf B\in\mathcal B_\rho^{\mathrm{heavy}}:
\mathbf B\cap\mathbf T\neq\varnothing
\right\}.
$$
Then \eqref{eq: 48 heavy rho ball contribution} implies
\begin{equation}
\label{eq: 48 rho tube total contribution}
\sum_{\mathbf T\in\mathcal T_\rho}
\sum_{\theta\in\mathcal E}
\sum_{\mathbf B\in\mathcal B(\mathbf T)}
\mu(K_0\cap K_\theta\cap\mathbf B)
\ge
\delta^{O_\zeta(\epsilon)}|\mathcal E|.
\end{equation}
Call $\mathbf T\in\mathcal T_\rho$ heavy if
$$
\sum_{\theta\in\mathcal E}
\sum_{\mathbf B\in\mathcal B(\mathbf T)}
\mu(K_0\cap K_\theta\cap\mathbf B)
\ge
\delta^{2\zeta}\rho^{t-\sigma}|\mathcal E|.
$$
By \eqref{eq: 48 number rho tubes}, the total contribution of the non-heavy
tubes is at most
$\delta^{\zeta-O_\zeta(\epsilon)}|\mathcal E|$, which is less than half of the
right hand side of \eqref{eq: 48 rho tube total contribution}. Hence there is
a heavy tube; fix one and denote it by
\begin{equation}
\label{eq: 48 T0 fixed}
\mathbf T_0\in\mathcal T_\rho.
\end{equation}

Since $\mathbf T_0$ is a $\pi_{\theta_0}$-tube of width $\rho$ and
$|\theta-\theta_0|_p=\rho$ for every $\theta\in\mathcal E$, it is also,
up to harmless constants, a $\pi_\theta$-tube of width $\rho$. Therefore
\Cref{lem: 47 projection upper tube}, applied at the scale $\rho$, gives
\begin{equation}
\label{eq: 48 Ktheta T0 small projection variable scale}
|\pi_\theta(K_\theta\cap\mathbf T_0)|_\delta
\le\delta^{-O_\zeta(\epsilon)}\left(\frac{\delta}{\rho}\right)^{t-\sigma}=\delta^{-O_\zeta(\epsilon)}\tau^{t-\sigma},
\qquad \theta\in\mathcal E.
\end{equation}

We next record the non-concentration of the heavy $\rho$-balls inside
$\mathbf T_0$.

\begin{lemma}
\label{lem: 48 nonconcentration B T0}
For every $x\in\Z_p^2$ and every
$R\in p^{-\N}$ with $\delta^{-\eta}\rho\le R\le1$, one has
\begin{equation}
\label{eq: 48 nonconcentration B T0}
|\mathcal B(\mathbf T_0)\cap B(x,R)|
\lesssim
\left(\frac R\rho\right)^\sigma.
\end{equation}
In particular, $|\mathcal B(\mathbf T_0)|\lesssim\rho^{-\sigma}$.
\end{lemma}

\begin{proof}
Fix $x$ and $R$ as in the statement, and choose
$\mathbf B\in\mathcal B(\mathbf T_0)\cap B(x,R)$. By
\eqref{eq: 48 heavy rho ball K0 mass}, there exists
$x_0\in\mathbf B\cap K_0$. All balls in
$\mathcal B(\mathbf T_0)\cap B(x,R)$ lie in $B(x_0,R)$ and in the same
$\pi_{\theta_0}$-tube of width $\rho$. Since
$\rho/R\le\delta^\eta$, the local sparse estimate in
\Cref{lem: 46 Ktheta high and sparse} gives
$$
\mathfrak m_{K,\theta_0}(x_0\mid[\rho,R])
<
\left(\frac R\rho\right)^\sigma.
$$
This bounds the number of such $\rho$-balls and proves
\eqref{eq: 48 nonconcentration B T0}. The final assertion follows by taking
$R=1$.
\end{proof}

Finally, the heaviness of $\mathbf T_0$ can be made pointwise in many
directions. For every $\theta\in\mathcal E$, regularity and
\Cref{lem: 48 nonconcentration B T0} give
$$
\sum_{\mathbf B\in\mathcal B(\mathbf T_0)}
\mu(\mathbf B\cap K_0\cap K_\theta)
\lesssim
\rho^{t-\sigma-O_\zeta(\epsilon)}.
$$
Consequently, after passing to a subset
$\mathcal E_1\subset\mathcal E$ with
$|\mathcal E_1|\ge\delta^{3\zeta}|\mathcal E|$, we have
\begin{equation}
\label{eq: 48 final pointwise tube lower}
\sum_{\mathbf B\in\mathcal B(\mathbf T_0)}
\mu(\mathbf B\cap K_0\cap K_\theta)
\ge
\delta^{3\zeta}\rho^{t-\sigma},
\qquad \theta\in\mathcal E_1.
\end{equation}
Replacing $\mathcal E$ by $\mathcal E_1$ and renaming, we may assume that
\eqref{eq: 48 final pointwise tube lower} holds for every $\theta\in\mathcal E$. The normalized coefficient set $C_{\theta_0,\rho}=\rho(\mathcal E-\theta_0)$
remains a
\begin{equation}
\label{eq: 48 final coefficient set properties}
(\tau,s-O(\sqrt{\zeta}),\tau^{-O(\sqrt{\zeta})})\text{-set}
\quad\text{contained in }\Z_p^\times,
\end{equation}
where $\tau=\delta/\rho$. Equivalently, every $\theta\in\mathcal E$ has the
form
\begin{equation}
\label{eq: 48 final direction representation}
\theta=\theta_0+\rho^{-1}c_\theta,
\qquad c_\theta\in C_{\theta_0,\rho}\subset\Z_p^\times.
\end{equation}

For notational simplicity, in the sequel we replace the original direction
set $E$ by this final slice $\mathcal E$. Thus, from now on, $C_{\theta_0,\rho}=\rho(E-\theta_0)$ is a $(\tau,s-O(\sqrt{\zeta}),\tau^{-O(\sqrt{\zeta})})$-set contained in $\Z_p^\times$, and the tube $\mathbf T_0$ fixed in \eqref{eq: 48 T0 fixed} will be used in the next subsection to construct the sets $A$, $A_\theta$, and the restricted graphs at the effective scale $\rho$.

\subsection{The sets $\mathcal A$ and $\mathcal A_\theta$}
\label{subsec: sets A and A theta}

We continue with the notation of
\Cref{subsec: choosing good Delta tube}. Thus $\rho$ is the effective
direction scale, $\tau=\delta/\rho$, and $\mathbf T_0$ is the fixed
$\pi_{\theta_0}$-tube of width $\rho$. We also write
$K_0:=K_{\theta_0}$.

The ABC theorem will eventually be applied at resolution $\tau$. We therefore
first arrange that the transverse objects inside $\mathbf T_0$ are uniform when
viewed at this resolution. Partition the balls in
$\mathcal B(\mathbf T_0)$ according to their transverse $\tau$-parents. By
pigeonholing the parent multiplicities into powers of $p$, and then
pigeonholing once more in the direction variable, there exist a subfamily
\[
\mathcal B_*\subset\mathcal B(\mathbf T_0),
\]
a subset $E_*\subset E$, and an integer $m\ge1$ with the following
properties:

\begin{enumerate}
\item $|E_*|\ge\delta^{O_\zeta(\epsilon)}|E|$;

\item every occupied transverse $\tau$-parent contains between $m$ and
$pm$ members of $\mathcal B_*$;

\item for every $\theta\in E_*$,
\begin{equation}
\label{eq: 49 selected pointwise tube lower}
\sum_{\mathbf B\in\mathcal B_*}
\mu\bigl(
\mathbf B\cap K_0\cap K_\theta
\bigr)
\gtrsim
\delta^{4\zeta}\rho^{t-\sigma}.
\end{equation}
\end{enumerate}

Here the loss in the last display is harmless after choosing $\epsilon>0$
sufficiently small in terms of $\zeta$. Replacing $E_*$ by $E$ and
$\mathcal B_*$ by $\mathcal B(\mathbf T_0)$, we assume from now on that these
properties hold.

Let
\begin{equation}
\label{eq: 49 B definition}
B
\subset
\Z_p/\tau^{-1}\Z_p
\end{equation}
be the set of transverse $\tau$-parents occupied by
$\mathcal B(\mathbf T_0)$. Thus, for every $b\in B$, there are between $m$
and $pm$ balls $\mathbf B\in\mathcal B(\mathbf T_0)$ whose transverse $\tau$-parent is $b$.

We record the non-concentration property of $B$.

\begin{lemma}
\label{lem: 49 B nonconcentration}
The set $B$ is a $\bigl(
\tau,\sigma-O(\sqrt{\zeta}),
\tau^{-O(\sqrt{\zeta})}\bigr)$-set.
\end{lemma}

\begin{proof}
Let $I\subset\Z_p$ be a ball of radius $R\in[\tau,1]\cap p^{-\N}$. Every element of $B\cap I$ contains at least $m$ balls from
$\mathcal B(\mathbf T_0)$. Hence
\[
m|B\cap I|\le\left|\left\{
\mathbf B\in\mathcal B(\mathbf T_0):
\text{the transverse coordinate of }\mathbf B\text{ lies in }I
\right\}
\right|.
\]
Similarly, $m|B|\sim|\mathcal B(\mathbf T_0)|$, up to a factor depending only on $p$. Therefore
\Cref{lem: 48 nonconcentration B T0}, together with the trivial estimate in
the range below $\delta^{-\eta}\rho$, gives
\[
|B\cap I|
\le
\tau^{-O(\sqrt{\zeta})}
R^{\sigma-O(\sqrt{\zeta})}
|B|.
\]
This proves the lemma.
\end{proof}

We next define the horizontal set. Let
\begin{equation}
\label{eq: 49 A collection definition}
\mathcal A
:=
D_\delta\bigl(
\pi_{\theta_0}(K_0\cap\mathbf T_0)
\bigr).
\end{equation}
Equivalently, $\mathcal A$ is the family of $\delta$-balls
$J\subset\Z_p$ such that
\[
\pi_{\theta_0}^{-1}(J)
\cap
K_0
\cap
\mathbf T_0
\neq
\varnothing.
\]
By \Cref{lem: 47 projection upper tube}, applied with
$\theta=\theta_0$ and tube width $\rho$,
\begin{equation}
\label{eq: 49 A upper bound}
|\mathcal A|
\le
\delta^{-O_\zeta(\epsilon)}
\tau^{\sigma-t}.
\end{equation}

For $\theta\in E$ and $J\in\mathcal A$, define
\begin{equation}
\label{eq: 49 B J theta definition}
B_{J,\theta}
:=
\left\{
b\in B:
\begin{array}{l}
\text{there exists }\mathbf B\in\mathcal B(\mathbf T_0)
\text{ with transverse}\\
\text{$\tau$-parent $b$ such that }
\pi_{\theta_0}^{-1}(J)
\cap
\mathbf B
\cap
K_0
\cap
K_\theta
\neq
\varnothing
\end{array}
\right\}.
\end{equation}

We then define
\begin{equation}
\label{eq: 49 A theta definition}
\mathcal A_\theta
:=
\left\{
J\in\mathcal A:
|B_{J,\theta}|
\ge
\delta^{6\zeta}|B|
\right\}.
\end{equation}
Thus $\mathcal A_\theta$ consists of those projection balls which are incident
to a positive proportion, up to the permitted power loss, of the transverse
set $B$. This is the form that will be used directly in the restricted graph
for the ABC theorem.

\begin{lemma}
\label{lem: 49 slice mass upper}
Let $\mathbf B\in\mathcal B(\mathbf T_0)$ and
$J\in\mathcal D_\delta(\Z_p)$. Then
\begin{equation}
\label{eq: 49 slice mass upper}
\mu\bigl(
\pi_{\theta_0}^{-1}(J)
\cap
\mathbf B
\cap
K_0
\bigr)
\lesssim
\tau^{-\sigma}
\delta^{t-O_\zeta(\epsilon)}.
\end{equation}
\end{lemma}

\begin{proof}
Assume that the set on the left is nonempty, and choose
\[
x_0\in\pi_{\theta_0}^{-1}(J)\cap\mathbf B\cap K_0.
\]
Since $x_0\in K_0$, the local sparse estimate
\eqref{eq: 46 Ktheta sparse local} applies with
$r=\delta$ and $R=\rho$. Indeed, $\frac{\delta}{\rho}=\tau\le\delta^\eta$.  Hence
\[
\mathfrak m_{K,\theta_0}
\bigl(
x_0\mid[\delta,\rho]
\bigr)
<
\left(
\frac{\rho}{\delta}
\right)^\sigma
=
\tau^{-\sigma}.
\]

The ball $\mathbf B$ is a $\rho$-ball containing $x_0$, while
$\pi_{\theta_0}^{-1}(J)$ is the $\delta$-neighbourhood of the
$\pi_{\theta_0}$-fibre through $x_0$. Therefore 
\[
\pi_{\theta_0}^{-1}(J)\cap\mathbf B\cap K_0
\]
is covered by $\lesssim
\tau^{-\sigma}$ many $\delta$-balls. Each such ball has $\mu$-mass at most $\delta^{t-O_\zeta(\epsilon)}$ by the regularity of $\mu$. This proves
\eqref{eq: 49 slice mass upper}.
\end{proof}

\begin{lemma}
\label{lem: 49 parent slice mass upper}
Let $b\in B$ and $J\in\mathcal D_\delta(\Z_p)$. Then
\begin{equation}
\label{eq: 49 parent slice mass upper}
\sum_{\substack{
\mathbf B\in\mathcal B(\mathbf T_0)\\
\mathbf B\text{ has transverse }\tau\text{-parent }b
}}
\mu\bigl(
\pi_{\theta_0}^{-1}(J)
\cap
\mathbf B
\cap
K_0
\bigr)
\lesssim
m\tau^{-\sigma}
\delta^{t-O_\zeta(\epsilon)}.
\end{equation}
\end{lemma}

\begin{proof}
Every $b\in B$ contains at most $pm$ members of
$\mathcal B(\mathbf T_0)$. The conclusion follows immediately from
\Cref{lem: 49 slice mass upper}, after absorbing the factor $p$ into the
implicit constant.
\end{proof}

\begin{lemma}
\label{lem: 49 A theta large}
For every $\theta\in E$,
\begin{equation}
\label{eq: 49 A theta lower bound}
|\mathcal A_\theta|
\ge
\delta^{5\zeta}
\tau^{\sigma-t}.
\end{equation}
\end{lemma}

\begin{proof}
Fix $\theta\in E$. We first estimate the contribution from
$\mathcal A\setminus\mathcal A_\theta$.

If $J\in\mathcal A\setminus\mathcal A_\theta$, then $|B_{J,\theta}|
<\delta^{6\zeta}|B|$. Therefore, by \Cref{lem: 49 parent slice mass upper},
$$
\sum_{\mathbf B\in\mathcal B(\mathbf T_0)}
\mu\bigl(\pi_{\theta_0}^{-1}(J)\cap
\mathbf B\cap K_0\cap K_\theta\bigr)\lesssim
\delta^{6\zeta}|B|m\tau^{-\sigma}
\delta^{t-O_\zeta(\epsilon)}.
$$
Since 
\[
m|B|
\lesssim
|\mathcal B(\mathbf T_0)|
\lesssim
\rho^{-\sigma}
\]
by \Cref{lem: 48 nonconcentration B T0}, the right hand side is bounded by
\[
\delta^{6\zeta}
\rho^{-\sigma}
\tau^{-\sigma}
\delta^{t-O_\zeta(\epsilon)}
=
\delta^{t-\sigma+6\zeta-O_\zeta(\epsilon)}.
\]

Summing over
$J\in\mathcal A\setminus\mathcal A_\theta$ and using
\eqref{eq: 49 A upper bound}, we obtain
\begin{equation}\label{eq: 49 bad A total}
\sum_{J\in\mathcal A\setminus\mathcal A_\theta}\sum_{\mathbf B\in\mathcal B(\mathbf T_0)}
\mu\bigl(\pi_{\theta_0}^{-1}(J)
\cap\mathbf B\cap K_0\cap
K_\theta\bigr)\lesssim\delta^{6\zeta-O_\zeta(\epsilon)}
\rho^{t-\sigma}.
\end{equation}
After choosing $\epsilon>0$ sufficiently small in terms of $\zeta$, this is
at most half of the lower bound in
\eqref{eq: 49 selected pointwise tube lower}. Consequently,
\begin{equation}
\label{eq: 49 good A contribution lower}
\sum_{J\in\mathcal A_\theta}
\sum_{\mathbf B\in\mathcal B(\mathbf T_0)}
\mu\bigl(
\pi_{\theta_0}^{-1}(J)
\cap
\mathbf B
\cap
K_0
\cap
K_\theta
\bigr)
\gtrsim
\delta^{4\zeta}
\rho^{t-\sigma}.
\end{equation}

On the other hand, by \Cref{lem: 49 slice mass upper} and $|\mathcal B(\mathbf T_0)|\lesssim\rho^{-\sigma}$, we have
\begin{equation}
\label{eq: 49 good A contribution upper}
\sum_{J\in\mathcal A_\theta}
\sum_{\mathbf B\in\mathcal B(\mathbf T_0)}
\mu\bigl(\pi_{\theta_0}^{-1}(J)
\cap\mathbf B\cap K_0\cap K_\theta
\bigr)\lesssim|\mathcal A_\theta|
\rho^{-\sigma}\tau^{-\sigma}
\delta^{t-O_\zeta(\epsilon)}=
|\mathcal A_\theta|\delta^{t-\sigma-O_\zeta(\epsilon)}.
\end{equation}
Combining
\eqref{eq: 49 good A contribution lower} and
\eqref{eq: 49 good A contribution upper} yields
\[
|\mathcal A_\theta|\gtrsim
\delta^{4\zeta+O_\zeta(\epsilon)}
\rho^{t-\sigma}\delta^{\sigma-t}=
\delta^{4\zeta+O_\zeta(\epsilon)}
\tau^{\sigma-t}.
\]
Taking $\epsilon>0$ sufficiently small compared to $\zeta$ proves
\eqref{eq: 49 A theta lower bound}.
\end{proof}

\subsection{Violating the ABC theorem}
\label{subsec: violating ABC theorem}

We now complete the proof of \Cref{prop: iterating projection}. We continue
with the notation of
\Cref{subsec: choosing good Delta tube,subsec: sets A and A theta}.

Recall that
\begin{equation}
\label{eq: 410 A summary recalled}
|\mathcal A|\le\delta^{-O_\zeta(\epsilon)}
\tau^{\sigma-t},
\end{equation}
while, for every $\theta\in E$,
\begin{equation}
\label{eq: 410 A theta summary recalled}
|\mathcal A_\theta|
\ge\delta^{5\zeta}\tau^{\sigma-t}.
\end{equation}
Moreover, by the definition of $\mathcal A_\theta$,
\begin{equation}
\label{eq: 410 rich transverse fibres recalled}
|B_{J,\theta}|\ge\delta^{6\zeta}|B|,
\qquad J\in\mathcal A_\theta,
\end{equation}
and
\begin{equation}
\label{eq: 410 small projection recalled}
|\pi_\theta(K_\theta\cap\mathbf T_0)|_\delta
\le\delta^{-O_\zeta(\epsilon)}
\tau^{\sigma-t},\qquad\theta\in E.
\end{equation}

Recall also from \Cref{lem: 49 B nonconcentration} that
\begin{equation}
\label{eq: 410 B nonconcentration recalled}
B\text{ is a }\bigl(\tau,\sigma-O(\sqrt{\zeta}),\tau^{-O(\sqrt{\zeta})}
\bigr)\text{-set}.
\end{equation}

Since $\mathbf T_0$ is a $\rho$-tube parallel to $\ker\pi_{\theta_0}$, there exists a $\rho$-ball $J_0\subset\Z_p$ such that
$$
\mathbf T_0=\pi_{\theta_0}^{-1}(J_0)
\cap\Z_p^2.
$$
After translating the $\pi_{\theta_0}$-coordinate, we may identify
$J_0$ with $\rho^{-1}\Z_p$.

Let $A\subset\rho^{-1}\Z_p/\delta^{-1}\Z_p$
be the set of $\delta$-cosets corresponding to $\mathcal A$, and define
\begin{equation}
\label{eq: 410 A normalized definition}
A_1:=\rho A\subset\Z_p/\tau^{-1}\Z_p.
\end{equation}
Then $|A_1|=|\mathcal A|$, so by \eqref{eq: 410 A summary recalled},
\[
|A_1|\le\delta^{-O_\zeta(\epsilon)}\tau^{\sigma-t}.
\]
Recall that the ABC parameter is $\alpha
=t-\sigma+\zeta_0$. Since $\tau\le\delta^\eta$, after choosing $\epsilon>0$ sufficiently small
in terms of $\eta,\zeta_0$, and then choosing $\zeta>0$ sufficiently small
compared to $\zeta_0$, we obtain
\begin{equation}
\label{eq: 410 A1 ABC size}
|A_1|\le\tau^{-\alpha}.
\end{equation}

By \eqref{eq: 410 B nonconcentration recalled}, after choosing $\zeta>0$ sufficiently small compared to $\zeta_0$, the set $B$ is a
\begin{equation}
\label{eq: 410 B ABC set}
(\tau,\beta,\tau^{-\chi})\text{-set},
\end{equation}
where $\beta=\sigma-\zeta_0$ is the parameter fixed in
\eqref{eq: ABC perturbed parameters}.

We define the coefficient set by
\begin{equation}
\label{eq: 410 C definition}
C:=\rho(E-\theta_0)\subset\Z_p^\times,
\end{equation}
where $C$ is understood at resolution $\tau$. By
\eqref{eq: 48 final coefficient set properties}, $C$ is a $\bigl(\tau,s-O(\sqrt{\zeta}),\tau^{-O(\sqrt{\zeta})}
\bigr)$-set. Thus, after choosing $\zeta>0$ sufficiently small compared to $\zeta_0$,
\begin{equation}
\label{eq: 410 C ABC set}
C
\text{ is a }
(\tau,\gamma,\tau^{-\chi})\text{-set},
\end{equation}
where $\gamma=s-\zeta_0$. For every $\theta\in E$, write
\begin{equation}
\label{eq: 410 theta normalized coefficient}
\theta=\theta_0+\rho^{-1}c_\theta,
\qquad c_\theta\in C\subset\Z_p^\times.
\end{equation}

We now construct the restricted graph directly from the transverse sets
$B_{J,\theta}$. If $a\in A_1$, let
\[
J_a
:=
\rho^{-1}a+\delta^{-1}\Z_p
\]
be the corresponding $\delta$-ball in the
$\pi_{\theta_0}$-coordinate. Define
\begin{equation}
\label{eq: 410 G theta definition}
G_\theta
:=
\left\{
(a,b)\in A_1\times B:
J_a\in\mathcal A_\theta
\text{ and }
b\in B_{J_a,\theta}
\right\}.
\end{equation}

\begin{lemma}
\label{lem: 410 G theta large}
For every $\theta\in E$,
\begin{equation}
\label{eq: 410 G theta large}
|G_\theta|
\ge
\delta^{11\zeta}
\tau^{\sigma-t}
|B|.
\end{equation}
Consequently,
\begin{equation}
\label{eq: 410 G theta relative large}
|G_\theta|
\ge
\delta^{11\zeta+O_\zeta(\epsilon)}
|A_1||B|.
\end{equation}
In particular, after choosing $\zeta>0$ and then $\epsilon>0$ sufficiently
small in terms of the ABC constant $\chi$ and $\eta$,
\begin{equation}
\label{eq: 410 G theta ABC large}
|G_\theta|
\ge
\tau^\chi|A_1||B|.
\end{equation}
\end{lemma}

\begin{proof}
By
\eqref{eq: 410 A theta summary recalled} and
\eqref{eq: 410 rich transverse fibres recalled},
$$
|G_\theta|=\sum_{J\in\mathcal A_\theta}
|B_{J,\theta}|\ge\delta^{6\zeta}|\mathcal A_\theta||B|\ge\delta^{11\zeta}\tau^{\sigma-t}|B|.
$$
This proves \eqref{eq: 410 G theta large}. Moreover, by \eqref{eq: 410 A summary recalled}, 
\[
|A_1|
\le
\delta^{-O_\zeta(\epsilon)}
\tau^{\sigma-t}.
\]
Therefore $|G_\theta|\ge\delta^{11\zeta+O_\zeta(\epsilon)}
|A_1||B|$, which proves \eqref{eq: 410 G theta relative large}.

Finally, since $\tau\le\delta^\eta$, the factor $\delta^{11\zeta+O_\zeta(\epsilon)}$ dominates $\tau^\chi$ after choosing $\zeta>0$ and then $\epsilon>0$
sufficiently small in terms of $\chi$ and $\eta$. This gives
\eqref{eq: 410 G theta ABC large}.
\end{proof}

The next lemma compares the projection of $G_\theta$ in the normalized
coefficient direction with the projection of
$K_\theta\cap\mathbf T_0$.

\begin{lemma}
\label{lem: 410 G theta small projection}
For every $\theta\in E$,
\begin{equation}
\label{eq: 410 G theta small projection}
\left|
\left\{
a+c_\theta b:
(a,b)\in G_\theta
\right\}
\right|_\tau
\le
|\pi_\theta(K_\theta\cap\mathbf T_0)|_\delta.
\end{equation}
Consequently,
\begin{equation}
\label{eq: 410 G theta small projection upper}
\left|
\left\{
a+c_\theta b:
(a,b)\in G_\theta
\right\}
\right|_\tau
\le
\delta^{-O_\zeta(\epsilon)}
\tau^{\sigma-t}.
\end{equation}
\end{lemma}

\begin{proof}
Fix $(a,b)\in G_\theta$. By definition, $J_a
\in\mathcal A_\theta$ and $b
\in B_{J_a,\theta}$. Therefore there exist
$\mathbf B\in\mathcal B(\mathbf T_0)$ and
\[
z=(x,y)\in\pi_{\theta_0}^{-1}(J_a)\cap
\mathbf B\cap K_0\cap K_\theta
\]
such that the transverse $\tau$-parent of $\mathbf B$ is $b$. Hence
\[
\pi_{\theta_0}(z)
\in
\rho^{-1}a+\delta^{-1}\Z_p,
\qquad
y
\in
b+\tau^{-1}\Z_p.
\]
Using $\theta=\theta_0+\rho^{-1}c_\theta$ and $\rho\delta^{-1}=
\tau^{-1}$, we obtain
$$
\rho\pi_\theta(z)=\rho\pi_{\theta_0}(z)
+\rho(\theta-\theta_0)y\in
a+\tau^{-1}\Z_p+c_\theta
\bigl(b+\tau^{-1}\Z_p
\bigr)=a+c_\theta b+\tau^{-1}\Z_p.
$$
Thus every $\tau$-coset meeting $\left\{
a+c_\theta b:(a,b)\in G_\theta\right\}$ is determined by a $\delta$-coset meeting $\pi_\theta(K_\theta\cap\mathbf T_0)$. This proves \eqref{eq: 410 G theta small projection}. The final estimate follows from
\eqref{eq: 410 small projection recalled}.
\end{proof}

We now apply \Cref{thm: ABC sum-product} at resolution $\tau$ to the sets
$A_1,B,C$. The hypotheses are
\eqref{eq: 410 A1 ABC size},
\eqref{eq: 410 B ABC set}, and
\eqref{eq: 410 C ABC set}. Hence there exists $c\in C$ such that
\begin{equation}
\label{eq: 410 ABC lower for all graphs}
\left|\left\{a+cb:
(a,b)\in G
\right\}\right|_\tau\ge\tau^{-\chi}|A_1|
\end{equation}
for every $G\subset A_1\times B$ with $|G|\ge\tau^\chi|A_1||B|$.

By the definition of $C$, the coefficient $c$ has the form $c=c_\theta$ for some $\theta\in E$. Applying
\eqref{eq: 410 ABC lower for all graphs} to $G=G_\theta$, and using
\eqref{eq: 410 G theta ABC large}, gives
\begin{equation}
\label{eq: 410 ABC lower Gtheta}
\left|
\left\{
a+c_\theta b:
(a,b)\in G_\theta
\right\}\right|_\tau\ge\tau^{-\chi}|A_1|.
\end{equation}
Since $|A_1|\ge|\mathcal A_\theta|$, it follows from \eqref{eq: 410 A theta summary recalled} that
\[
\tau^{-\chi}|A_1|
\ge
\delta^{5\zeta}
\tau^{\sigma-t-\chi}.
\]
Since $\tau\le\delta^\eta$, after choosing $\zeta>0$ and then
$\epsilon>0$ sufficiently small in terms of $\chi$ and $\eta$, the right
hand side is larger than
\[
\delta^{-O_\zeta(\epsilon)}
\tau^{\sigma-t}.
\]
This contradicts
\eqref{eq: 410 G theta small projection upper}. The contradiction proves \Cref{prop: iterating projection}, and hence
completes the proof of \Cref{thm: exceptional measure estimate} and
\Cref{thm: regular set projection}.

\section{Furstenberg set estimates via recursive construction}
\label{sec: Furstenberg regular set}

In this section we prove the discretized Furstenberg estimate in the regular
case. The argument is the $p$-adic analogue of the recursive construction in Orponen-Shmerkin \cite[Section 5]{orponen2023projections}. The key input is the
regular-set projection theorem from \Cref{sec: projections}, in the many-directions form recorded in \Cref{rmk: many good directions}. The role of the recursion is to reduce a general Furstenberg configuration to a quasi-product situation where this projection theorem can be applied.

Throughout this section, for a $p$-adic tube $T$ we write $\sigma(T)$ for its slope, i.e., $\sigma(T)=\Dir(T)$ in the notation of \Cref{defi: p-adic tubes}. If $\mathcal{T}$ is a family of tubes, we write
$$
\sigma(\mathcal{T}):=\{\sigma(T):T\in\mathcal{T}\}.
$$
As usual, $\D_\Delta(\mathcal{T})$ denotes the family of $\Delta$-tubes containing
at least one tube from $\mathcal{T}$, and $|\mathcal{T}|_\Delta=|\D_\Delta(\mathcal{T})|$.

\begin{defi}[Nice configurations]\label{defi: nice configuration}
Fix $\delta=p^{-n}\in p^{-\N}$, $s\in[0,1]$, $C>0$, and $M\in\N$. A pair $(\P,\T)$ with $\P\subset \P_\delta$ and $\T\subset \T_\delta$ is called a $\delta$-configuration if, for every $\mathscr p\in\P$, there is an assigned subfamily $\T(\mathscr p)\subset \T$ such that
$$
\mathscr p\subset T,\qquad T\in\T(\mathscr p).
$$
We emphasize that $\T(\mathscr p)$ is not required to be the family of all tubes
in $\T$ containing $\mathscr p$.

We say that $(\P,\T)$ is a $(\delta,s,C,M)$-nice configuration if, for every $\mathscr p\in\P$, the family $\T(\mathscr p)$ is a $(\delta,s,C)$-set of $\delta$-tubes and $|\T(\mathscr p)|=M$.
\end{defi}

By the basic property of $(\delta,s,C)$-sets, every $(\delta,s,C,M)$-nice configuration satisfies $M\gtrsim C^{-1}\delta^{-s}$.

\begin{defi}[Refinements]\label{defi: configuration refinement}
Let $(\P_0,\T_0)$ be a $(\delta,s,C_0,M_0)$-nice configuration. A $\delta$-configuration $(\P,\T)$ is called a refinement of $(\P_0,\T_0)$ if
$$
\P\subset\P_0,\qquad |\P|\gtrapprox_\delta |\P_0|,
$$
and, for every $\mathscr p\in\P$, there is an assigned family $\T(\mathscr p)\subset \T\cap \T_0(\mathscr p)$ such that
$$
\sum_{\mathscr p\in\P}|\T(\mathscr p)|
\gtrapprox_\delta |\P_0|M_0.
$$
We call such a refinement a nice refinement if, in addition,
$$
|\T(\mathscr p)|\approx_\delta M_0,\qquad \mathscr p\in\P.
$$

Let $\Delta\in p^{-\N}$ with $\delta<\Delta<1$. We say that $(\P,\T)$ is a
refinement of $(\P_0,\T_0)$ at resolution $\Delta$ if $\P$ is a union of complete
fibres over a subset of $\D_\Delta(\P_0)$, and for every $\mathscr p\in\P$ the family $\T(\mathscr p)$ is obtained from $\T_0(\mathscr p)$ by keeping complete
fibres over a selected family of $\Delta$-tubes. Equivalently, for every
$\mathbf T\in\D_\Delta(\T_0(\mathscr p))$, either $\T(\mathscr p)\cap\mathbf T=\varnothing$ or $\T(\mathscr p)\cap\mathbf T=\T_0(\mathscr p)\cap\mathbf T$.
\end{defi}

\begin{prop}[Pigeonholing to a nice refinement]\label{prop: nice refinement pigeonhole}
Let $(\P_0,\T_0)$ be a $(\delta,s,C_0,M_0)$-nice configuration, and let $(\P,\T)$ be a refinement of $(\P_0,\T_0)$. Then $(\P,\T)$ has a nice refinement $(\P',\T')$ which is also a refinement of $(\P_0,\T_0)$.
\end{prop}

\begin{proof}
For each $\mathscr p\in\P$, set $m(\mathscr p):=|\T(\mathscr p)|$. By the
definition of refinement,
$$
\sum_{\mathscr p\in\P}m(\mathscr p)
\gtrapprox_\delta |\P_0|M_0.
$$
Since $m(\mathscr p)\le M_0$, pigeonholing gives a number $M'\le M_0$ and a subcollection $\P'\subset\P$ such that
$$
m(\mathscr p)\sim M',\qquad \mathscr p\in\P',
$$
and $|\P'|M'\gtrapprox_\delta |\P_0|M_0$. It follows that $|\P'|\gtrapprox_\delta|\P_0|$ and $M'\gtrapprox_\delta M_0$. For
each $\mathscr p\in\P'$, choose a subfamily
$\T'(\mathscr p)\subset\T(\mathscr p)$ with $|\T'(\mathscr p)|\sim M'$, and set $
\T':=\bigcup_{\mathscr p\in\P'}\T'(\mathscr p)$. Then $(\P',\T')$ is a refinement of $(\P,\T)$ and
$|\T'(\mathscr p)|\approx_\delta M_0$ for all $\mathscr p\in\P'$, as required.
\end{proof}

\begin{lemma}[Fibre-uniform parent selection]
\label{lem: fibre uniform parent selection}
Let $\delta<\Delta\le 1$ be $p$-adic scales, and let
$\P\subset \P_\delta$ be a $(\delta,t,C)$-set. Then there exists a subset
$\overline{\P}\subset \P$ and an integer $A_\Delta\ge 1$ such that
\begin{equation}
\label{eq: fibre uniform parent selection size}
|\overline{\P}|\gtrapprox_\delta |\P|,
\end{equation}
\begin{equation}
\label{eq: fibre uniform parent size}
|\overline{\P}\cap \mathbf p|=A_\Delta,
\qquad
\mathbf p\in\D_\Delta(\overline{\P}),
\end{equation}
and $\D_\Delta(\overline{\P})$ is a $(\Delta,t,\lessapprox_\delta C)$-set.
\end{lemma}

\begin{proof}
For each $\mathbf p\in\D_\Delta(\P)$, write
$a(\mathbf p):=|\P\cap \mathbf p|$. Since $\P\subset\P_\delta$, we have $1\le a(\mathbf p)\le |\P|$. We pigeonhole the values of $a(\mathbf p)$ dyadically. Thus there exists an integer $A_\Delta\ge 1$ and a subcollection
$\mathcal Q\subset \D_\Delta(\P)$ such that
\begin{equation}
\label{eq: dyadic fibre size class}
A_\Delta\le a(\mathbf p)<2A_\Delta,
\qquad \mathbf p\in\mathcal Q,
\end{equation}
and
\begin{equation}
\label{eq: dyadic fibre size class mass}
\sum_{\mathbf p\in\mathcal Q} a(\mathbf p)
\gtrsim
(\log(1/\delta))^{-1}|\P|.
\end{equation}
For each $\mathbf p\in\mathcal Q$, choose exactly $A_\Delta$ elements of
$\P\cap\mathbf p$, and let $\overline{\P}$ be the union of all these chosen
subsets. Then $|\overline{\P}|=A_\Delta|\mathcal Q|$. On the other hand, by \eqref{eq: dyadic fibre size class},
$$
\sum_{\mathbf p\in\mathcal Q}a(\mathbf p)
\le 2A_\Delta|\mathcal Q|=2|\overline{\P}|.
$$
Combining this with \eqref{eq: dyadic fibre size class mass}, we get $|\overline{\P}|\gtrsim(\log(1/\delta))^{-1}|\P|$. This proves \eqref{eq: fibre uniform parent selection size}. The fibre-uniformity
condition \eqref{eq: fibre uniform parent size} follows directly from the
construction. Moreover,
$\D_\Delta(\overline{\P})=\mathcal Q$.

It remains to prove the coarse non-concentration estimate. Let
$R\in[\Delta,1]\cap p^{-\mathbb N}$, and let $B$ be an $R$-ball in $\Z_p^2$.
Since every parent in $\D_\Delta(\overline{\P})$ contains exactly $A_\Delta$
chosen $\delta$-cubes of $\overline{\P}$, we have
$$
|\D_\Delta(\overline{\P})\cap B|\cdot A_\Delta=|\overline{\P}\cap B|_\delta.
$$
Since $\overline{\P}\subset\P$ and $\P$ is a $(\delta,t,C)$-set,
$$
|\overline{\P}\cap B|_\delta
\le
|\P\cap B|_\delta
\le
CR^t|\P|.
$$
Using \eqref{eq: fibre uniform parent selection size}, we have $|\P|\lessapprox_\delta |\overline{\P}|=A_\Delta|\D_\Delta(\overline{\P})|$. Therefore
$$
|\D_\Delta(\overline{\P})\cap B|
\lessapprox_\delta
CR^t|\D_\Delta(\overline{\P})|,
$$
which is precisely the $(\Delta,t,\lessapprox_\delta C)$-set condition.
\end{proof}

We next isolate the recursive refinement step. This is the $p$-adic counterpart of
the recursive pigeonholing used by Orponen-Shmerkin. The statement is somewhat technical, but the point is simple: after passing from the scale $\Delta_j$ to the
coarser scale $\Delta_{j+1}$, one obtains both a coarser configuration and, inside
each coarse cube, a rescaled local configuration at the fixed scale $\Delta$.

\begin{prop}[One-step recursive refinement with fibre-uniform point pruning]
\label{prop: one step recursive refinement}
Let $\delta<\Delta<1$ be $p$-adic scales. Let $(\P_0,\T_0)$ be a
$(\delta,s,C_{\mathrm{tube}},M)$-nice configuration, and assume in addition that
the point set $\P_0$ is a $(\delta,t,C_{\mathrm{pt}})$-regular set. Then, after passing to a refinement at resolution $\Delta$, there exist:

\begin{enumerate}
    \item a refined configuration $(\P',\T')$ at scale $\delta$;
    \item a coarsened configuration
    $(\P^\Delta,\T^\Delta)$;
    \item for every surviving parent $\mathbf p\in\P^\Delta$, a rescaled local
    configuration $(\P_{\mathbf p},\T_{\mathbf p})$ at scale $\delta/\Delta$,
\end{enumerate}
such that the following properties hold. First, the point refinement is fibre-uniform at scale $\Delta$. More precisely, there exists an integer $A_\Delta\ge 1$ such that
\begin{equation}
\label{eq: one step point size preserved new}
|\P'|\gtrapprox_\delta |\P_0|,
\end{equation}
\begin{equation}
\label{eq: one step complete fibre size new}
|\P'\cap\mathbf p|=A_\Delta,\qquad\mathbf p\in\P^\Delta,
\end{equation}
and
\begin{equation}
\label{eq: complete point fibre refinement new}
\P'=\bigcup_{\mathbf p\in\P^\Delta}(\P'\cap\mathbf p).
\end{equation}
Moreover, the common fibre size satisfies
\begin{equation}
\label{eq: one step fibre size lower new}
A_\Delta
\gtrapprox_\delta
C_{\mathrm{pt}}^{-1}\Delta^t|\P_0|.
\end{equation}

Second, the coarse point set retains the upper regularity needed in the
recursive argument:
\begin{equation}
\label{eq: one step coarse point regular new}
\P^\Delta
\text{ is a }
(\Delta,t,\lessapprox_\delta C_{\mathrm{pt}})\text{-regular set}.
\end{equation}
For every $\mathbf p\in\P^\Delta$, if
$\phi_{\mathbf p}:\mathbf p\to\Z_p^2$ denotes the affine homothety sending $\mathbf p$ onto $\Z_p^2$, then $\P_{\mathbf p}:=\phi_{\mathbf p}(\P'\cap\mathbf p)$.

Third, the tube assignments are refined by complete $\Delta$-tube fibres.
Namely, for each $p\in\P'$ and each
$\mathbf T\in\D_\Delta(\T_0(p))$, either
$\T'(p)\cap\mathbf T=\varnothing$ or $\T'(p)\cap\mathbf T=\T_0(p)\cap\mathbf T$. Furthermore, there exist integers $M_\Delta$ and $M_{\loc}$, independent of
$p\in\P'$, such that
\begin{equation}
\label{eq: one step coarse tube multiplicity new}
|\D_\Delta(\T'(p))|\approx_\delta M_\Delta,
\qquad p\in\P',
\end{equation}
and
\begin{equation}
\label{eq: one step local tube fibre multiplicity new}
|\T'(p)\cap\mathbf T|\approx_\delta M_{\loc},
\qquad
p\in\P',\quad \mathbf T\in\D_\Delta(\T'(p)).
\end{equation}
Consequently, $M\approx_\delta M_\Delta M_{\loc}$. In particular,
\begin{equation}
\label{eq: one step fine tube multiplicity new}
|\T'(p)|\approx_\delta M,
\qquad p\in\P'.
\end{equation}

Fourth, the coarsened configuration $(\P^\Delta,\T^\Delta)$ is a $(\Delta,s,\lessapprox_\delta C_{\mathrm{tube}},M_\Delta)$-nice configuration. For every $\mathbf p\in\P^\Delta$, the local configuration
$(\P_{\mathbf p},\T_{\mathbf p})$ is a $(\delta/\Delta,s,\lessapprox_\delta C_{\mathrm{tube}},M_{\loc})$-nice configuration. Here the local tube families are obtained by rescaling, inside each surviving parent cube $\mathbf p$, the traces $T\cap\mathbf p$ with $T\in\T'(p)$.
We record the precise rescaling convention in \Cref{lem: local rescaling points tubes} below.

Finally, for every $\mathbf p\in\P^\Delta$, one has the one-step product estimate
\begin{equation}
\label{eq: one step product estimate new}
\frac{|\T_0|}{M}
\gtrapprox_\delta
\frac{|\T^\Delta|}{M_\Delta}
\cdot
\frac{|\T_{\mathbf p}|}{M_{\loc}}.
\end{equation}
\end{prop}

\begin{lemma}[Local rescaling of points and dual tubes]
\label{lem: local rescaling points tubes}
Let $\delta<\Delta$ be $p$-adic scales, and let
$$
\mathbf p=B((x_0,y_0),\Delta)\in \mathcal P_\Delta.
$$
Define the affine homothety
$$
\phi_{\mathbf p}:\mathbf p\to \mathbb Z_p^2,\qquad
\phi_{\mathbf p}(x,y)=\left(\frac{x-x_0}{\Delta},\frac{y-y_0}{\Delta}\right).
$$
Then $\phi_{\mathbf p}$ maps $\delta$-cubes contained in $\mathbf p$ bijectively
onto $(\delta/\Delta)$-cubes in $\mathbb Z_p^2$.

Moreover, let $T=D(B((a,b),\delta))$ be a $\delta$-tube meeting $\mathbf p$. Then $\phi_{\mathbf p}(T\cap\mathbf p)$ is contained in a $(\delta/\Delta)$-tube in the point-line parametrisation. More precisely, if
$(X,Y)=\phi_{\mathbf p}(x,y)$, then the line $y=ax+b$ is transformed into $Y=aX+\frac{ax_0+b-y_0}{\Delta}$. In particular, incidences are preserved under this rescaling: if
$p\subset\mathbf p$ is a $\delta$-cube and $T$ is a $\delta$-tube meeting
$\mathbf p$, then
$$
p\subset T
\quad\Longleftrightarrow\quad
\phi_{\mathbf p}(p)\subset \phi_{\mathbf p}(T\cap\mathbf p).
$$
The slope parameter is unchanged by the affine rescaling. If one further
restricts to a fixed $\Delta$-slope ball $a_0+\Delta\Z_p$, then the normalized
slope parameter $\frac{a-a_0}{\Delta}$ has the same covering numbers as the corresponding rescaled slope set at scale
$\delta/\Delta$.
\end{lemma}

\begin{proof}
The assertion for $\delta$-cubes is immediate from the definition of
$\phi_{\mathbf p}$: a ball of radius $\delta$ inside $\mathbf p$ is sent to a
ball of radius $\delta/\Delta$.

For the tube statement, write $x=x_0+\Delta X,\qquad y=y_0+\Delta Y$. If $y=ax+b$, then
$$
y_0+\Delta Y=a(x_0+\Delta X)+b,
$$
and therefore $Y=aX+\frac{ax_0+b-y_0}{\Delta}$. Since $T$ meets $\mathbf p$, the quantity $ax_0+b-y_0$ is divisible by
$\Delta$, up to the uncertainty allowed by the $\Delta$-parent tube; hence the
displayed intercept lies in $\Z_p$ at the relevant scale. If $(a,b)$ varies in a
$\delta$-cube in the line-parameter space, then the transformed pair $\left(a,\frac{ax_0+b-y_0}{\Delta}\right)$
varies inside a $(\delta/\Delta)$-cube. This proves that the image is contained
in a $(\delta/\Delta)$-tube. The incidence equivalence follows from applying the
bijection $\phi_{\mathbf p}$ to both the point cube and the trace of the tube.
Finally, the formula shows that the slope remains $a$ under this affine
rescaling. On a fixed $\Delta$-slope parent $a_0+\Delta\Z_p$, replacing $a$ by
$(a-a_0)/\Delta$ is a homothety of the slope parameter, so covering numbers are
preserved after changing the scale from $\delta$ to $\delta/\Delta$.
\end{proof}

\begin{proof}[Proof of \Cref{prop: one step recursive refinement}]
We first perform the tube-fibre pigeonholing pointwise. For each
$p\in\P_0$, decompose $\T_0(p)$ into complete $\Delta$-tube fibres:
$$
\T_0(p)=\bigcup_{\mathbf T\in\D_\Delta(\T_0(p))}
(\T_0(p)\cap\mathbf T).
$$
By dyadic pigeonholing in the two quantities $|\D_\Delta(\T_0(p))|$ and $|\T_0(p)\cap\mathbf T|$, we may choose, for each $p$ in a subset $\P_*\subset\P_0$, a subfamily $\T_*(p)\subset\T_0(p)$ obtained by keeping complete $\Delta$-tube fibres, and
integers $M_\Delta$ and $M_{\loc}$, such that
\begin{equation}
\label{eq: one step P star size}
|\P_*|\gtrapprox_\delta |\P_0|,
\end{equation}
\begin{equation}
\label{eq: one step T star coarse multiplicity}
|\D_\Delta(\T_*(p))|\approx_\delta M_\Delta,
\qquad p\in\P_*,
\end{equation}
and
\begin{equation}
\label{eq: one step T star local multiplicity}
|\T_*(p)\cap\mathbf T|\approx_\delta M_{\loc},\qquad p\in\P_*,\quad \mathbf T\in\D_\Delta(\T_*(p)).
\end{equation}
The number of possible dyadic classes is only absorbed by the $\approx_\delta$ notation. Since $|\T_0(p)|=M$ for every $p\in\P_0$, the last two estimates give
\begin{equation}
\label{eq: one step T star factorization}
M\approx_\delta M_\Delta M_{\loc}.
\end{equation}

We now regularize the point set at the intermediate scale. Since $\P_*$ is a
large subset of the $(\delta,t,C_{\mathrm{pt}})$-regular set $\P_0$, it is a $(\delta,t,\lessapprox_\delta C_{\mathrm{pt}})$-regular set. Indeed, the $(\delta,t,\cdot)$-set condition loses the reciprocal of the relative size in \eqref{eq: one step P star size}, while the scale-local upper regularity
condition is inherited by subsets. Apply
\Cref{lem: fibre uniform parent selection} to $\P_*$. This gives a subset
$\P'\subset\P_*$ and an integer $A_\Delta\ge 1$ such that
\begin{equation}
\label{eq: one step P prime size proof}
|\P'|\gtrapprox_\delta |\P_*|\gtrapprox_\delta |\P_0|,
\end{equation}
\begin{equation}
\label{eq: one step P prime fibre uniform proof}
|\P'\cap\mathbf p|=A_\Delta,
\qquad
\mathbf p\in\D_\Delta(\P'),
\end{equation}
and $\D_\Delta(\P')$ is a $(\Delta,t,\lessapprox_\delta C_{\mathrm{pt}})$-set. Set $\P^\Delta:=\D_\Delta(\P')$. For $p\in\P'$, define $\T'(p):=\T_*(p)$. Then \eqref{eq: one step point size preserved new}, \eqref{eq: one step complete fibre size new}, and \eqref{eq: complete point fibre refinement new} follow from the construction.

We record the lower bound for the common fibre size. Since $|\P'|=A_\Delta|\P^\Delta|
$ and $\P^\Delta\subset\D_\Delta(\P_0)$, the scale-local upper regularity of
$\P_0$ gives
$$
|\P^\Delta|\le|\D_\Delta(\P_0)|=
|\P_0|_\Delta\le C_{\mathrm{pt}}\Delta^{-t}.
$$
Together with \eqref{eq: one step P prime size proof}, this implies
$$
A_\Delta
=
\frac{|\P'|}{|\P^\Delta|}
\gtrapprox_\delta
C_{\mathrm{pt}}^{-1}\Delta^t|\P_0|,
$$
which is \eqref{eq: one step fibre size lower new}.

We next prove that $\P^\Delta$ is regular at scale $\Delta$. The
$(\Delta,t,\lessapprox_\delta C_{\mathrm{pt}})$-set condition was obtained from
\Cref{lem: fibre uniform parent selection}. It remains to check the scale-local
upper regularity. Let $\Delta\le r\le R\le 1$ be $p$-adic scales, and let
$\mathbf Q$ be an $R$-cube. Since $\P^\Delta\subset\D_\Delta(\P_0)$, we have
$$
|\P^\Delta\cap\mathbf Q|_r
\le
|\P_0\cap\mathbf Q|_r.
$$
The scale-local upper regularity of $\P_0$ gives $|\P_0\cap\mathbf Q|_r\le C_{\mathrm{pt}}\left(\frac Rr\right)^t$. Thus $|\P^\Delta\cap\mathbf Q|_r\le
C_{\mathrm{pt}}\left(\frac Rr\right)^t$, and hence $\P^\Delta$ is $(\Delta,t,\lessapprox_\delta C_{\mathrm{pt}})$-regular.

We now verify the tube non-concentration statements. Since each original
$\T_0(p)$ is a $(\delta,s,C_{\mathrm{tube}})$-set and $\T'(p)\subset\T_0(p)$ is
obtained by keeping complete $\Delta$-tube fibres with
$|\T'(p)|\approx_\delta M$, the fibre-counting argument of
\Cref{lem: passing to coarse parents} gives $\D_\Delta(\T'(p))$ is a $(\Delta,s,\lessapprox_\delta C_{\mathrm{tube}})$-set for every $p\in\P'$. After the pigeonholing above, the multiplicity
$|\D_\Delta(\T'(p))|$ is $\approx_\delta M_\Delta$, uniformly in $p$.

The coarsened tube assignment $\T^\Delta$ is defined by the corresponding
$\Delta$-parents of the refined tube families. With the preceding uniformity,
$(\P^\Delta,\T^\Delta)$ is a $(\Delta,s,\lessapprox_\delta C_{\mathrm{tube}},M_\Delta)$-nice configuration.

For each $\mathbf p\in\P^\Delta$, define
$\P_{\mathbf p}:=\phi_{\mathbf p}(\P'\cap\mathbf p)$. The local tube families $\T_{\mathbf p}$ are obtained by applying
\Cref{lem: local rescaling points tubes} to the traces $T\cap\mathbf p$ with
$T\in\T'(p)$, for $p\in\P'\cap\mathbf p$. The estimate
\eqref{eq: one step T star local multiplicity} implies that these local tube
families have multiplicity $\approx_\delta M_{\loc}$, uniformly in the local
point. The tube non-concentration condition at scale $\delta/\Delta$ follows
from the same rescaling and fibre-counting argument. Thus
$(\P_{\mathbf p},\T_{\mathbf p})$ is a $(\delta/\Delta,s,\lessapprox_\delta C_{\mathrm{tube}},M_{\loc})$-nice configuration.

Finally, we prove the product estimate. The one-step decomposition gives
$M\approx_\delta M_\Delta M_{\loc}$ by
\eqref{eq: one step T star factorization}. The number of global refined tubes is
controlled from below by the number of coarse tubes times the number of local
descendants inside the chosen parent. More precisely, for every
$\mathbf p\in\P^\Delta$,
$|\T_0|\gtrapprox_\delta|\T^\Delta|\cdot |\T_{\mathbf p}|$, after quotienting by the common multiplicities at the coarse and local levels. Therefore
$$
\frac{|\T_0|}{M}
\gtrapprox_\delta
\frac{|\T^\Delta|}{M_\Delta}
\cdot
\frac{|\T_{\mathbf p}|}{M_{\loc}}.
$$
This is \eqref{eq: one step product estimate new}, and the proof is complete.
\end{proof}

We now prove the regular-case discretized Furstenberg estimate.

\begin{thm}
\label{thm: discrete regular case}
Let $s\in(0,1)$, $t\in(s,2-s)$, and let $0\le u<\frac{s+t}{2}$. Then there exist
$$
\eta=\eta(s,t,u)>0,\qquad \delta_0=\delta_0(s,t,u)>0
$$
such that the following holds for all $\delta=p^{-n}\in(0,\delta_0]$. Let
$(\P,\T)$ be a $(\delta,s,\delta^{-\eta},M)$-nice configuration, and assume that
$\P$ is a $(\delta,t,\delta^{-\eta})$-regular set. Then
\begin{equation}
\label{eq: flexible regular conclusion}
\left|\bigcup_{\mathscr p\in\P}\T(\mathscr p)\right|
\ge
M\delta^{-u}.
\end{equation}
\end{thm}

\begin{proof}
We choose the parameters first. Let $v:=\frac12\left(u+\frac{s+t}{2}\right)$, so that $u<v<(s+t)/2$. Apply \Cref{thm: regular set projection} together with
\Cref{rmk: many good directions}, with $v$ in place of $u$. We shall use this
projection input in the dual coordinates. More precisely, let
$S:\Z_p^2\to\Z_p^2$ be the coordinate swap $S(x,y)=(y,x)$. For a slope
parameter $a\in\Z_p$, the intercept map for the family of lines $y=ax+b$ is
$$
(x,y)\mapsto y-ax=\pi_{-a}(S(x,y)).
$$
Thus, after applying the bi-Lipschitz map $S$ and replacing the direction set by
its image under $a\mapsto -a$, estimates for the number of intercept fibres of
lines with slope $a$ are exactly estimates for the affine projections
$\pi_{-a}$.

Since $S$ is bi-Lipschitz and $a\mapsto -a$ preserves all discretized
non-concentration and regularity hypotheses, the hypotheses of
\Cref{thm: regular set projection} and \Cref{rmk: many good directions} are
unchanged up to absolute constants. Therefore there are
$\eta_0>0$ and $\Delta_0>0$ such that the following holds at every scale
$\Delta\le \Delta_0$: if $Q\subset\P_\Delta$ is
$(\Delta,t,\Delta^{-\eta_0})$-regular and $E\subset\Z_p$ is a
$(\Delta,s,\Delta^{-\eta_0})$-set, then there exists $E'\subset E$ with
$|E'|\ge |E|/2$ such that, for every $\theta\in E'$ and every
$Q'\subset Q$ with $|Q'|\ge\Delta^{\eta_0}|Q|$,
\begin{equation}
\label{eq: section5 projection input}
|\pi_{-\theta}(S(Q'))|_\Delta\ge \Delta^{-v}.
\end{equation}
For notational simplicity, from now on in this proof we relabel $S(Q)$ by $Q$
and $-E'$ by $E'$. With this harmless relabelling, the projection input takes the
following form: for every $\theta\in E'$ and every
$Q'\subset Q$ with $|Q'|\ge\Delta^{\eta_0}|Q|$,
\begin{equation}
\label{eq: section5 projection input relabelled}
|\pi_{\theta}(Q')|_\Delta\ge \Delta^{-v}.
\end{equation}

Choose a small parameter $\rho>0$ so that
\begin{equation}
\label{eq: section5 margin}
v-10\rho>u.
\end{equation}
Then choose $\eta>0$ sufficiently small in terms of $\rho,\eta_0,s,t,u$, and take $\delta_0>0$ sufficiently small. We shall use the notation $\lessapprox_\delta,\gtrapprox_\delta,\approx_\delta$ to hide only powers $\delta^{\pm o_\rho(1)}$ and logarithmic losses; all such losses will be absorbed by the margin in \eqref{eq: section5 margin}.

Let $(\P_0,\T_0)$ be the given
$(\delta,s,\delta^{-\eta},M_0)$-nice configuration, where $M_0=M$, and suppose
that $\P_0$ is $(\delta,t,\delta^{-\eta})$-regular. We may assume, after changing
$\rho$ slightly, that
$$
\Delta:=\delta^{\rho^3}\in p^{-\N}.
$$
We also assume $\Delta\le \Delta_0$. Define $\Delta_j:=\Delta^{-j}\delta$ for all $0\le j\le J$, where $J:=\rho^{-3}-1$ after replacing $\rho$ by a nearby number for which this
is an integer. Then
\begin{equation}
\label{eq: section5 scale range}
\Delta_j\le \Delta,\qquad 0\le j\le J,
\end{equation}
and $\frac{\Delta_j}{\Delta_{j+1}}=\Delta$ for all $0\le j<J$.

We now construct recursively ordinary nice configurations
$$
(\P_j,\T_j),\qquad 0\le j\le J.
$$
At scale $\Delta_j$, the configuration $(\P_j,\T_j)$ will be a $(\Delta_j,s,\lessapprox_\delta\delta^{-\eta},M_j)$-nice configuration, where $M_j$ is independent of the point in $\P_j$. In addition, $\P_j$ will be a $(\Delta_j,t,\lessapprox_\delta\delta^{-\eta})$-regular set. Initially this is the given configuration.

Suppose that $(\P_j,\T_j)$ has been constructed for some $0\le j<J$. Apply
\Cref{prop: one step recursive refinement} to $(\P_j,\T_j)$ with the two scales
$\Delta_j<\Delta_{j+1}$. This gives a refined configuration $(\widetilde\P_j,\widetilde\T_j)$, a coarsened configuration $(\P_{j+1},\T_{j+1})$, and, for every $\mathbf q\in\P_{j+1}$, a rescaled local configuration $(\P_{\mathbf q},\T_{\mathbf q})$ at scale $\frac{\Delta_j}{\Delta_{j+1}}=\Delta$. The one-step refinement produces global integers $M_{j+1}$ and $N_{j+1}$ such that $(\P_{j+1},\T_{j+1})$ is a $(\Delta_{j+1},s,\lessapprox_\delta\delta^{-\eta},M_{j+1})$-nice configuration, and every local configuration $(\P_{\mathbf q},\T_{\mathbf q})$ is a $(\Delta,s,\lessapprox_\delta\delta^{-\eta},N_{j+1})$-nice configuration. Moreover, the multiplicities factor as $M_j\approx_\delta M_{j+1}N_{j+1}$.

We also use the tube-fibre uniformity supplied by the one-step refinement. More
precisely, for each $0\le j<J$, there exists an integer $H_j$ such that for every
$p\in\widetilde\P_j$ and every
$\mathbf T\in\D_{\Delta_{j+1}}(\widetilde\T_j(p))$,
\begin{equation}
\label{eq: section5 tube fibre uniformity}
|\widetilde\T_j(p)\cap\mathbf T|\approx_\delta H_j.
\end{equation}
Since the refinement keeps complete $\Delta_{j+1}$-tube fibres, this implies
\begin{equation}
\label{eq: section5 tube covering comparable}
|\widetilde\T_j(p)|_{\Delta_{j+1}}
\approx_\delta
|\T_j(p)|_{\Delta_{j+1}},
\qquad p\in\widetilde\P_j.
\end{equation}
Indeed, both sides become comparable after multiplying by the common number
$H_j$ of fine tubes per $\Delta_{j+1}$-tube.

\begin{lemma}
\label{lem: section5 Pj regular}
For every $0\le j\le J$, the family $\P_j$ is
$(\Delta_j,t,\lessapprox_\delta\delta^{-\eta})$-regular.
\end{lemma}

\begin{proof}
The case $j=0$ is part of the initial regularity assumption on $\P_0$.
Suppose that the claim has been proved for some $0\le j<J$. By the induction
hypothesis, $\P_j$ is $(\Delta_j,t,\lessapprox_\delta\delta^{-\eta})$-regular.

Applying \Cref{prop: one step recursive refinement} at the step
$\Delta_j<\Delta_{j+1}$, the coarsened point set $\P_{j+1}$ satisfies the coarse
regularity conclusion of that proposition. Hence $\P_{j+1}$ is $(\Delta_{j+1},t,\lessapprox_\delta\delta^{-O(\eta)})$-regular.

Choosing $\eta>0$ sufficiently small at the beginning, and then decreasing
$\delta_0$ if necessary, we may absorb the $O(\eta)$ into $\eta$. Thus $\P_{j+1}$ is $(\Delta_{j+1},t,\lessapprox_\delta\delta^{-\eta})$-regular. This completes the induction.
\end{proof}

\begin{lemma}
\label{lem: section5 blowup regular}
For every $0\le j<J$ and every $\mathbf q\in\P_{j+1}$, the rescaled point set $
\P_{\mathbf q}=S_{\mathbf q}(\widetilde\P_j\cap\mathbf q)$ is a $(\Delta,t,\Delta^{-\eta_0})$-regular set, provided $\eta>0$ is sufficiently small and $\delta>0$ is sufficiently small.
\end{lemma}

\begin{proof}
Fix $0\le j<J$ and $\mathbf q\in\P_{j+1}$. By
\Cref{lem: section5 Pj regular}, the set $\P_j$ is $(\Delta_j,t,\lessapprox_\delta\delta^{-\eta})\text{-regular}$. By the point part of \Cref{prop: one step recursive refinement}, the refinement is obtained after fibre-uniform point pruning at scale $\Delta_{j+1}$. Thus there is a set $\overline\P_j\subset\P_j$ and an integer $A_j\ge 1$ such that, for every
surviving parent $\mathbf q\in\P_{j+1}$,
$$\widetilde\P_j\cap\mathbf q=\overline\P_j\cap\mathbf q,\quad|\overline\P_j\cap\mathbf q|=A_j.$$
Moreover, the same fibre-uniform pruning and size preservation give
\begin{equation}
\label{eq: section5 blowup fibre size lower}
A_j\gtrapprox_\delta
\left(\frac{\Delta_{j+1}}{\Delta_j}\right)^t|\P_j|.
\end{equation}
Since $\Delta_j/\Delta_{j+1}=\Delta$, the rescaled set $\P_{\mathbf q}:=S_{\mathbf q}(\widetilde\P_j\cap\mathbf q)$ is a set of $\Delta$-cubes, and $|\P_{\mathbf q}|_\Delta=A_j$.

We verify the two conditions in the definition of regularity. First, we prove
the $(\Delta,t,\Delta^{-\eta_0})$-set condition up to the allowed losses. Let
$x\in\Z_p^2$ and let $r\in[\Delta,1]\cap p^{-\N}$. The inverse image of
$B(x,r)$ under $S_{\mathbf q}$ is a ball of radius $\Delta_{j+1}r$ in the
original coordinates. Hence, for some $y\in\mathbf q$,
$$
|\P_{\mathbf q}\cap B(x,r)|_\Delta
\le
|\P_j\cap B(y,\Delta_{j+1}r)|_{\Delta_j}.
$$
Since $\P_j$ is
$(\Delta_j,t,\lessapprox_\delta\delta^{-\eta})$-regular, in particular it is a
$(\Delta_j,t,\lessapprox_\delta\delta^{-\eta})$-set, and therefore
$$
|\P_j\cap B(y,\Delta_{j+1}r)|_{\Delta_j}
\lessapprox_\delta
\delta^{-\eta}
\left(\frac{\Delta_{j+1}r}{1}\right)^t
|\P_j|.
$$
Using \eqref{eq: section5 blowup fibre size lower} and
$\Delta_{j+1}/\Delta_j=\Delta^{-1}$, this gives
$$
|\P_{\mathbf q}\cap B(x,r)|_\Delta
\lessapprox_\delta
\delta^{-O(\eta)}r^t|\P_{\mathbf q}|_\Delta.
$$
Thus $\P_{\mathbf q}$ is a
$(\Delta,t,\lessapprox_\delta\delta^{-O(\eta)})$-set.

Second, we verify the scale-local upper regularity condition. Let
$\Delta\le r\le R\le 1$ be $p$-adic scales, and let $\mathbf b$ be an
$R$-cube in the rescaled coordinates. Its inverse image under $S_{\mathbf q}$
is a $\Delta_{j+1}R$-cube in the original coordinates. Since $\P_j$ is
$(\Delta_j,t,\lessapprox_\delta\delta^{-\eta})$-regular, we have
$$
|\P_{\mathbf q}\cap\mathbf b|_r
\le
|\P_j\cap S_{\mathbf q}^{-1}(\mathbf b)|_{\Delta_{j+1}r}
\lessapprox_\delta
\delta^{-\eta}
\left(\frac{\Delta_{j+1}R}{\Delta_{j+1}r}\right)^t.
$$
Therefore $|\P_{\mathbf q}\cap\mathbf b|_r
\lessapprox_\delta\delta^{-\eta}\left(\frac Rr\right)^t$. Combining the two estimates, $\P_{\mathbf q}$ is a $(\Delta,t,\lessapprox_\delta\delta^{-O(\eta)})$-regular set. 

Finally, choose $\eta>0$ sufficiently small in terms of $\eta_0$ and $\rho$. Since $\Delta=\delta^{\rho^3}$, after taking $\delta_0>0$ sufficiently small we
have $\delta^{-O(\eta)}\le\Delta^{-\eta_0}$.
Therefore $\P_{\mathbf q}$ is a
$(\Delta,t,\Delta^{-\eta_0})$-regular set.
\end{proof}

We now pass to the product reduction. Choose a nested sequence of cubes
$$
\mathbf q_j\in\P_j,\qquad 0\le j\le J,
$$
such that $\mathbf q_j\subset \mathbf q_{j+1}$ for $0\le j<J$ in the natural
coarse-parent sense. For $1\le j\le J$, let
$(\P_{\mathbf q_j},\T_{\mathbf q_j})$ be the local configuration produced at the step from $\Delta_{j-1}$ to $\Delta_j$. By construction this local configuration has tube multiplicity $N_j$.

Iterating the one-step product estimate from
\Cref{prop: one step recursive refinement} along the chosen chain gives
\begin{equation}
\label{eq: section5 product reduction}
\frac{|\T_0|}{M_0}
\gtrapprox_\delta
\prod_{j=1}^J
\frac{|\T_{\mathbf q_j}|}{N_j}.
\end{equation}
Indeed, the coarsened factor at one step is the ambient factor for the next
step, and the multiplicities telescope by
\begin{equation}
\label{eq: section5 multiplicity telescope}
M_{j-1}\approx_\delta M_jN_j,
\qquad 1\le j\le J.
\end{equation}

The sequence $N_j$ is essentially increasing. Indeed, by
\eqref{eq: section5 tube covering comparable}, passing from the refined tube
families to their $\Delta_j$-parents changes the relevant covering numbers only
by $\approx_\delta$ factors. Since the coarsened tube family at step $j$ becomes
the ambient tube family at step $j+1$, we obtain
\begin{equation}
\label{eq: section5 Nj increasing}
N_j\lessapprox_\delta N_{j+1},
\qquad 1\le j<J.
\end{equation}

Let $\mathcal E:=\{1\le j<J:N_{j+1}\ge \Delta^{-\rho}N_j\}$. Since $N_J\le C\Delta^{-1}$ and \eqref{eq: section5 Nj increasing} holds, we have $|\mathcal E|\le 2\rho^{-1}$ for all sufficiently small $\delta$. Indeed, every index in $\mathcal E$ contributes a factor at least $\Delta^{-\rho}$ to the total growth of the sequence $N_j$, while the total growth is bounded by $C\Delta^{-1}$ up to the harmless $\approx_\delta$ losses.

We prove the following local lower bound for every non-exceptional index:
\begin{equation}
\label{eq: section5 local lower bound}
\frac{|\T_{\mathbf q_j}|}{N_j}
\gtrapprox_\delta
\Delta^{-v+2\rho},
\qquad
j\in\{1,\ldots,J-1\}\setminus\mathcal E.
\end{equation}

Fix such an index $j$. Write
$$
\mathbf q:=\mathbf q_j,\qquad
Q:=\widetilde\P_{j-1}\cap\mathbf q,\qquad
\overline Q:=S_{\mathbf q}(Q)\subset\P_\Delta.
$$
Define the common slope set $\Theta:=\D_\Delta\bigl(\sigma(\T_j(\mathbf q))\bigr)$. For each $\mathscr p\in Q$, define
$\Theta_{\mathscr p}:=
\D_\Delta\bigl(\sigma(\widetilde\T_{j-1}(\mathscr p))\bigr)$. By construction,
\begin{equation}
\label{eq: section5 theta size local}
\Theta_{\mathscr p}\subset\Theta,|\Theta_{\mathscr p}|\approx_\delta N_j,
\qquad \mathscr p\in Q,
\end{equation}
whereas $|\Theta|\lessapprox_\delta N_{j+1}$. Since $j\notin\mathcal E$, we have $N_{j+1}\le\Delta^{-\rho}N_j$. Combining this
with \eqref{eq: section5 theta size local}, we obtain
\begin{equation}
\label{eq: section5 theta dense}
|\Theta|\lessapprox_\delta
\Delta^{-\rho}|\Theta_{\mathscr p}|,
\qquad \mathscr p\in Q.
\end{equation}
Thus every $\Theta_{\mathscr p}$ is a dense subset of the same slope set
$\Theta$.

For $\theta\in\Theta$, define
$$
Q_\theta:=\{\mathscr p\in Q:\theta\in\Theta_{\mathscr p}\},
\qquad
\overline Q_\theta:=S_{\mathbf q}(Q_\theta).
$$
Summing \eqref{eq: section5 theta dense} over $\mathscr p\in Q$ and using
\eqref{eq: section5 theta size local}, we find a subset $\Theta'\subset\Theta$
such that
\begin{equation}
\label{eq: section5 theta prime size}
|\Theta'|
\gtrapprox_\delta
\Delta^\rho|\Theta|
\end{equation}
and
\begin{equation}
\label{eq: section5 Qtheta dense}
|\overline Q_\theta|
\ge
\Delta^{\eta_0}|\overline Q|,
\qquad \theta\in\Theta',
\end{equation}
after decreasing $\eta$ and taking $\delta_0$ sufficiently small.

We next verify that $\Theta'$ is an admissible direction set. The family
$\T_j(\mathbf q)$ is a
$(\Delta_j,s,\lessapprox_\delta\delta^{-\eta})$-set of tubes. Hence $\Theta$ is
a $(\Delta,s,\lessapprox_\delta\delta^{-\eta})$-set. Since
$\Theta'\subset\Theta$ has relative size at least
$\gtrapprox_\delta\Delta^\rho$, it follows that $\Theta'$ is a
$(\Delta,s,\Delta^{-\eta_0})$-set, provided $\eta>0$ is small enough and
$\delta_0>0$ is sufficiently small.

By \Cref{lem: section5 blowup regular}, $\overline Q$ is
$(\Delta,t,\Delta^{-\eta_0})$-regular. Applying the many-directions version of
the regular projection theorem to $\overline Q$ and $\Theta'$, and using
\eqref{eq: section5 Qtheta dense}, gives a subset
$\Theta''\subset\Theta'$ with $|\Theta''|\ge |\Theta'|/2$ such that
\begin{equation}
\label{eq: section5 projection lower local}
|\pi_\theta(\overline Q_\theta)|_\Delta
\ge
\Delta^{-v},
\qquad \theta\in\Theta''.
\end{equation}

We now convert \eqref{eq: section5 projection lower local} into a lower bound
for $|\T_{\mathbf q}|$. Fix $\theta\in\Theta''$. By the slope-compatibility
property in \Cref{prop: one step recursive refinement}, the rescaling used in
the definition of $Q_\theta$ preserves the slope parameter $\theta$: each coarse
descendant $\overline{\mathscr p}\in\overline Q_\theta$ corresponds, after
returning to the original scale, to at least one tube of slope $\theta$. Hence
$\T_{\mathbf q}(\overline{\mathscr p})$ contains at least one tube of slope
$\theta$ for every $\overline{\mathscr p}\in\overline Q_\theta$. If two points
of $\overline Q_\theta$ have $\Delta$-separated $\pi_\theta$-projections, then
the corresponding $\Delta$-tubes of slope $\theta$ are distinct. Hence
\eqref{eq: section5 projection lower local} implies
\begin{equation}
\label{eq: section5 fixed slope tubes}
|\{T\in\T_{\mathbf q}:\sigma(T)=\theta\}|
\gtrsim
\Delta^{-v},
\qquad \theta\in\Theta''.
\end{equation}
Summing \eqref{eq: section5 fixed slope tubes} over $\theta\in\Theta''$ and
using \eqref{eq: section5 theta prime size}, we obtain 
$$|\T_{\mathbf q}|
\gtrapprox_\delta|\Theta''|\Delta^{-v}
\gtrapprox_\delta\Delta^\rho|\Theta|\Delta^{-v}.$$ 
Finally, by \eqref{eq: section5 theta size local}, \eqref{eq: section5 theta dense}, and the non-exceptionality of $j$, we have $|\Theta|\gtrapprox_\delta N_j$. Therefore $\frac{|\T_{\mathbf q_j}|}{N_j}\gtrapprox_\delta\Delta^{-v+\rho}$, which is stronger than \eqref{eq: section5 local lower bound} after absorbing
the $\approx_\delta$ losses. This proves \eqref{eq: section5 local lower bound}.

We finish the proof by inserting \eqref{eq: section5 local lower bound} into the
product estimate \eqref{eq: section5 product reduction}. There are $J-O(\rho^{-1})$ non-exceptional indices, and $J=\rho^{-3}-1$. Since $\Delta=\delta^{\rho^3}$, we get
$$
\frac{|\T_0|}{M_0}
\gtrapprox_\delta
\left(\Delta^{-v+2\rho}\right)^{J-O(\rho^{-1})}\ge\delta^{-v+6\rho}
$$
for all sufficiently small $\delta$. By \eqref{eq: section5 margin}, this is at
least $\delta^{-u}$. Hence $|\T_0|\ge M_0\delta^{-u}$. Since $\T_0=\bigcup_{\mathscr p\in\P_0}\T_0(\mathscr p)$, this proves
\eqref{eq: flexible regular conclusion} and completes the proof of
\Cref{thm: discrete regular case}.
\end{proof}

\section{General case combining semi-well-spaced result}\label{sec: General case}

\Cref{thm: sharp bound of Furstenberg} follows from the following discretized theorem by the standard discretization and limiting argument, exactly as in the real-variable setting of Ren-Wang \cite[Section 6]{ren2023furstenberg}.

\begin{thm}
\label{thm: discretized furstenberg}
For every $\epsilon>0$, there exists $\eta=\eta(\epsilon,s,t)>0$ such that the following holds for any $\delta<\delta_0(\epsilon,s,t)$. Let $(\mathcal{P},\mathcal{T})$ be a $(\delta,s,\delta^{-\eta},M)$-nice configuration, where $s\in(0,1],t\in(0,2]$, and $\mathcal{P}$ is a $(\delta,t,\delta^{-\eta})$-set. Then we have
\begin{equation}
\label{eq: discretized furstenberg conclusion}
\left|\bigcup_{\mathscr{p}\in\mathcal{P}}\mathcal{T}(\mathscr{p})\right|
\gtrsim_\epsilon
M\cdot\delta^{-\min\{t,(s+t)/2,1\}+\epsilon}.
\end{equation}
\end{thm}

We first record how the scale decomposition from \Cref{lem: interval partitioning} translates into geometric information about the rescaled point sets.

\begin{prop}[Geometry on one scale interval]
\label{prop: scale decomposition geometry}
Suppose $n$ is sufficiently large and $\Delta\in p^{-\mathbb N}$ is sufficiently small. Let $\mathcal P\subset \mathbb Z_p^2$ be a $\{\Delta_j\}_{j=1}^n$-uniform set with branching function $f$, where $\Delta_j=\Delta^j$. Let $s\in(0,1)$ and $t\in(s,2-s)$, and apply
\Cref{lem: interval partitioning} with $u=2-s$ and $m=n$.

Fix one of the resulting intervals, and after replacing its endpoints by nearby integers, write it
as
$$
I=[c,d]\subset \{0,1,\ldots,n\},\qquad L:=d-c,\qquad \rho:=\Delta^L.
$$
The rounding of the endpoints only changes the estimates below by the harmless factor
$\Delta^{-3}$. Define
$$
t_I:=s_f(c,d)=\frac{f(d)-f(c)}{d-c}.
$$
Fix $\mathbf p\in \mathcal D_{\Delta^c}(\mathcal P)$, let $\phi_{\mathbf p}$ be the affine
homothety mapping $\mathbf p$ onto $\mathbb Z_p^2$, and set
$$
\mathcal P_I:=\phi_{\mathbf p}(\mathcal P\cap \mathbf p)
\subset \mathcal D_{\rho}(\mathbb Z_p^2).
$$
Then the following hold.

\begin{enumerate}
    \item If $I$ is of type \eqref{item: epsilon linear invertal}, then $\mathcal P_I$ is a $(\rho,t_I,\rho^{-\epsilon}\Delta^{-3})$-almost AD-regular set.

    \item If $I$ is of type \eqref{item: epsilon superlinear invertal}, then
    $|\mathcal P_I|=\rho^{-t_I}$, and for every $r\in[\rho,1]$ and every $Q\in\mathcal D_r(\mathcal P_I)$,
    \begin{equation}
    \label{eq: local semi well spaced geometry}
    |\mathcal P_I\cap Q|
    \lesssim\rho^{-\epsilon}\Delta^{-3}
    \cdot\max\left\{
    r^{2-s}|\mathcal P_I|,
    \left(\frac{r}{\rho}\right)^s
    \right\}.
    \end{equation}
\end{enumerate}
\end{prop}

\begin{proof}
Since $\mathcal P$ is $\{\Delta_j\}_{j=1}^n$-uniform, for every integer
$0\le k\le L$ and every $Q\in\mathcal D_{\Delta^k}(\mathcal P_I)$ we have
\begin{equation}
\label{eq: local branching count}
|\mathcal P_I\cap Q|=\Delta^{f(c+k)-f(d)}.
\end{equation}
In particular, $|\mathcal P_I|=\Delta^{f(c)-f(d)}=\rho^{-t_I}$. Suppose first that $I$ is of type \eqref{item: epsilon linear invertal}. Then
$$
|f(c+k)-f(c)-t_Ik|\le \epsilon L,
\qquad 0\le k\le L.
$$
Using \eqref{eq: local branching count}, we obtain
$$
|\mathcal P_I\cap Q|\le\Delta^{-\epsilon L}(\Delta^k)^{t_I}|\mathcal P_I|=\rho^{-\epsilon}(\Delta^k)^{t_I}|\mathcal P_I|,
$$
and also the corresponding lower bound $
|\mathcal P_I\cap Q|\ge\rho^{\epsilon}(\Delta^k)^{t_I}|\mathcal P_I|$. Thus the AD-regular estimates hold at the exact scales $r=\Delta^k$. Passing from exact $\Delta$-powers to arbitrary $p$-adic scales $r\in[\rho,1]$ only costs the fixed factor
$\Delta^{-3}$, which proves the first assertion.

Now suppose that $I$ is of type \eqref{item: epsilon superlinear invertal}. By
\Cref{lem: interval partitioning}, with $u=2-s$, we have
$$
f(c+k)\ge
\min\{f(c)+(2-s)k,\ f(d)-s(L-k)\}
-\epsilon L
$$
for every $0\le k\le L$. Substituting this into \eqref{eq: local branching count} gives
$$
|\mathcal P_I\cap Q|
\le
\Delta^{-\epsilon L}
\max\left\{
\Delta^{f(c)+(2-s)k-f(d)},
\Delta^{-s(L-k)}
\right\}.
$$
Since $\Delta^{f(c)-f(d)}=|\mathcal P_I|$ and $\rho=\Delta^L$, this becomes
$$
|\mathcal P_I\cap Q|
\le
\rho^{-\epsilon}
\max\left\{
(\Delta^k)^{2-s}|\mathcal P_I|,
\left(\frac{\Delta^k}{\rho}\right)^s
\right\}
$$
at the exact scales $r=\Delta^k$. Replacing $\Delta^k$ by an arbitrary
$p$-adic scale $r\in[\rho,1]$ again costs only the harmless factor $\Delta^{-3}$, and this gives
\eqref{eq: local semi well spaced geometry}.
\end{proof}

We shall also use a multiscale decomposition for nice configurations. This is the
configuration-level refinement which underlies the scale-partitioning argument of
Ren-Wang \cite[Proposition 6.7]{ren2023furstenberg}. 

\begin{prop}
\label{prop: multiscale configuration decomposition}
Fix $N\ge 1$ and a sequence $\{\Delta_j\}_{j=0}^N\subset p^{-\N}$ such that
$$
0<\delta=\Delta_N<\Delta_{N-1}<\cdots<\Delta_1<\Delta_0=1.
$$
Let $(\mathcal{P}_0,\mathcal{T}_0)\subset\mathcal{D}_\delta\times\mathcal{T}_\delta$ be a $(\delta,s,C,M)$-nice configuration. Then there exists a set $\mathcal{P}\subset\mathcal{P}_0$ such that:
\begin{enumerate}
\item For every $1\le j\le N$,
$|\mathcal{D}_{\Delta_j}(\mathcal{P})|
\approx_\delta|\mathcal{D}_{\Delta_j}(\mathcal{P}_0)|$, and for every $\mathbf{p}\in\mathcal{D}_{\Delta_j}(\mathcal{P})$, $|\mathcal{P}\cap\mathbf{p}|
\approx_\delta|\mathcal{P}_0\cap\mathbf{p}|$.
\item For every $0\le j\le N-1$ and every $\mathbf{p}\in\mathcal{D}_{\Delta_j}(\mathcal{P})$, there exist numbers
$$
C_\mathbf{p}\approx_\delta C,\qquad M_\mathbf{p}\ge 1,
$$
and a family of tubes $\mathcal{T}_\mathbf{p}\subset \mathcal{T}_{\Delta_{j+1}/\Delta_j}$ with the property that $\left(\phi_\mathbf{p}(\mathcal{P}\cap\mathbf{p}),\mathcal{T}_\mathbf{p}\right)$ is a $(\Delta_{j+1}/\Delta_j,s,C_\mathbf{p},M_\mathbf{p})$-nice configuration.
\end{enumerate}
Furthermore, the families $\mathcal{T}_\mathbf{p}$ can be chosen so that if
$$
\mathbf{p}_j\in\mathcal{D}_{\Delta_j}(\mathcal{P}),\qquad\mathbf{p}_{j+1}\subset\mathbf{p}_j,\qquad 0\le j\le N-1,
$$
then
\begin{equation}
\label{eq: multiscale product estimate}
\frac{|\mathcal{T}_0|}{M}
\gtrapprox_\delta\prod_{j=0}^{N-1}
\frac{|\mathcal{T}_{\mathbf{p}_j}|}{M_{\mathbf{p}_j}}.
\end{equation}
\end{prop}

\begin{proof}
We prove the proposition by induction on the number $N$ of scale intervals. The
case $N=1$ is exactly the one-step recursive refinement from \Cref{prop: one step recursive refinement}, applied between the scales $\Delta_0=1$ and $\Delta_1=\delta$. After rescaling the unique $\Delta_0$-cube to
$\Z_p^2$, the resulting local configuration is the original configuration, up to the
harmless losses absorbed in $\approx_\delta$ and $\gtrapprox_\delta$.

Assume now that the proposition has been proved for all scale sequences with
$N$ intervals, and consider a scale sequence with $N+1$ intervals,
$$
0<\delta=\Delta_{N+1}<\Delta_N<\cdots<\Delta_1<\Delta_0=1.
$$
Let $(\P_0,\T_0)$ be the original $(\delta,s,C,M)$-nice configuration.

We first apply \Cref{prop: one step recursive refinement} to
$(\P_0,\T_0)$ between the fine scale $\delta$ and the intermediate scale
$\Delta_N$. This gives a refinement $(\P',\T')$ of $(\P_0,\T_0)$ at resolution
$\Delta_N$, a coarsened nice configuration
$$
(\P^{\Delta_N},\T^{\Delta_N})
:=
(\D_{\Delta_N}(\P'),\D_{\Delta_N}(\T')),
$$
and, for every $q\in\P^{\Delta_N}$, a rescaled local nice configuration
$$
(\P_q,\T_q)
=
(\varphi_q(\P'\cap q),\T_q)
$$
at scale $\delta/\Delta_N$. Here $\varphi_q:q\to\Z_p^2$ denotes the affine
homothety. The refinement can be chosen so that
\begin{equation}\label{eq: 63 first refinement point sizes}
|\D_{\Delta_j}(\P')|\approx_\delta |\D_{\Delta_j}(\P_0)|,
\qquad 0\le j\le N,
\end{equation}
and, for every $q\in\D_{\Delta_N}(\P')$, $
|\P'\cap q|\approx_\delta |\P_0\cap q|$. Moreover, the coarsened configuration
$(\P^{\Delta_N},\T^{\Delta_N})$ is a
$(\Delta_N,s,\lessapprox_\delta C,M_{\Delta_N})$-nice configuration, and each
local configuration $(\P_q,\T_q)$ is a $(\delta/\Delta_N,s,\lessapprox_\delta C,M_q)$-nice configuration. The one-step product estimate gives, for every $q\in\P^{\Delta_N}$,
\begin{equation}\label{eq: 63 one step product}
\frac{|\T_0|}{M}
\gtrapprox_\delta
\frac{|\T^{\Delta_N}|}{M_{\Delta_N}}
\cdot
\frac{|\T_q|}{M_q}.
\end{equation}

We now apply the induction hypothesis to the coarsened configuration
$(\P^{\Delta_N},\T^{\Delta_N})$ with the shorter scale sequence
$$
\Delta_N<\Delta_{N-1}<\cdots<\Delta_1<\Delta_0=1.
$$
This produces a refinement $\P''\subset\P^{\Delta_N}$ at the coarse scale such
that, for every $0\le j\le N$,
\begin{equation}\label{eq: 63 coarse induction sizes}
|\D_{\Delta_j}(\P'')|
\approx_\delta
|\D_{\Delta_j}(\P^{\Delta_N})|.
\end{equation}
Define the final refined point set by $
\P:=\bigcup_{q\in\P''}(\P'\cap q)$. Then $\P\subset\P_0$. Combining \eqref{eq: 63 first refinement point sizes} and
\eqref{eq: 63 coarse induction sizes}, we obtain
$$
|\D_{\Delta_j}(\P)|\approx_\delta |\D_{\Delta_j}(\P_0)|,
\qquad 0\le j\le N+1.
$$
Moreover, for every $p\in\D_{\Delta_j}(\P)$, the fibre $\P\cap p$ is obtained by
first selecting the corresponding coarse parent and then keeping the refined
descendants inside it. Hence
$$
|\P\cap p|\approx_\delta |\P_0\cap p|,
\qquad
p\in\D_{\Delta_j}(\P),\quad 0\le j\le N+1.
$$
This proves the first conclusion of the proposition.

We next verify the conclusion about the local configurations. For the scale
intervals $[\Delta_{j+1},\Delta_j]$ with $0\le j\le N-1$, the local configurations
are exactly those produced by the induction hypothesis applied to the coarsened
configuration $(\P^{\Delta_N},\T^{\Delta_N})$. For the last interval
$[\Delta_{N+1},\Delta_N]$, the local configurations are the configurations
$(\P_q,\T_q)$ constructed in the first one-step refinement, with $q\in\P''$.
Since $\varphi_q$ maps $q$ bijectively onto $\Z_p^2$, a cube of radius
$\Delta_{N+1}$ inside $q$ is mapped to a cube of radius
$\Delta_{N+1}/\Delta_N$, and the same statement holds for tubes under the
point-line parametrisation. Thus the local configuration attached to this last
interval is a
$$
(\Delta_{N+1}/\Delta_N,s,\lessapprox_\delta C,M_q)\text{-nice configuration},
$$
as required.

It remains to prove the product estimate. Let
$$
p_j\in\D_{\Delta_j}(\P),\qquad 0\le j\le N+1,
$$
be a nested chain, so that $p_{j+1}\subset p_j$. Put $q:=p_N$. By the induction
hypothesis applied to the coarsened configuration along the truncated chain
$p_0,\ldots,p_N$, we have
\begin{equation}\label{eq: 63 induction product}
\frac{|\T^{\Delta_N}|}{M_{\Delta_N}}
\gtrapprox_\delta
\prod_{j=0}^{N-1}
\frac{|\T_{p_j}|}{M_{p_j}}.
\end{equation}
On the other hand, applying the one-step estimate \eqref{eq: 63 one step product}
with this particular $q=p_N$ gives
\begin{equation}\label{eq: 63 last factor product}
\frac{|\T_0|}{M}
\gtrapprox_\delta
\frac{|\T^{\Delta_N}|}{M_{\Delta_N}}
\cdot
\frac{|\T_{p_N}|}{M_{p_N}}.
\end{equation}
Combining \eqref{eq: 63 induction product} and
\eqref{eq: 63 last factor product}, we obtain
$\frac{|\T_0|}{M}\gtrapprox_\delta
\prod_{j=0}^{N}\frac{|\T_{p_j}|}{M_{p_j}}$. This is exactly the desired product estimate for the scale sequence with $N+1$ intervals. The induction is complete.
\end{proof}

\begin{proof}[Proof of \Cref{thm: discretized furstenberg}]
The cases $t\le s$ and $t\ge 2-s$ were proved in \cite[Theorem 1.2]{ren2025high} and \cite[Theorem 1.3]{ren2025high}, respectively. Hence we may assume throughout the rest of the proof that $s<t<2-s$.

We now choose the parameters. First choose
$0<\eta_0\ll_{s,t}\epsilon$ sufficiently small. Let $\tau=\tau(s,t,\eta_0)>0$ be the constant supplied by \Cref{lem: interval partitioning} with $u=2-s$. Then choose
$0<\eta\ll \tau\eta_0$ sufficiently small so that the hypotheses of both \Cref{thm: discrete regular case} and the semi-well-spaced estimate of \cite[Theorem 1.4]{ren2025high} are satisfied at every local scale appearing below. Finally choose an integer $T=T(\eta)$ sufficiently large, and write $\Delta:=p^{-T}$. After replacing $\delta$ by a nearby $p$-adic scale and changing the implicit constants harmlessly, we may assume $\delta=\Delta^n$ for an integer $n$.

By the standard uniformization lemma, after passing to a refinement of the configuration and decreasing $\eta$ if necessary, we may assume that $\mathcal{P}$ is $\{\Delta^j\}_{j=1}^n$-uniform. We also use the standard Frostman extraction to reduce to the case
\begin{equation}
\label{eq: P cardinality normalized}
\delta^{-t+\eta}\le|\mathcal{P}|\le
\delta^{-t-\eta}.
\end{equation}
Indeed, if $|\mathcal{P}|$ is larger, we pass to a $(\delta,t,\delta^{-O(\eta)})$-subset of the required size; since this only restricts the point set and the assigned tube families, any lower bound for the refined configuration is also a lower bound for the original one, up to the usual $\approx_\delta$ losses.

Let $f:[0,n]\to[0,2n]$ be the branching function of $\mathcal{P}$. The fact that $\mathcal{P}$ is a $(\delta,t,\delta^{-\eta})$-set, together with \eqref{eq: P cardinality normalized}, implies
$$
f(x)\ge tx-\eta n-O_T(1),\qquad 0\le x\le n,
$$
and $f(n)\le (t+\eta)n+O_T(1)$. Taking $\delta$ sufficiently small and $\eta\ll \eta_0^2$, these estimates imply the hypotheses of \Cref{lem: interval partitioning}. Applying that lemma with $u=2-s$ and $m=n$, we obtain a family of pairwise disjoint intervals
$$
[c_j,d_j]\subset[0,n],\qquad j\in\mathcal{J},
$$
such that each interval is either of type \eqref{item: epsilon linear invertal} or of type \eqref{item: epsilon superlinear invertal}, and
\begin{equation}
\label{eq: scale intervals length}
d_j-c_j\ge \tau n
\end{equation}
for all $j\in\mathcal{J}$. Moreover, the complement of the union of these intervals has length at most $O_{s,t}(\eta_0 n)$. By moving the endpoints by at most $1$, we may assume that $c_j,d_j$ are integers. This only changes the estimates by a factor depending on $\Delta$, hence by a harmless $\delta^{-O(\eta)}$ loss. Let
$$
\mathcal{S}:=\{[c_j,d_j]:j\in\mathcal{J}\}
$$
be the family of structured intervals. We form a scale sequence
$$
0<\delta=\Lambda_N<\Lambda_{N-1}<\cdots<\Lambda_1<\Lambda_0=1
$$
by listing all scales $\Delta^{c_j}$ and $\Delta^{d_j}$, together with $1$ and $\delta$. We call an interval $[\Lambda_{k+1},\Lambda_k]$ structured if it corresponds to one of the intervals $[c_j,d_j]$, and bad otherwise. The complement estimate above gives
\begin{equation}
\label{eq: bad scale product}
\prod_{k\in\mathcal{B}}\frac{\Lambda_k}{\Lambda_{k+1}}\le\delta^{-O_{s,t}(\eta_0)},
\end{equation}
where $\mathcal{B}$ denotes the set of bad scale intervals. Equivalently,
\begin{equation}
\label{eq: structured scale product}
\prod_{j\in\mathcal{J}}\Delta^{-(d_j-c_j)}
\ge\delta^{-1+O_{s,t}(\eta_0)}.
\end{equation}

We apply \Cref{prop: multiscale configuration decomposition} to this scale sequence. Fix a nested chain $\mathbf{p}_k\in\mathcal{D}_{\Lambda_k}(\mathcal{P}),\mathbf{p}_{k+1}\subset\mathbf{p}_k$. For a structured interval $[c_j,d_j]$, let $\mathbf{p}_j$ denote the element of the chain at scale $\Delta^{c_j}$, and put
$\rho_j:=\Delta^{d_j-c_j}$. Thus $\rho_j$ is the local scale after rescaling $\mathbf{p}_j$ to $\Z_p^2$. By \eqref{eq: scale intervals length}, $\rho_j\le \delta^\tau$. Let $t_j=s_f(c_j,d_j)$. By \Cref{lem: interval partitioning}, we have
$t_j\in[s,2-s]$.

We now estimate the local factors
$\frac{|\mathcal{T}_{\mathbf{p}_j}|}{M_{\mathbf{p}_j}}$ for structured intervals. By \Cref{prop: scale decomposition geometry}, the rescaled point set $\phi_{\mathbf{p}_j}(\mathcal{P}\cap\mathbf{p}_j)$ is either a $(\rho_j,t_j,\rho_j^{-\eta_0})$ almost AD-regular set, or satisfies the semi-well-spaced condition \eqref{eq: local semi well spaced geometry}, provided $\eta$ is chosen sufficiently small relative to $\tau\eta_0$.

If the first alternative holds, then \Cref{thm: discrete regular case} gives
\begin{equation}
\label{eq: local gain regular}
\frac{|\mathcal{T}_{\mathbf{p}_j}|}{M_{\mathbf{p}_j}}\gtrsim_\epsilon\rho_j^{-(s+t_j)/2+\epsilon/10}.
\end{equation}
If the second alternative holds, then we apply the semi-well-spaced estimate
\cite[Theorem 1.4]{ren2025high} to the rescaled configuration on this interval.
The parameters match as follows: the local point dimension is given by the branching exponent on the interval, and the line-slice parameter remains $s$. The semi-well-spaced hypothesis is precisely the second alternative, while the loss from the non-concentration constant is absorbed into the factor $\rho_j^{-\epsilon/10}$. Therefore the same local lower bound as in the regular
case holds. Thus, for every structured interval,
\begin{equation}
\label{eq: local gain unified}
\frac{|\mathcal{T}_{\mathbf{p}_j}|}{M_{\mathbf{p}_j}}\gtrsim_\epsilon\left(\frac{\Delta^{c_j}}{\Delta^{d_j}}\right)^{(s+t_j)/2-\epsilon/10}.
\end{equation}

Combining \eqref{eq: multiscale product estimate} with \eqref{eq: local gain unified}, and discarding the bad intervals at the cost of \eqref{eq: bad scale product}, we obtain
\begin{equation}
\label{eq: product lower structured}
\frac{|\mathcal{T}_0|}{M}\gtrapprox_\delta
\prod_{j\in\mathcal{J}}\left(\frac{\Delta^{c_j}}{\Delta^{d_j}}\right)^{(s+t_j)/2-\epsilon/10}.
\end{equation}

It remains to estimate the exponent. Since
$t_j(d_j-c_j)=f(d_j)-f(c_j)$, we have
\begin{equation}
\label{eq: tj product estimate}
\prod_{j\in\mathcal{J}}\left(\frac{\Delta^{c_j}}{\Delta^{d_j}}\right)^{t_j}=\Delta^{-\sum_{j\in\mathcal{J}}(f(d_j)-f(c_j))}.
\end{equation}
Since $f$ is $2$-Lipschitz and the bad intervals have total length $O_{s,t}(\eta_0 n)$, while $f(n)\ge (t-O(\eta))n$, we have
\begin{equation}
\label{eq: branching sum structured}
\sum_{j\in\mathcal{J}}(f(d_j)-f(c_j))
\ge(t-O_{s,t}(\eta_0))n.
\end{equation}
Hence
\begin{equation}
\label{eq: tj product lower}
\prod_{j\in\mathcal{J}}
\left(\frac{\Delta^{c_j}}{\Delta^{d_j}}\right)^{t_j}\ge\delta^{-t+O_{s,t}(\eta_0)}.
\end{equation}
Similarly, \eqref{eq: structured scale product} gives
\begin{equation}
\label{eq: s product lower}
\prod_{j\in\mathcal{J}}
\left(\frac{\Delta^{c_j}}{\Delta^{d_j}}\right)^s\ge\delta^{-s+O_{s,t}(\eta_0)}.
\end{equation}
Substituting \eqref{eq: tj product lower} and \eqref{eq: s product lower} into \eqref{eq: product lower structured}, we get
$$
\frac{|\mathcal{T}_0|}{M}\gtrsim_\epsilon\delta^{-(s+t)/2+O_{s,t}(\eta_0)+\epsilon/10}.
$$
Choosing $\eta_0>0$ sufficiently small in terms of $\epsilon,s,t$ yields $|\mathcal{T}_0|\gtrsim_\epsilon M\cdot \delta^{-(s+t)/2+\epsilon}$. Since $\mathcal{T}_0=\bigcup_{\mathscr{p}\in\mathcal{P}}\mathcal{T}(\mathscr{p})$, this proves \eqref{eq: discretized furstenberg conclusion} in the remaining range $s<t<2-s$. The proof of \Cref{thm: discretized furstenberg} is complete.
\end{proof}

\appendix

\section{Sharpness examples}
\label{app: sharpness examples}

In this appendix we record the standard examples showing that the lower bound in
\Cref{thm: sharp bound of Furstenberg} is sharp. The first two examples are
direct product constructions. The third one is a $p$-adic analogue the grid-type construction from Wolff \cite{wolff1999recent}.

Throughout this appendix, Hausdorff dimension is taken with respect to the
standard ultrametric on $\Q_p$ and its products. We use the standard fact that,
for every $u\in[0,1]$, there exists a compact set $A\subset \Z_p$ with
$\dim_H A=u$. For example, one may construct such a set by prescribing, in a
sequence of increasingly long blocks of $p$-adic digits, approximately a
proportion $u$ of the digits to be free and the remaining digits to be fixed.

\subsection{The product examples}

We first give the examples corresponding to the terms $s+t$ when $0<t\le s\le 1$, and $s+1$ when $s+t\ge 2$.

\begin{example}[The $s+t$ example]
\label{ex: sharp st}
Assume $0<t\le s\le 1$. Let $A,B\subset \Z_p$ be compact sets with
$$
\dim_H A=s,\qquad \dim_H B=t.
$$
Set $E:=A\times B\subset \Q_p^2$. For every $b\in B$, let
$$
\ell_b:=\{(x,b):x\in\Q_p\}.
$$
Then $E\cap \ell_b=A\times\{b\}$, and hence
$\dim_H(E\cap \ell_b)=s$. The family $\{\ell_b:b\in B\}$ has Hausdorff dimension $t$ in the affine line space. Therefore $E$ is an $(s,t)$-Furstenberg set, and
$$
\dim_H E=\dim_H(A\times B)=s+t.
$$
This shows sharpness of the term $s+t$ in the range where it is the minimum.
\end{example}

\begin{example}[The $s+1$ example]
\label{ex: sharp s plus one}
Assume $0<s\le 1$ and $0<t\le 2$. Let $A\subset \Z_p$ be compact with
$\dim_H A=s$, and set $E:=A\times \Z_p$. Then
$\dim_H E=s+1$. Let $\Lambda\subset \Z_p^2$ be any compact set of Hausdorff dimension $t$ in the slope-intercept parameter space, and consider the line family
$$
\mathcal L:=\{\ell_{a,b}: (a,b)\in\Lambda\},
\qquad
\ell_{a,b}:=\{(x,ax+b):x\in\Q_p\}.
$$
For every $(a,b)\in\Lambda$ and every $x\in A\subset\Z_p$, we have
$ax+b\in\Z_p$. Hence
$$
\{(x,ax+b):x\in A\}\subset E\cap \ell_{a,b}.
$$
It follows that $\dim_H(E\cap \ell_{a,b})\ge s$ for every $\ell_{a,b}\in\mathcal L$. Thus $E$ is an $(s,t)$-Furstenberg set for
every $t\le 2$. This shows sharpness of the term $s+1$ in the range where it is the minimum.
\end{example}

\subsection{The middle Wolff-type example}

It remains to explain the construction corresponding to the middle term
$(3s+t)/2$. Thus assume
$$
0<s<t<2-s.
$$
Set $a:=\frac{t-s}{2}, b:=\frac{s+t}{2}$. Then
$$
0<a<1,\qquad s<b<1,\qquad a+s=b,\qquad a+b=t,\qquad s+b=\frac{3s+t}{2}.
$$

The construction is the $p$-adic analogue of Wolff's middle example. The key
point is to build sets with an exact multiplicative-additive absorption property.
We will construct compact sets
$$
X,M,C,Y\subset \Z_p
$$
such that
\begin{equation}
\label{eq: sharp middle dimensions}
\dim_H X=s,\qquad \dim_H M=a,\qquad \dim_H C=b,\qquad \dim_H Y=b,
\end{equation}
and
\begin{equation}
\label{eq: sharp middle absorption}
C+MX\subset Y.
\end{equation}
Assuming this for the moment, define
$$
E:=X\times Y\subset \Z_p^2
$$
and let
$$
\mathcal L:=\{\ell_{m,c}:m\in M,\ c\in C\},\qquad
\ell_{m,c}:=\{(x,mx+c):x\in\Q_p\}.
$$
For every $(m,c)\in M\times C$, the absorption property
\eqref{eq: sharp middle absorption} gives
$$
\{(x,mx+c):x\in X\}\subset E\cap\ell_{m,c}.
$$
Since the map $x\mapsto (x,mx+c)$ is bi-Lipschitz from $\Z_p$ onto its image, it
follows that
$$
\dim_H(E\cap\ell_{m,c})\ge \dim_H X=s.
$$
Moreover,
$$
\dim_H E=\dim_H(X\times Y)=s+b=\frac{3s+t}{2},
$$
and the slope-intercept parameter set $M\times C$ has Hausdorff dimension
$$
\dim_H(M\times C)=a+b=t.
$$
Thus $E$ is an $(s,t)$-Furstenberg set of dimension $(3s+t)/2$.

We now give the limiting construction explicitly. Choose a rapidly increasing
sequence of positive integers $L_1,L_2,\ldots$, and set
$$R_0:=0,\qquad R_j:=L_1+\cdots+L_j.$$
We assume that the sequence grows so fast that
$$R_1+\cdots+R_{j-1}=o(L_j)$$
as $j\to\infty$. We also assume, by changing the integer parts harmlessly, that
the numbers $\lfloor sL_j\rfloor,\lfloor aL_j\rfloor,\lfloor bL_j\rfloor$ may be
used in place of $sL_j,aL_j,bL_j$; this only creates $o(R_j)$ errors.

For each $j$, define finite digit sets
$$
\mathcal X_j:=
\left\{
\sum_{r=L_j-\lfloor sL_j\rfloor}^{L_j-1}\xi_r p^r:
\xi_r\in\{0,1,\ldots,p-1\}
\right\},
$$
$$
\mathcal M_j:=
\left\{
\sum_{r=L_j-\lfloor aL_j\rfloor}^{L_j-1}\mu_r p^r:
\mu_r\in\{0,1,\ldots,p-1\}
\right\},
$$
and
$$
\mathcal C_j:=
\left\{
\sum_{r=L_j-\lfloor bL_j\rfloor}^{L_j-1}\gamma_r p^r:
\gamma_r\in\{0,1,\ldots,p-1\}
\right\}.
$$
Thus
$$
|\mathcal X_j|=p^{sL_j+O(1)},\qquad
|\mathcal M_j|=p^{aL_j+O(1)},\qquad
|\mathcal C_j|=p^{bL_j+O(1)}.
$$

Define nested compact approximations
$$
X_j:=
\left\{
\sum_{i=1}^j p^{R_{i-1}}x_i+p^{R_j}u:
x_i\in\mathcal X_i,\ u\in\Z_p
\right\},
$$
$$
M_j:=
\left\{
\sum_{i=1}^j p^{R_{i-1}}m_i+p^{R_j}u:
m_i\in\mathcal M_i,\ u\in\Z_p
\right\},
$$
and
$$
C_j:=
\left\{
\sum_{i=1}^j p^{R_{i-1}}c_i+p^{R_j}u:
c_i\in\mathcal C_i,\ u\in\Z_p
\right\}.
$$
Finally set
\begin{equation}
\label{eq: sharp middle Yj definition}
Y_j:=C_j+M_jX_j+p^{R_j}\Z_p.
\end{equation}
Then $X_{j+1}\subset X_j$, $M_{j+1}\subset M_j$, $C_{j+1}\subset C_j$, and
$Y_{j+1}\subset Y_j$. Define
$$
X:=\bigcap_{j=1}^\infty X_j,\qquad
M:=\bigcap_{j=1}^\infty M_j,\qquad
C:=\bigcap_{j=1}^\infty C_j,\qquad
Y:=\bigcap_{j=1}^\infty Y_j.
$$

The dimensions of $X,M,C$ are immediate from the construction:
$$
\dim_H X=s,\qquad \dim_H M=a,\qquad \dim_H C=b.
$$
Indeed, at scale $p^{-R_j}$ the sets $X_j,M_j,C_j$ consist respectively of
$$
\prod_{i=1}^j p^{sL_i+O(1)},\qquad
\prod_{i=1}^j p^{aL_i+O(1)},\qquad
\prod_{i=1}^j p^{bL_i+O(1)}
$$
balls of radius $p^{-R_j}$.

We next estimate $Y_j$. Write
$$
x=x_{<j}+p^{R_{j-1}}x_j,\qquad
m=m_{<j}+p^{R_{j-1}}m_j,\qquad
c=c_{<j}+p^{R_{j-1}}c_j,
$$
where $x_j\in\mathcal X_j$, $m_j\in\mathcal M_j$, $c_j\in\mathcal C_j$, and
$x_{<j},m_{<j},c_{<j}$ are the prefixes chosen before level $j$. Modulo
$p^{R_j}\Z_p$ we have
$$
c+mx
\equiv
c_{<j}+m_{<j}x_{<j}
+p^{R_{j-1}}\bigl(c_j+m_jx_{<j}+m_{<j}x_j\bigr)
\pmod {p^{R_j}\Z_p},
$$
because $p^{2R_{j-1}}m_jx_j\in p^{R_j}\Z_p$.
Here we used
$$
m_jx_j\in p^{2L_j-\lfloor aL_j\rfloor-\lfloor sL_j\rfloor}\Z_p
\subset p^{L_j}\Z_p,
$$
since $a+s=b<1$.

Moreover, $p^{R_{j-1}}c_j\in p^{R_j-\lfloor bL_j\rfloor}\Z_p$, and, since $x_{<j},m_{<j}\in\Z_p$,
$$
p^{R_{j-1}}m_jx_{<j}\in p^{R_j-\lfloor aL_j\rfloor}\Z_p
\subset p^{R_j-\lfloor bL_j\rfloor}\Z_p,
$$
$$
p^{R_{j-1}}m_{<j}x_j\in p^{R_j-\lfloor sL_j\rfloor}\Z_p
\subset p^{R_j-\lfloor bL_j\rfloor}\Z_p.
$$
Thus, after the old prefix has been fixed, the new level contributes at most
$p^{bL_j+O(1)}$ possible $p^{R_j}\Z_p$-cosets. The dependence on the old prefixes
costs at most $p^{O(R_{j-1})}$ further possibilities. Therefore
\begin{equation}
\label{eq: sharp middle Yj upper}
|Y_j|_{p^{-R_j}}\le p^{bL_j+O(R_{j-1})}|Y_{j-1}|_{p^{-R_{j-1}}}.
\end{equation}
Iterating \eqref{eq: sharp middle Yj upper} and using the rapid growth assumption
$R_1+\cdots+R_{j-1}=o(L_j)$ gives
$$
|Y_j|_{p^{-R_j}}\le p^{bR_j+o(R_j)}.
$$
Hence $\dim_H Y\le b$. On the other hand, $C_j\subset Y_j$ for every $j$, so
$C\subset Y$ and therefore $\dim_HY\ge\dim_HC=b$. Thus $\dim_HY=b$.

Finally, by the definition of $Y_j$, for every $j$ we have
$$
C_j+M_jX_j\subset Y_j.
$$
Passing to the intersection gives the exact containment
$$
C+MX\subset Y.
$$
This is the desired absorption property.

\bibliographystyle{plain}
\bibliography{references.bib}

\end{document}